\documentclass [a4paper] {article}
\usepackage{verbatim}                       

\usepackage{graphicx}                       
\usepackage{tikz-cd}

\usepackage{listings}
\usepackage[plain]{fancyref}

\usepackage{amsmath, amsthm, stmaryrd}         
\usepackage{thmtools}
\usepackage{amssymb, mathrsfs,bm, bbm}
\usepackage{enumerate}                      

\usepackage{braket,subfig,titlesec,cleveref,lipsum,xcolor-patch, textcomp,float,booktabs,siunitx}
\usepackage[numbers]{natbib}
\usepackage[section]{placeins}
\usepackage{enumitem}
\usepackage[margin=2.2cm]{geometry}
\usepackage[colorinlistoftodos,prependcaption]{todonotes}
\usepackage{tcolorbox}
\usepackage{thm-restate}

\newcommand{\mf}[1]{\mathfrak{#1}}
\newcommand{\mbb}[1]{\mathbb{#1}}

\newcommand{\N}{\mathbb{N}}
\newcommand{\Z}{\mathbb{Z}}

\newcommand{\R}{\mathbb{R}}
\newcommand{\T}{\mathbb{T}}
\newcommand{\C}{\mathbb{C}}
\newcommand{\mcO}{\mathcal{O}}

\newcommand{\K}{\mathbb{K}}

\renewcommand{\Im}{\operatorname{Im}}

\DeclareMathOperator*{\supp}{supp}
\DeclareMathOperator*{\dom}{dom}

\DeclareMathOperator{\ev} {ev}

\DeclareMathOperator{\Sym} {Sym}
\DeclareMathOperator{\Spec} {Spec}

\DeclareMathOperator{\Diff} {Diff}

\DeclareMathOperator{\Map} {Map}
\DeclareMathOperator{\Mult}{Mult}

\newcommand{\HolInd}{\mathrm{HolInd}}

\renewcommand{\set}[1]{ \left\{ \,#1\,\right\}}

\newcommand{\olE}{\overline{E \vphantom{\widehat{a}}}}

\makeatletter
\newcommand{\bcdot}{} 
\DeclareRobustCommand\bcdot{%
	\mathord{\mathpalette\bcdot@{0.5}}%
}
\newcommand{\bcdot@}[2]{%
	\vcenter{\hbox{\scalebox{#2}{$\m@th#1\bullet$}}}%
}
\makeatother

\newcommand{\squotes}[1]{`#1'}
\newcommand{\dquotes}[1]{``#1''}

\newcommand{\mrm}[1]{\mathrm{#1}}

\newcommand{\mc}{\mathcal}
\renewcommand{\mf}{\mathfrak}

\renewcommand{\sl}{\mathfrak{sl}}
\renewcommand{\a}{\mathfrak{a}}
\renewcommand{\b}{\mathfrak{b}}
\newcommand{\g}{\mathfrak{g}}

\renewcommand{\t}{\mathfrak{t}}
\renewcommand{\k}{\mathfrak{k}}
\newcommand{\h}{\mathfrak{h}}

\newcommand{\n}{\mathfrak{n}}

\newcommand{\s}{\mathfrak{s}}
\renewcommand{\L}{\mathcal{L}}

\renewcommand{\H}{\mathcal{H}}
\newcommand{\M}{\mathcal{M}}

\newcommand{\A}{\mathcal{A}}

\newcommand{\B}{\mathcal{B}}
\newcommand{\mcD}{\mathcal{D}}

\newcommand{\F}{\mathcal{F}}

\newcommand{\U}{\mathrm{U}}

\renewcommand{\O}{\mathcal{O}}
\renewcommand{\P}{\mathrm{P}}
\newcommand{\p}{\mathfrak{p}}

\newcommand{\Heis}{\textrm{Heis}}
\newcommand{\res}{\textrm{res}}

\newcommand{\ad}{\textrm{ad}}
\newcommand{\Ad}{\textrm{Ad}}

\renewcommand{\S}{\mathcal{S}}

\newcommand{\restr}[2]{\left. #1\right|_{#2}}

\DeclareMathOperator*{\st}{\; : \;}

\newcommand{\Span}{\textrm{Span}}

\newcommand{\pos}{\textrm{pos}}
\newcommand{\eff}{\textrm{eff}}

\newcommand{\GL}{\textrm{GL}}
\newcommand{\Gau}{\textrm{Gau}}
\newcommand{\gau}{\mathfrak{gau}}

\newcommand{\HSp}{\mathrm{HSp}}
\newcommand{\Sp}{\mathrm{Sp}}
\newcommand{\SU}{\mathrm{SU}}

\newcommand{\SL}{\mathrm{SL}}
\renewcommand{\sp}{\mathfrak{sp}}

\newcommand{\Aut}{\textrm{Aut}}

\newcommand{\id}{\textrm{id}}

\theoremstyle{plain}
\newtheorem{theorem}{Theorem}[subsection]

\theoremstyle{definition}
\newtheorem{definition}[theorem]{Definition}

\theoremstyle{plain}
\newtheorem{lemma}[theorem]{Lemma}
\newtheorem{proposition}[theorem]{Proposition}
\newtheorem{corollary}[theorem]{Corollary}

\theoremstyle{plain}

\theoremstyle{remark}
\newtheorem{remark}[theorem]{Remark}

\theoremstyle{definition}
\newtheorem{example}[theorem]{Example}

\theoremstyle{problem}
\newtheorem{problem}[theorem]{Problem}

\numberwithin{equation}{section}

\newcommand*{\fancyrefthmlabelprefix}{thm}
\fancyrefaddcaptions{english}{%
  \providecommand*{\frefthmname}{theorem}%
  \providecommand*{\Frefthmname}{Theorem}%
}
\frefformat{plain}{\fancyrefthmlabelprefix}{\frefthmname\fancyrefdefaultspacing#1}
\Frefformat{plain}{\fancyrefthmlabelprefix}{\Frefthmname\fancyrefdefaultspacing#1}

\newcommand*{\fancyreflemlabelprefix}{lem}
\fancyrefaddcaptions{english}{%
  \providecommand*{\freflemname}{lemma}%
  \providecommand*{\Freflemname}{Lemma}%
}
\frefformat{plain}{\fancyreflemlabelprefix}{\freflemname\fancyrefdefaultspacing#1}
\Frefformat{plain}{\fancyreflemlabelprefix}{\Freflemname\fancyrefdefaultspacing#1}

\newcommand*{\fancyrefproplabelprefix}{prop}
\fancyrefaddcaptions{english}{%
  \providecommand*{\frefpropname}{proposition}%
  \providecommand*{\Frefpropname}{Proposition}%
}
\frefformat{plain}{\fancyrefproplabelprefix}{\frefpropname\fancyrefdefaultspacing#1}
\Frefformat{plain}{\fancyrefproplabelprefix}{\Frefpropname\fancyrefdefaultspacing#1}

\newcommand*{\fancyrefcorlabelprefix}{cor}
\fancyrefaddcaptions{english}{%
  \providecommand*{\frefcorname}{corollary}%
  \providecommand*{\Frefcorname}{Corollary}%
}
\frefformat{plain}{\fancyrefcorlabelprefix}{\frefcorname\fancyrefdefaultspacing#1}
\Frefformat{plain}{\fancyrefcorlabelprefix}{\Frefcorname\fancyrefdefaultspacing#1}

\newcommand*{\fancyrefdeflabelprefix}{def}
\fancyrefaddcaptions{english}{%
  \providecommand*{\frefdefname}{definition}%
  \providecommand*{\Frefdefname}{Definition}%
}
\frefformat{plain}{\fancyrefdeflabelprefix}{\frefdefname\fancyrefdefaultspacing#1}
\Frefformat{plain}{\fancyrefdeflabelprefix}{\Frefdefname\fancyrefdefaultspacing#1}

\newcommand*{\fancyrefasmlabelprefix}{asm}
\fancyrefaddcaptions{english}{%
	\providecommand*{\frefasmname}{assumption}%
	\providecommand*{\Frefasmname}{Assumption}%
}
\frefformat{plain}{\fancyrefasmlabelprefix}{\frefasmname\fancyrefdefaultspacing#1}
\Frefformat{plain}{\fancyrefasmlabelprefix}{\Frefasmname\fancyrefdefaultspacing#1}

\newcommand*{\fancyrefremlabelprefix}{rem}
\fancyrefaddcaptions{english}{%
	\providecommand*{\frefremname}{remark}%
	\providecommand*{\Frefremname}{Remark}%
}
\frefformat{plain}{\fancyrefremlabelprefix}{\frefremname\fancyrefdefaultspacing#1}
\Frefformat{plain}{\fancyrefremlabelprefix}{\Frefremname\fancyrefdefaultspacing#1}

\newcommand*{\fancyrefexlabelprefix}{ex}
\fancyrefaddcaptions{english}{%
	\providecommand*{\frefexname}{example}%
	\providecommand*{\Frefexname}{Example}%
}
\frefformat{plain}{\fancyrefexlabelprefix}{\frefexname\fancyrefdefaultspacing#1}
\Frefformat{plain}{\fancyrefexlabelprefix}{\Frefexname\fancyrefdefaultspacing#1}

\newcommand*{\fancyrefproblabelprefix}{prob}
\fancyrefaddcaptions{english}{%
	\providecommand*{\frefprobname}{problem}%
	\providecommand*{\Frefprobname}{Problem}%
}
\frefformat{plain}{\fancyrefproblabelprefix}{\frefprobname\fancyrefdefaultspacing#1}
\Frefformat{plain}{\fancyrefproblabelprefix}{\Frefprobname\fancyrefdefaultspacing#1}

\title{Holomorphic Induction Beyond the Norm-Continuous Setting,\\ With Applications to Positive Energy Representations}
\date{\today}
\author{Milan Niestijl}

\begin{document}
	\pagenumbering{roman}
	\maketitle

	\begin{abstract}
		\noindent
		We extend the theory of holomorphic induction of unitary representations of a possibly infinite-dimensional Lie group $G$ beyond the setting where the representation being induced is required to be norm-continuous. We allow the group $G$ to be a connected BCH(Baker--Campbell--Hausdorff) Fr\'echet--Lie group. Given a smooth $\R$-action $\alpha$ on $G$, we proceed to show that the corresponding class of so-called positive energy representations is intimately related with holomorphic induction. Assuming that $G$ is regular, we in particular show that if $\rho$ is a unitary ground state representation of $G \rtimes_\alpha \R$ for which the energy-zero subspace $\H_\rho(0)$ admits a dense set of $G$-analytic vectors, then $\restr{\rho}{G}$ is holomorphically induced from the representation of the connected subgroup $H := (G^\alpha)_0$ of $\alpha$-fixed points on $\H_\rho(0)$. As a consequence, we obtain an isomorphism $\B(\H_\rho)^G \cong \B(\H_\rho(0))^H$ between the corresponding commutants. We also find that any two such ground state representations are necessarily unitary equivalent if their energy-zero subspaces are unitarily equivalent as $H$-representations. These results were previously only available under the assumption of norm-continuity of the $H$-representation on $\H_\rho(0)$.
	\end{abstract}
	
	\tableofcontents

	\section{Introduction}~
	\pagenumbering{arabic}	
	
	\noindent
	This paper is concerned with unitary representations of a possibly infinite-dimensional connected Lie group $G$ that is modeled on a locally convex vector space (cf.\ \citep{milnor_inf_lie, neeb_towards_lie}). Let $\alpha : \R \to \Aut(G)$ be a smooth action of $\R$ on $G$. We consider those $G$-representations that extend to a unitary representation $\rho$ of $G \rtimes_\alpha \R$ which is \textit{smooth}, in the sense that it admits a dense set of smooth vectors, and which is of \textit{positive energy}, meaning that the self-adjoint generator $-i \restr{d\over dt}{t=0}\rho(1_G, t)$ of the unitary $1$-parameter group $t\mapsto \rho(1_G, t)$ has non-negative spectrum. \\
	
	\noindent
	For infinite-dimensional Lie groups, a full classification of all irreducible representations is typically not tractable, and even less so for factor representations. The positive energy condition serves to isolate a class of representations that are more susceptible to systematic study. It is also quite natural from a physical perspective, because the Hamiltonian in quantum physics is nearly always required to be a positive self-adjoint operator. It is then no surprise that positive energy representations of Lie groups are abundant in physics literature \citep{Wightman_pct_spin_statistics, Borchers_spec_locality, Borchers_energy_momentum, Haag_loc_qt_ph, Luscher_Mack_glob_conf_inv, Olshanski_inv_cones_hol_discr_series, Segal_Loop_Groups, Segal_unreps_of_some_inf_dim_gps}.\\

	\noindent
	Holomorphic induction has proven to be a particularly effective tool in the study of positive energy representations. Let us first describe the main idea of holomorphic induction in the case where $G$ is finite-dimensional. Let $H := (G^\alpha)_0$ be the connected subgroup of $\alpha$-fixed points in $G$, with Lie algebra $\h = \bm{\mrm{L}}(H)$. A unitary $G$-representation $\rho$ is typically called holomorphically induced from the unitary $H$-representation $\sigma$ on $V_\sigma$ if the homogeneous Hermitian vector bundle $\mbb{V} := G \times_H V_\sigma$ over $G/H$ can be equipped with a $G$-invariant complex-analytic bundle structure, with respect to which the Hilbert space $\H_\rho$ can be $G$-equivariantly embedded into the space of holomorphic sections $\mcO(G/H, \mbb{V})$ of $\mbb{V}$, in such a way that the corresponding point evaluations $\mc{E}_x : \H_\rho \to \mbb{V}_x$ are continuous and satisfy $\mc{E}_x\mc{E}_x^\ast = \id_{\mbb{V}_x}$ for every $x \in G/H$. In particular, these conditions imply that $\H_\rho$ is unitarily equivalent to the $G$-representation on a reproducing kernel Hilbert space, and that $\H_\rho$ contains $V_\sigma$ as an $H$-subrepresentation.\\
	
	\noindent
	An important special case is obtained when $V_\sigma$ is one-dimensional. If $\rho$ is holomorphically induced from $\sigma$, we may identify $V_\sigma$ with a cyclic ray $[v_0]$ in $\H_\rho$, whose $G$-orbit in the projective space $\P(\H_\rho)$ is a complex submanifold. This means that $\rho$ is a so-called \textit{coherent state representation} \citep[Def.\ XV.2.1]{Neeb_book_hol_conv}. In this case, the $G$-homogeneous line bundle $\mbb{V}$ is the pull-back of the tautological line bundle over $\P(\H_\rho)$ along the map $G/H \to \P(\H_\rho), \; gH \mapsto [\rho(g)v_0]$, and elements in the image of the corresponding map $\mbb{V} \to \H_\rho$ are usually called \textit{coherent states}. This is also the setting of the well-known Borel--Weil Theorem \citep[Thm.\ 4.12.5]{Duistermaat_book}. Such representations have been studied extensively \citep{Perelomov_coh_states, Neeb_book_hol_conv, Lisiecki_coh_states_survey}, and are known to be tightly related to highest-weight representations \citep[Def.\ X.2.9, Ch.\ XV]{Neeb_book_hol_conv}. In particular, every unitary highest weight representation of $G$ is a coherent state representation \citep[Prop.\ XV.2.6]{Neeb_book_hol_conv}. The converse is not true. The Schr\"odinger representation of the Heisenberg group $\Heis(\R^2, \omega)$ provides a counterexample \citep[Ex.\ XV.3.5]{Neeb_book_hol_conv}.\\
	
	\noindent
	Holomorphic induction, defined as above, was studied in \citep{Neeb_hol_reps} in the context where $G$ is a Banach--Lie group and where $\sigma$ is \textit{bounded}, meaning that it is continuous with respect to the norm-topology on $\B(V_\sigma)$. Writing $\g$ for the Lie algebra of $G$ and $\g_\C$ for its complexification, invariant complex structures on $G/H$ correspond to closed Lie subalgebras $\b \subseteq \g_\C$ satisfying $\b + \overline{\b} = \g_\C$, $\b \cap \overline{\b} = \h_\C$ and $\Ad_h(\b) \subseteq \b$ for all $h \in H$ \citep[Thm.\ 15]{Beltita_inv_cpls_str} (cf.\ \citep[p.\ 203]{Kirillov_book_representations} for the case where $G$ is finite dimensional). The corresponding $G$-invariant holomorphic bundle structures on $\mbb{V}$ then turn out to be parametrized by extensions of $d\sigma : \h \to \B(V_\sigma)$ to a Lie algebra homomorphism $\chi : \b \to \B(V_\sigma)$ satisfying $\chi(\Ad_h(\xi)) = \sigma(h)\chi(\xi)\sigma(h)^{-1}$ for all $\xi \in \b$ and $h \in H$ \citep[Thm.\ 2.6]{Neeb_hol_reps}, as is to be expected from the finite-dimensional setting \citep[Thm.\ 3.6]{Tirao_Wolf_hom_hol_vb}. The holomorphic structure is used to relate various important properties of the $G$-representation $\rho$ with those of $\sigma$. For example, \citep[Thm.\  3.12]{Neeb_hol_reps} entails that the commutants $\B(\H_\rho)^{G} \cong  \B(V_\sigma)^{H, \chi}$ are isomorphic as von Neumann algebras if $V_\sigma \subseteq \H_\rho$ is invariant under $\B(\H_\rho)^{G}$, which implies in particular that $\rho$ is irreducible, multiplicity-free or of type $\mrm{I}, \mrm{II}$ or $\mrm{III}$ if and only if this is true for $\sigma$ \citep[Cor.\  3.14]{Neeb_hol_reps}. Moreover, \citep[Cor.\ 3.16]{Neeb_hol_reps} states that there is up to unitary equivalence at most one unitary $G$-representation $\rho$ that is holomorphically induced from a given pair $(\sigma, \chi)$. The relation between holomorphically induced representations and the positive energy condition is then explained by \citep[Thm.\ 4.12, 4.14]{Neeb_hol_reps}, which essentially state that in the above context, and under suitable assumptions, holomorphically induced representations correspond to so-called \textit{semibounded} ones, the semiboundedness condition being a \squotes{stable} and stronger version of the positive energy condition (cf.\ \citep{neeb_semibounded_inv_cones}). These observations suggest that the class of holomorphically induced representations may well admit a fruitful classification theory of its factor representations. This line of reasoning was pursued in \citep[Thm.\ 5.4, 5.10]{Neeb_semibounded_hilbert_loop} and \citep[Thm.\ 6.1, 7.3, 8.1]{Neeb_semibounded_hermitian}, resulting in a classification of the irreducible semibounded unitary representations of certain double extensions of Hilbert Loop groups and of hermitian Lie groups corresponding to infinite-dimensional irreducible symmetric spaces. \\

	\noindent
	In \citep[Appendix C]{Neeb_semibounded_hilbert_loop}, the theory of holomorphic induction was further developed, allowing $G$ to be a connected BCH Fr\'echet--Lie group, under certain additional assumptions. (Recall that $G$ is BCH if it is real-analytic and has an analytic exponential map which is a local diffeomorphism in $0 \in \g$.) Still, $\sigma$ was required to be norm-continuous. Let us mention that a particular and well-known special case of such a situation had already appeared in the study of smooth positive energy representations of loop groups. In fact, these had been completely classified using holomorphic induction \citep{Segal_Loop_Groups} (cf.\ \cite{Neeb_borel_weil_loop_gps}). \\

	\noindent
	Still, the assumption of norm-continuity of $\sigma$ is too restrictive in numerous examples, some of which we encounter in \Fref{sec: examples} below. It is typically only suitable for describing the class of semibounded unitary representations of $G$. In order to obtain a theory that can be used to describe the possibly larger class of all positive energy representations, one must go beyond the norm-continuity of $\sigma$. \\

	\noindent
	The purpose of this paper is to remove this assumption of norm-continuity of the representation $\sigma$ being induced, whilst still allowing $G$ to be a connected BCH Fr\'echet--Lie group. A main difficulty in this direction is that of equipping the homogeneous vector bundle $G \times_H V_\sigma$ with a $G$-invariant complex-analytic bundle structure. The proof of \citep[Thm.\  2.6]{Neeb_hol_reps} breaks down beyond the norm-continuous setting, so a new approach is required. \\

	\noindent
	We provide two possible solutions to this problem. As in \citep[Appendix C]{Neeb_semibounded_hilbert_loop}, we assume that $\g_\C$ admits a triangular decomposition of the form $\g_\C = \n_- \oplus \h_\C \oplus \n_+$, where $\n_\pm$ and $\h_\C$ are closed Lie subalgebras of $\g_\C$ satisfying $\overline{\n_\pm} \subseteq \n_\mp$, and where $\b = \h_\C \oplus \n_-$. In the first, which we call the \textit{general approach}, we avoid specifying a complex-analytic vector bundle altogether. Instead we replace the space of holomorphic sections by a suitable subspace $C^\omega(G, V_\sigma)^{H, \chi}$ of the space of real-analytic $H$-equivariant maps $C^\omega(G, V_\sigma)^{H}$, defined directly in terms of an extension $\chi : \b \to \L(\mcD)$ of $d\sigma$ to $\b$ with some domain $\mcD \subseteq V_\sigma^\omega$ consisting of analytic vectors. This also avoids the need for a $G$-invariant complex structure on the homogeneous space $G/H$. In the second, which we call the \textit{geometric approach}, we define a stronger notion of holomorphic induction. In this case, $\H_\rho^\infty$ actually embeds into a space of holomorphic mappings on a homogeneous vector bundle. It therefore requires complex geometry. A significant drawback of this approach is that it requires a dense set of so-called b-strongly-entire vectors, whose availability is usually not known, unless $G$ happens to be finite-dimensional, in which case it is completely understood by the results of \citep{Goodman_analytic_entire_vectors} and \citep{Penney_holomorphic_extensions}, see also \Fref{thm: hol_extensions_of_reps} below.\\
	
	\noindent
	Let us also mention that this paper does not complete the story of holomorphic induction. The developed theory still excludes Fr\'echet--Lie groups that are not BCH, such as the Virasoro group. Yet, it is known that holomorphic induction can be used to obtain a complete classification of the positive energy representations of the Virasoro group \citep{Neeb_Salmasian_Virasoro_pe}. Nevertheless, the present paper makes substantial progress towards a more complete understanding of holomorphic induction in the infinite-dimensional context. In relation to positive energy representations, progress was made in a different direction in \citep{neeb_russo_ground_state_top}, where the class of ground state representations is studied in the setting of topological groups.\\
	
	\subsubsection*{Structure of the paper}
	
	\begin{itemize}
		\item In \Fref{sec: preliminaries_holind}, we first recall some preliminaries regarding analytic functions on locally convex spaces. We proceed to define smooth, analytic and strongly-entire representations, which are increasingly regular. We also recall some important results related to positive energy and ground state representations.
		\item We proceed in \Fref{sec: entire_vectors} to define and study the spaces $\H_\rho^{\mcO}$ and $\H_\rho^{\mcO_b}$ of so-called strongly-entire and b-strongly-entire vectors, respectively. We equip these spaces with a locally convex topology, and extend the results of \citep{Goodman_analytic_entire_vectors} from the setting of finite-dimensional Lie groups to the present one, where $G$ is allowed to be infinite-dimensional. In particular, if $G_\C$ is a complex $1$-connected regular BCH Fr\'echet--Lie group with Lie algebra $\g_\C$, we obtain that both $\H_\rho^{\mcO}$ and $\H_\rho^{\mcO_b}$ carry a representation of $G_\C$ that has holomorphic orbit maps. The space $\H_\rho^{\mcO_b}$ plays an important role in the geometric approach to holomorphic induction.
		\item In \Fref{sec: hol_induced} we present the general approach towards holomorphic induction. After determining a useful equivalent formulation, we characterize the inducibility of pairs $(\sigma, \chi)$ in terms of positive definite functions on $G$, which leads to the uniqueness of the holomorphically induced representation up to unitary equivalence. We then proceed to show that there is an isomorphism of von Neumann algebras $\B(\H_\rho)^G \cong \B(V_\sigma)^{H, \chi}$ between the commutants, provided that $V_\sigma \subseteq \H_\rho$ is invariant under $\B(\H_\rho)^G$, in complete analogy with the previously described norm-continuous setting. We also briefly discuss holomorphic induction in stages.
		\item After equipping the $G$-homogeneous vector bundle $\mbb{V}_{\sigma}:= G \times_H V_\sigma^{\mcO_b}$ with a complex-analytic bundle structure, using a suitable extension $\chi$ of $d\sigma$ with domain $V_\sigma^{\mcO_b}$ and under certain assumptions, we define in \Fref{sec: geom_hol_ind} the geometric notion of holomorphically induced representations. We also compare the geometric notion to the one presented in \Fref{sec: hol_induced}.
		\item In relating holomorphic induction with the positive energy condition, we shall have need for a suitably general notion of Arveson spectral subspaces. We therefore generalize in \Fref{sec: Arveson_spectral_theory} the results of \citep[Sec.\ A.3]{KH_Zellner_inv_smooth_vectors} and \citep[Sec.\ A.2]{Neeb_hol_reps} to the level of generality needed in the next section.
		\item In \Fref{sec: pe_and_hol_ind} we study the relation between holomorphic induction and the positive energy condition, under the additional assumption that $G$ is regular. In particular, we show that if $\rho$ is a unitary ground state representation of $G \rtimes_\alpha \R$ for which the energy-zero subspace $\H_\rho(0)$ admits a dense set of $G$-analytic vectors, then $\restr{\rho}{G}$ is holomorphically induced from the $H$-representation on $\H_\rho(0)$. As a consequence, we obtain an isomorphism $\B(\H_\rho)^G \cong \B(\H_\rho(0))^H$ of von Neumann algebras between the corresponding commutants. We also find that any two such ground state representations are necessarily unitary equivalent if their energy-zero subspaces are unitarily equivalent as $H$-representations.
		\item In \Fref{sec: examples}, we consider numerous interesting examples of unitary representations that are holomorphically induced from representations that are not norm-continuous.
	\end{itemize}
	
	\subsubsection*{Acknowledgments:}
	This research is supported by the NWO grant 639.032.734 \dquotes{Cohomology and representation theory of infinite-dimensional Lie groups}. I would like to thank my PhD supervisor Bas Janssens for his guidance. I am moreover grateful to Karl-Hermann Neeb, who has carefully read an earlier version of this manuscript and given various suggestions for improvement. The conversations with Karl-Hermann Neeb were also enlightening.

	\section{Preliminaries}\label{sec: preliminaries_holind}
	\subsection{Analytic functions on locally convex vector spaces}\label{sec: analytic_fns}
	
	\noindent
	Let us recall some definitions and properties of analytic functions between locally convex vector spaces. The main references are \citep{Bochnak_Siciak_1}, \citep{Bochnak_Siciak_2} and \citep{Glockner_inf_dim_lie_without_completeness}. Throughout the following, we fix locally convex vector spaces $E$ and $F$ over the field $\mbb{K}$ that both are complete and Hausdorff, where $\K$ is either $\R$ or $\C$. Define $$\Delta_k : E \to E^k, \; \Delta_k(h) = (h, \ldots, h).$$
	
	\subsubsection{Homogeneous polynomials}

	\begin{definition}
		Suppose $U \subseteq E$ is open and $f \in C^\infty(U, F)$. For $x \in U$ and $k\in \N$, define $\delta^0_x(f):E \to F$ and $\delta^k_x(f) : E \to F$ by $\delta^0_x(f)(v) := f(x)$ and $\delta^k_x(f)(v) := d^kf(x; v, \ldots, v)$.
	\end{definition}
	
	\begin{definition}
		Let $k \in \N$. A map $f : E \to F$ is called a \textit{homogeneous polynomial of degree $k$} if there exists a $k$-linear symmetric map $\widetilde{f} : E^k \to F$ such that $f = \widetilde{f} \circ \Delta_k$. Let $P^k(E, F)$ denote the space of continuous homogeneous polynomials $E \to F$ of degree $k$. For $k = 0$, we set $P^0(E, F) := F$.
	\end{definition}
	
	\noindent
	Set $E^0 := \mbb{K}$. For $k \in \N_{\geq 0}$, we write $\Mult(E^k, F)$ for the space of continuous $k$-linear maps $E^k \to F$, equipped with the topology of uniform convergence on products of compact sets in $E$. For the case $k=1$, we also write $\B(E, F) := \Mult(E, F)$. Let $\Sym^k(E, F) \subseteq \Mult(E^k, F)$ denote the closed subspace of continuous symmetric $k$-linear maps $E^k \to F$. Let $E \widehat{\otimes} F$ denote the completed projective tensor product of $E$ and $F$ \citep[Def.\ 43.2, 43.5]{Treves_tvs}. Define $E^{\widehat{\otimes}k} := E \widehat{\otimes} \cdots \widehat{\otimes} E \; (k \text{ times})$. The topology on $E^{\widehat{\otimes}k}$ is defined by the seminorms $q_1\otimes \cdots \otimes q_k$, where each $q_i$ is a continuous seminorms on $E$, see also \citep[Def.\ 43.3]{Treves_tvs}. On algebraic tensors $t \in E^{\otimes k}$, this seminorm is given by
	\begin{equation}\label{eq: seminorms_on_proj_tensors}
		(q_1 \otimes \cdots \otimes q_k)(t) := \inf\set{ \sum_j \prod_{i=1}^kq_i(\xi_i^{(j)}) \st t = \sum_{j}\xi_1^{(j)}\otimes \cdots \otimes \xi_k^{(j)},  \text{ with } \xi_i^{(j)} \in E}.
	\end{equation}
	\noindent
	On simple tensors we have $(q_1 \otimes \cdots \otimes q_k)(\xi_1\otimes \cdots \otimes \xi_k) = \prod_{i=1}^k q_i(\xi_i)$, where $\xi_i \in E$ \citep[Prop.\ 43.1]{Treves_tvs}. 
	
	\begin{proposition}[{\citep[Prop.\ 43.4, Cor.\ 3 on p.\ 465]{Treves_tvs}}]\label{prop: univ_property_proj_tensors}~\\
		There is a canonical linear isomorphism $\Mult(E^k, F) \cong \B(E^{\widehat{\otimes}k}, F)$. It is a homeomorphism if $E$ is Fr\'echet.
	\end{proposition}
	
	\noindent
	Equip $P^k(E, F)$ with the topology of uniform convergence on compact sets. If $p$ is a continuous seminorm on $F$, $B \subseteq E$ is a subset and $f : E \to F$ is a function, we write $p_B(f) := \sup_{x \in B}p(f(x))$.
	
	\begin{proposition}\label{prop: isom_polys_mult_maps}
		Let $k \in \N_{\geq 0}$. Then $P^k(E, F) \cong \Sym^k(E, F)$ as locally convex vector spaces.
	\end{proposition}
	\begin{proof}
		If $\widetilde{f} : E^k \to F$ is a symmetric $k$-linear map and $f = \widetilde{f} \circ \Delta_k$ is the corresponding homogeneous polynomial, then $\widetilde{f}$ can be recovered from $f$ using the formula \citep[Thm.\ A]{Bochnak_Siciak_1}:
		\begin{equation}\label{eq: polarization_formula}
			\widetilde{f}(x_1, \ldots, x_k) = {1\over k!}\sum_{\epsilon_1, \ldots, \epsilon_k = 0}^1 (-1)^{k - (\epsilon_1 + \cdots + \epsilon_k)}f(\epsilon_1x_1 + \cdots + \epsilon_k x_k).
		\end{equation}
		This formula moreover shows that $\widetilde{f}$ is continuous if and only if $f$ is so, and there is a linear isomorphism $\Sym^k(E, F) \to P^k(E, F)$ given by $\widetilde{f} \mapsto \widetilde{f} \circ \Delta_k =: f$. It remains to show that this map is also a homeomorphism. Suppose that $f = \widetilde{f} \circ \Delta_k$ for some $\widetilde{f} \in \Sym^k(E, F)$. If $B \subseteq E$ is a compact subset and $p$ is a continuous seminorm on $F$, then $p_B(f) \leq p_{B^k}(\widetilde{f})$. Hence $\Sym^k(E, F) \to P^k(E, F), \widetilde{f} \mapsto f$ is continuous. For the continuity of the inverse, we use \eqref{eq: polarization_formula}, from which it follows that if $B_i \subseteq E$ are compact subsets for $i \in \N$ and $p$ is a continuous seminorm on $F$, then 
		\begin{equation}\label{eq: estimate_symk_cts_in_Pk}
			\sup_{x_i \in B_i} p(\widetilde{f}(x_1, \ldots, x_k)) \leq  {2^k \over k!} p_B(f),
		\end{equation}
		where 
		$$B = \set{ \epsilon_1 x_1 + \cdots + \epsilon_k x_k \st \epsilon_i  \in \{0, 1\}, \; x_i \in B_i \quad \text{ for } i \in \{1, \ldots, k\} },$$
		which is a compact subset of $E$. Consequently the map $f \mapsto \widetilde{f}$ is continuous $P^k(E, F) \to \Sym^k(E, F)$.
	\end{proof}
	
	\noindent
	Define the locally convex space $P(E, F) := \prod_{k=0}^\infty P^k(E, F)$, equipped with the product topology. If $F=\mbb{K}$, we simply write $P^n(E) := P^n(E, \mbb{K})$.

	\subsubsection{Analytic functions}
	
	\noindent
	Let $U \subseteq E$ be open and let $f : U \to F$ be a function.
	
	\begin{definition}\label{def: analytic}~
		\begin{itemize}
			\item Suppose $\mbb{K} = \C$. The function $f : U \to F$ is called \textit{complex-analytic} or \textit{holomorphic} if it is continuous, and for every $x \in U$ there exists a $0$ neighborhood $V$ in $E$ with $x + V \subseteq U$ and functions $f_k \in P^k(E,F)$ for $k \in \N_{\geq 0}$ such that:
			$$ f(x+h) = \sum_{k=0}^\infty f_k(h), \qquad \forall h \in V.$$
			\item Suppose $\mbb{K} = \R$. The function $f : U \to F$ is called \textit{real-analytic} if it extends to some complex-analytic map $f_\C : U_\C \to F_\C$ for some open neighborhood $U_\C$ of $U$ in $E_\C$.
			\item Suppose $\mbb{K} = \C$ and $U = E$. The function $f : E \to F$ is called \textit{entire} if it is continuous and there exist functions $f_k \in P^k(E,F)$ for $k \in \N_{\geq 0}$ such that $f(x) = \sum_{k=0}^\infty f_k(x)$ for all $x \in E$.
		\end{itemize}
	\end{definition}
	
	\begin{remark}\label{rem: different_notions_of_real_analytic}
		The above definition of a real-analytic map differs from the one used in \citep{Bochnak_Siciak_2}, where a function $f : U \to F$ is called real-analytic if it is continuous and for every $x \in U$ there exists a $0$-neighborhood $V$ in $U$ with $x + V \subseteq U$ and homogeneous polynomials $f_k :E \to F$ such that $f(x+h) = \sum_{k=0}^\infty f_k(h)$ holds for all $h \in V$. The two notions are equivalent if $E$ and $F$ are Fr\'echet spaces \citep[Rem.\ 2.9]{Glockner_inf_dim_lie_without_completeness}, \citep[Thm.\ 7.1]{Bochnak_Siciak_2}. 
	\end{remark}
	
	\begin{proposition}[{\citep[Prop.\ 5.1]{Bochnak_Siciak_2}}]\label{prop: normally_convergent_series}~\\
		Suppose $\mbb{K} = \C$. Let $f_k \in P^k(E, F)$ for every $k \in \N_{\geq 0}$. Let $U \subseteq E$ be a $0$-neighborhood s.t.\ $f(h) := \sum_{k}f_k(h)$ is convergent for every $h \in U$. Assume that $f : U \to F$ is continuous at $0\in U$. Then, for every continuous seminorm $p$ on $F$, there exists a $0$-neighborhood $V \subseteq U$ such that $\sum_{k=0}^\infty p_{V}(f_k) < \infty$.
	\end{proposition}

	\begin{lemma}\label{lem: uniform_absol_conv_on_bdd_for_entire_fns}
		Suppose $\mbb{K} = \C$. Let $f_n \in P^n(E,F)$ for every $n \in \N_{\geq 0}$. Consider the following assertions:
		\begin{enumerate}
			\item $f := \sum_{n=0}^\infty f_n$ defines an entire function $E \to F$.
			\item $\sum_{n=0}^\infty p_B(f_n) < \infty$ for any compact subset $B \subseteq E$ and continuous seminorm $p$ on $F$.
		\end{enumerate}
		We have that $\mrm{(1)} \implies \mrm{(2)}$. If $E$ is a Fr\'echet space, then also $\mrm{(2)} \implies \mrm{(1)}$ holds true.
	\end{lemma}
	\begin{proof}
		Assume that $f = \sum_{n=0}^\infty f_n$ defines an entire function $E \to F$. Let $B \subseteq E$ be a compact subset and let $p$ be a continuous seminorm on $F$. We may assume that $B$ is balanced. As $f$ is continuous, $f(2B) \subseteq F$ is compact and hence bounded. So $M_p := p_{2B}(f) < \infty$. As $f$ is entire, we have $f(zx) = \sum_{n=0}^\infty f_n(x)z^n$ for any $x \in E$ and $z \in \C$. Let $x \in 2B$. Then also $zx \in 2B$ for any $z\in \C$ with $|z| \leq 1$, as $B$ is balanced. Applying \citep[Cor.\ 3.2]{Bochnak_Siciak_2} to the holomorphic map $g : \C \to F, \; g(z) := f(zx)$, we find that $f_n(x) = {1\over 2\pi i} \int_{|z|=1} {g(z) \over z^{n+1}}dz$ and moreover that
		$$p(f_n(x)) \leq  \sup_{|z|=1}p(g(z)) \leq p_{2B}(f) = M_p, \qquad \forall n \in \N_{\geq 0}.$$
		Hence $p_{2B}(f_n)\leq M_p$, so that $p_B(f_n) \leq M_p 2^{-n}$ for all $n \in \N_{\geq 0}$. Thus $\sum_{n=0}^\infty p_B(f_n) \leq M_p \sum_{n=0}^\infty 2^{-n} < \infty$.\\
		
		\noindent
		Suppose that $E$ is a Fr\'echet space. Assume that $(2)$ holds true. Then in particular the series $\sum_{n=0}^\infty f_n(x)$ is convergent for any $x \in E$. So $f := \sum_{n=0}^\infty f_n$ defines a function $E \to F$. To show $f$ is entire, it remains only to show that it is continuous. The condition $(2)$ implies that $s_N \to f$ uniformly on compact subsets, where $s_N := \sum_{n=0}^N f_n$ for any $N \in \N$. As $s_N$ is continuous for every $N \in \N$ and $E$ is Fr\'echet by assumption, this implies that $f$ is continuous (by a standard $3\epsilon$ argument).
	\end{proof}
	
	
	\begin{proposition}[{\citep[Prop.\ 2.4]{Glockner_inf_dim_lie_without_completeness}}]\label{prop: analytic_implies_smooth}~\\
		Every real- or complex-analytic map is smooth. 
	\end{proposition}
	
	\begin{proposition}[{\citep[Prop.\ 5.5]{Bochnak_Siciak_2}}]\label{prop: Taylor_expansion}~\\
		Suppose $\mbb{K} = \C$. If $f : U \to F$ is complex-analytic, then $f(x+h) = \sum_{k=0}^\infty {1\over k!}\delta_x^k(f)(h)$ for all $h \in V$, where $V$ is the maximal balanced $0$-neighborhood of $E$ such that $x + V \subseteq U$. 
	\end{proposition}

	\begin{proposition}[{\citep[Lem.\ 2.5]{Glockner_inf_dim_lie_without_completeness}}]\label{prop: complex_analytic_and_complex_linear_differential}~\\
		Suppose $\mbb{K} = \C$. Then $f$ is complex-analytic if and only if $f$ is smooth and the map $\delta^1_x = df(x; \--) : E \to F$ is complex-linear for every $x \in U$.
	\end{proposition}
	
	\begin{proposition}[{\citep[Lem.\ 2.6]{Glockner_inf_dim_lie_without_completeness}}]~\\
		Suppose $\mbb{K} = \C$. If $f : U \to F$ is complex-analytic, then so is $df : U \times E \to F$.
	\end{proposition}
	
	\noindent
	With these definitions, the chain rule holds for both real- and complex-analytic mappings. One proceeds to define real- and complex- analytic manifolds and Lie groups, see e.g.\ \citep{milnor_inf_lie} and \citep{neeb_towards_lie} for more details.

	\begin{definition}
		If $M$ is a real-analytic manifold and $V$ is a locally convex vector space, we write $C^\omega(M, V)$ for the set of analytic functions $M \to V$. If $M$ is a complex-analytic manifold and $V$ is complex, we write $\mcO(M, V)$ for the space of complex-analytic mappings $M \to V$. 
	\end{definition}
	
	%
	
	\begin{proposition}[Identity Theorems {\citep[Prop.\ 6.6]{Bochnak_Siciak_2}}]\label{prop: identity_theorem}~
		\begin{enumerate}
			\item Suppose that $E$ and $F$ are complex. Let $f : U \to F$ be complex-analytic and assume that $U$ is connected. If $f(x) = 0$ for all $x \in V$ for some open and non-empty $V \subseteq U$, then $f = 0$. 
			\item Suppose that $E$ is real and $F$ is complex. Let $f : U_\C \to F$ be complex-analytic, where $U_\C \subseteq E_\C$ is open and connected. If $U_\C$ contains a non-empty subset $V \subseteq E$ that is open in $E$ and $f(x) = 0$ holds for every $x\in V$, then $f = 0$. 
		\end{enumerate}		
	\end{proposition}
	
	\begin{proposition}\label{prop: jet_vanishes_then_function_trivial}
		Let $x \in U$. The following linear map is continuous:
		\[
			j^\infty_x : C^\infty(U, F) \to P(E,F), \qquad
			f \mapsto \sum_{k=0}^\infty{1\over k!}\delta_x^k(f)
		\]
		If $U$ is connected, then its restriction to $C^\omega(U, F)$ is injective.
	\end{proposition}
	\begin{proof}
		The map $j^\infty_x$ is linear, as each $\delta_x^k : C^\infty(U, F) \to P^k(E,F)$ is so. As $P(E,F) = \prod_{n=0}^\infty P^n(E, F)$ carries the product topology, to see $j^\infty_x$ is continuous it suffices to show that $\delta_x^k$ is continuous for every $k \in \N_{\geq 0}$. This is immediate from the definition of the compact-open $C^\infty$-topology on $C^\infty(U, F)$ \citep[Def.\ I.5.1(d)]{neeb_towards_lie}, and the topology of uniform convergence on compact subsets carried by $P^k(E,F)$. Assume that $U$ is connected. Let $f \in C^\omega(U, F)$ and suppose that $j^\infty_x(f) = 0$. Using \Fref{prop: Taylor_expansion} it follows that $f(x+h) = 0$ for all $h$ in some $0$-neighborhood of $E$. By \Fref{prop: identity_theorem} this implies that $f = 0$.
	\end{proof}
	
	\subsection{Smooth, analytic and strongly-entire representations}\label{sec: regularity_of_reps}
	
	\noindent
	Let $G$ be a BCH(Baker--Campbell--Hausdorff) Fr\'echet--Lie group with Lie algebra $\g$. We write $\g_\C$ for the complexification of $\g$. We refer to \citep{neeb_towards_lie} and \citep{milnor_inf_lie} for an overview on locally convex Lie theory. Let us introduce some notation. If $\mcD$ is a pre-Hilbert space, we write $\L(\mcD)$ for the set of linear operators on $\mcD$. We further define the algebra
	$$\L^\dagger(\mc{D}) := \set{ X \in \L(\mc{D}) \st \exists X^\dagger \in \L(\mc{D}) \st \forall \xi, \eta \in \mc{D} \st \langle X^\dagger \psi, \eta \rangle = \langle \psi, X \eta\rangle }.  $$ 
	The element $X^\dagger$ corresponding to $X \in \L^\dagger(\mc{D})$ is unique, again an element of $\L^\dagger(\mc{D})$ and satisfies $(X^\dagger)^\dagger = X$, so $(\--)^\dagger$ endows $\L^\dagger(\mc{D})$ with an involution. We will also have need for various involutions on the universal enveloping algebra $\mc{U}(\g_\C)$ of $\g_\C$. Let $\theta : \g_\C \to \g_\C$ be defined by $\theta(\xi + i\eta) := \xi - i\eta$ for $\xi, \eta \in \g$.
	
	\begin{definition}\label{def: involutions}
		Extend the conjugation $\theta$ on $\g_\C$ to a complex conjugate-linear automorphism of $\mc{U}(\g_\C)$. Let $\tau$ denote the involutive anti-automorphism of $\mc{U}(\g_\C)$ extending $\xi \mapsto -\xi$ on $\g_\C$. Define $x^\ast := \tau(\theta(x))$ for $x \in \mc{U}(\g_\C)$. Explicitly, $\theta, \tau$ and $(\--)^\ast$ satisfy the following relations, where $\xi_j \in \g_\C$ for $j \in \N$:
		\begin{align*}
			\theta(\xi_1\cdots \xi_n) &= \theta(\xi_1)\cdots \theta(\xi_n), \\
			\tau(\xi_1\cdots \xi_n) &= (-1)^n\xi_n\cdots \xi_1,\\
			(\xi_1\cdots \xi_n)^\ast &= (-1)^n\theta(\xi_n)\cdots \theta(\xi_1).
		\end{align*}
	\end{definition}
	
	\noindent
	If $(\rho, \H_\rho)$ is a unitary $G$-representation, we say that it is continuous if it is so with respect to the strong operator topology on $\U(\H_\rho)$. 
	
	\begin{definition}\label{def: smooth_analytic_vectors}
		Let $(\rho, \H_\rho)$ be a continuous unitary representation of $G$. A vector $\psi \in \H_\rho$ is called \textit{smooth}, resp.\ \textit{analytic}, if the orbit map $G \to \H_\rho, g \mapsto \rho(g)v$ is smooth, resp.\ analytic. We write $\H_\rho^\infty$ and $\H_\rho^\omega$ for the linear subspaces of smooth and analytic vectors, respectively. We say that the representation $\rho$ is \textit{smooth} if $\H_\rho^\infty$ is dense in $\H_\rho$ and \textit{analytic} if $\H_\rho^\omega$ is dense in $\H_\rho$.
	\end{definition}
	
	\begin{remark}\label{rem: involution_rep}
		If $\rho$ is a smooth unitary representation of $G$, then the derived representation $d\rho$ of $\g_\C$ on $\H_\rho^\infty$ extends to a homomorphism $d\rho : \mc{U}(\g_\C) \to \L^\dagger(\H_\rho^\infty)$ satisfying $d\rho(x)^\dagger = d\rho(x^\ast)$ for any $x \in \mc{U}(\g_\C)$.
	\end{remark}

	\begin{definition}\label{def: smooth_analytic_entire_vectors}
		Let $(\rho, \H_\rho)$ be a smooth unitary representation of $G$. 
		\begin{itemize}
			\item Following \citep[Def.\ 3.9]{BasNeeb_ProjReps}, we define two locally convex topologies on $\H_\rho^\infty$:
			\begin{itemize}
				\item The \textit{weak topology} on $\H_\rho^\infty$ is defined by the seminorms 
				$$p_{\xi}(\psi) := \|d\rho(\xi_1\cdots \xi_n)\psi\|, \qquad \text{where } n \in \N_{\geq 0} \text{ and } \xi = (\xi_1, \ldots, \xi_n) \in \g^n.$$
				\item The \textit{strong topology} is defined by the seminorms 
				$$p_B(\psi) := \sup_{\xi \in B} \|d\rho(\xi_1\cdots \xi_n)\psi\|,$$
				where $B \subseteq \g^n$ is bounded and $n \in \N_{\geq 0}$.
			\end{itemize}
			If $G$ is regular, then the space $\H_\rho^\infty$ is complete w.r.t.\ to either of these topologies \citep[Prop.\ 3.19]{BasNeeb_ProjReps}, where we also used that $G$ is a Fr\'echet--Lie group. 
			\item A vector $\psi \in \H_\rho^\infty$ is called \textit{entire} if for every compact $B \subseteq \g_\C$:
			$$\sum_{n=0}^\infty {1\over n!}\sup_{\xi \in B}\|d\rho(\xi^n)\psi\| < \infty.$$
			\item If $\psi \in \H_\rho^\infty$ and $B \subseteq \g_\C$, we define 
			$$p_B^n(\psi) := \sup_{\xi_1, \ldots, \xi_n \in B} \|d\rho(\xi_1\cdots \xi_n)\psi\|$$
			and 
			$$q_B(\psi) := \sum_{n=0}^\infty {1\over n!}p_B^n(\psi).$$
			\item A vector $\psi \in \H_\rho^\infty$ is called \textit{strongly-entire} if $q_B(\psi) < \infty$ for every compact subset $B \subseteq \g_\C$. It is said to be \textit{b-strongly-entire} if $q_B(\psi) < \infty$ for every bounded subset $B \subseteq \g_\C$.
			\item We write $\H_\rho^{\mcO}\subseteq \H_\rho^\infty$ and $\H_\rho^{\mcO_b}$ for the linear subspace of strongly-entire and b-strongly-entire vectors, respectively. Equip $\H_\rho^{\mcO}$ (resp.\ $\H_\rho^{\mcO_b}$) with the locally convex topology defined by the seminorms $q_B$ for compact (resp.\ bounded) subsets $B \subseteq \g_\C$.
			\item We say that the representation $\rho$ is \textit{strongly-entire} if $\H_\rho^{\mcO}$ is dense in $\H_\rho$, and that it is \textit{b-strongly-entire} if $\H_\rho^{\mcO_b}$ is dense in $\H_\rho$.\\
		\end{itemize}
	\end{definition}
	
	\noindent
	If $\psi \in \H_\rho^\infty$, we write $f^\psi : G \to \H_\rho$ for the orbit map $f^\psi(g) = \rho(g)\psi$. As $f^\psi$ is smooth, the homogeneous polynomial $f_n^\psi(\xi) := {1\over n!}d\rho(\xi^n)\psi$ is continuous as a map $\g_\C \to \H_\rho$, so $f_n^\psi \in P^n(\g_\C, \H_\rho)$. Notice further that $j^\infty_0(f^\psi) = \sum_{n=0}^\infty f_n^\psi \in P(\g_\C, \H_\rho)$. Let $\beta_n^\psi$ be the unique element of $\Sym^n(\g_\C, \H_\rho)$ satisfying $f_n^\psi = \beta_n^\psi \circ \Delta_n$. Explicitly, $\beta_n^\psi(\xi_1, \ldots, \xi_n) = {1\over (n!)^2} \sum_{\sigma \in S_n} d\rho(\xi_{\sigma_1}\cdots \xi_{\sigma_n})v$.
	
	\begin{lemma}\label{lem: entire_vectors_complexified_seminorms}
		Let $\psi \in \H_\rho^\infty$. Assume that $q_B(\psi) < \infty$ for every compact subset $B \subseteq \g$. Then $q_B(\psi) < \infty$ for every compact subset $B \subseteq \g_\C$.
	\end{lemma}
	\begin{proof}
		Let $B_\C \subseteq \g_\C$ be compact. Replacing $B_\C$ by its balanced hull, we may assume that $B_\C$ is balanced. Let $B := \set{ \xi + \overline{\xi} \st \xi \in B_\C} \subseteq \g$, which is compact in $\g$. Then $B_\C \subseteq B + iB$ and so $q_{B_\C}(\psi) \leq q_{2B}(\psi) < \infty$.
	\end{proof}
	
	\begin{proposition}\label{prop: analytic_vectors_equivalent_definitions}
		Let $(\rho, \H_\rho)$ be a smooth unitary representation of $G$. Let $\psi \in \H_\rho^\infty$. The following assertions are equivalent:
		\begin{enumerate}
			\item $\psi \in \H_\rho^\omega$.
			\item There exists a $0$-neighborhood $V\subseteq \g$ such that $\sum_{n=0}^\infty {1\over n!}d\rho(\xi^n)\psi$ converges for every $\xi \in V$ and the map $V \to \H_\rho, \; \xi \mapsto \sum_{n=0}^\infty {1\over n!}d\rho(\xi^n)\psi$ is continuous.
			\item $\sum_{n=0}^\infty {1\over n!}d\rho(\xi^n)\psi$ converges for every $\xi$ in a $0$-neighborhood $\g$.
			\item There is a $0$-neighborhood $V \subseteq \g$ such that $\sum_{n=0}{1\over n!}p_V^n(\psi) < \infty$.
			\item There is a $0$-neighborhood $V \subseteq \g$ such that $\sum_{n=0}^\infty {1\over n!}\langle \psi, d\rho(\xi^n)\psi\rangle$ converges for all $\xi \in V$.
			\item The map $G \to \C, \; g \mapsto \langle \psi, \rho(g)\psi\rangle$ is analytic on a neighborhood of $1 \in G$.
		\end{enumerate}
	\end{proposition}
	\begin{proof}
		Assume that $\psi \in \H_\rho^\omega$. Then the orbit map $f^\psi : G \to \H_\rho$ is real-analytic, and hence so is $f^\psi \circ \exp : \g \to \H_\rho$. Notice that $f^\psi(e^\xi) = \rho(e^\xi)\psi$, so that $\delta_0^n(f^\psi \circ \exp) = d\rho(\xi^n)\psi$. Using \Fref{prop: Taylor_expansion}, it follows that $f^\psi(e^\xi) = \sum_{n=0}^\infty {1\over n!}d\rho(\xi^n)\psi$ on some balanced $0$-neighborhood $V \subseteq \g$. So $(1) \implies (2)$.\\
		
		\noindent
		We show that $(2) \implies (1)$. Let $V\subseteq \g$ be a $0$-neighborhood such that $\sum_{n=0}^\infty {1\over n!}d\rho(\xi^n)\psi$ converges for every $\xi \in V$ and s.t.\ the map $\xi \mapsto \sum_{n=0}^\infty {1\over n!}d\rho(\xi^n)\psi$ is continuous on $V$. Replacing $V$ by some smaller balanced open set, we may assume that $V$ is balanced. Define $h^\psi(\xi) := \sum_{n=0}^\infty {1\over n!}d\rho(\xi^n)\psi$. In view of \Fref{rem: different_notions_of_real_analytic}, the assumptions imply that $h^\psi$ is real-analytic on $V$, where it was used that $\g$ is Fr\'echet and $\H_\rho$ is a Hilbert space. Then $h^\psi$ is smooth by \Fref{prop: analytic_implies_smooth}. Let $\xi \in V$. We show that $h^\psi(\xi) = \rho(e^\xi)\psi$. Let $s \in I := [-1, 1]$. Then $s\xi \in V$, because $V$ is balanced. Notice that 
		$$\restr{d\over dt}{t=s}h^\psi(t\xi)\psi = d\rho(\xi)h^\psi(s\xi), \qquad \text{ and }\qquad \restr{d\over dt}{t=s}\rho(e^{t\xi})\psi = d\rho(\xi)\rho(e^{s\xi})\psi.$$
		Let $\eta \in \H_\rho^\infty$. Using $d\rho(\xi)^\ast \eta = -d\rho(\xi)\eta$ it follows that $\restr{d\over dt}{t=s} \langle \rho(e^{t\xi})\eta, h^\psi(t\xi)\rangle = 0$ for all $s \in I$. Hence $\langle \eta, \rho(e^{-t\xi})h^\psi(t\xi)\rangle = \langle \eta, \psi\rangle$ for all $t \in I$. As this is valid for any $\eta$ in the dense set $\H_\rho^\infty$ it follows that $\rho(e^{-t\xi})h^\psi(t\xi)\psi = \psi$ or equivalently that $h^\psi(t\xi)\psi = \rho(e^{t\xi})\psi$ for all $t \in I$. In particular, taking $t=1$ we conclude that $h^\psi(\xi) = \rho(e^{\xi})\psi$ for all $\xi \in V$. As $h^\psi$ is real-analytic on $V$, so is $\xi \mapsto \rho(e^\xi)\psi$. Since $G$ is BCH, this implies that $g \mapsto \rho(g)\psi$ is analytic on a neighborhood of $1 \in G$. In turn, this implies that it is analytic everywhere, where we have used that $G$ is a real-analytic Lie group and that the composition of real-analytic maps is again real-analytic \citep[Proposition 2.8]{Glockner_inf_dim_lie_without_completeness}. Thus $\psi \in \H_\rho^\omega$.\\
		
		\noindent
		The implication $(2) \implies (3)$ is trivial whereas $(3) \implies (4)$ follows from \citep[Prop.\ 5.2]{Bochnak_Siciak_2} because $V$ is absorbing and $\g$ is a Baire space, as it is Fr\'echet. To see that $(4) \implies (2)$, assume that $V\subseteq \g$ is a $0$-neighborhood such that $\sum_{n=0}^\infty{1\over n!}p_V^n(\psi) < \infty$. For $\xi \in V$, we write $s_N(\xi) := \sum_{n=0}^N {1\over n!}d\rho(\xi^n)\psi$ and $s(\xi) := \sum_{n=0}^\infty {1\over n!}d\rho(\xi^n)\psi$. It remains only to prove that $s$ is continuous on $V$. Let $\xi \in V$. Suppose that $(\xi_k)$ is a sequence in $V$ with $\xi_k \to \xi$. Let $\epsilon > 0$. Let $N \in \N$ be s.t.\ $\sum_{n=N+1}^\infty {1\over n!}p_V^n(\psi) < \epsilon$. Then for any $\eta \in V$ we have $\|s(\eta) - s_N(\eta)\| \leq \sum_{n=N+1}^\infty {1\over n!}p_V^n(\psi) < \epsilon$. Using that $s_N$ is continuous, let $N^\prime \in \N$ be s.t.\ $\|s_N(\xi) - s_N(\xi_k)\| < \epsilon$ and $\xi_k \in V$ for all $k \geq N^\prime$. Then 
		$$\|s(\xi) - s(\xi_k)\| \leq \|s(\xi) - s_N(\xi)\| + \|s_N(\xi) - s_N(\xi_k)\| + \|s_N(\xi_k) - s(\xi_k)\| < 3 \epsilon, \quad \forall k \geq N^\prime.$$
		Thus $s(\xi_k) \to s(\xi)$. Hence $s$ is sequentially continuous at $0$. As $\g$ is Fr\'echet, this implies that $s$ is continuous at $\xi$. Thus $(1) \iff (2) \iff (3) \iff (4)$. It  is trivial that $(3) \implies (5)$ whereas $(5) \implies (3)$ follows immediately from \citep[Prop.\ 3.4, 6.3]{Neeb_analytic_vectors} (by considering $\mc{D} := \H_\rho^\infty$ and $v := \psi$). Finally, $(6) \iff (1)$ is precisely \citep[Thm.\ 5.2]{Neeb_analytic_vectors}. This completes the proof.\qedhere\\
	\end{proof}

	\noindent
	Let us consider an analogous statements for entire vectors:
	
	\begin{proposition}\label{prop: entire_vectors_seminorm_finite}
		Let $\psi \in \H_\rho^\infty$. The following assertions are equivalent:
		\begin{enumerate}
			\item The series $\sum_{n=0}^\infty f_n^\psi(\xi) =  \sum_{n=0}^\infty {1\over n!}d\rho(\xi^n)\psi$ defines an entire function 
			$$\g_\C \to \H_\rho, \qquad \xi \mapsto \sum_{n=0}^\infty f_n^\psi(\xi).$$
			\item $\psi$ is an entire vector for $\rho$, i.e., $\sum_{n=0}^\infty {1\over n!}\sup_{\xi \in B}\|d\rho(\xi^n)\psi\| < \infty$ for every compact $B \subseteq \g_\C$.
			\item The map $\g \to \H_\rho, \; \xi \mapsto \rho(e^\xi)\psi$ extends to an entire function $\g_\C \to \H_\rho$.
			\item $\sum_{n=0}^\infty \sup_{\xi_i \in B} \|\beta_n^\psi(\xi_1, \ldots, \xi_n)\| < \infty$ for every compact $B \subseteq \g$.
		\end{enumerate}
	\end{proposition}
	\begin{proof}
		As $\g_\C$ is Fr\'echet by assumption, we know using \Fref{lem: uniform_absol_conv_on_bdd_for_entire_fns} that the series $\sum_{n=0}^\infty f_n^\psi(\xi) =  \sum_{n=0}^\infty {1\over n!}d\rho(\xi^n)\psi$ defines an entire function on $\g_\C$ if and only if 
		$$\sum_{n=0}^\infty  {1\over n!}\sup_{\xi \in B} \|d\rho(\xi^n)\psi\| < \infty, \qquad \forall B \subseteq \g_\C \text{ compact}.$$
		That is, if and only if $(2)$ holds true. Thus $(1) \iff (2)$. Assume next that $(2)$ is valid. As singletons are compact, it follows in particular that $\sum_{n=0} f_n^\psi(\xi)$ converges for every $\xi \in \g_\C$. By \Fref{prop: analytic_vectors_equivalent_definitions}, this implies that $\psi \in \H_\rho^\omega$. Hence the orbit map $f^\psi : G \to \H_\rho$ is real-analytic. As $G$ is BCH, the exponential map $\exp : \g \to G$ is real-analytic and hence $\xi \mapsto f^\psi(e^\xi) = \rho(e^\xi)\psi$ is a real-analytic map $\g \to \H_\rho$. Since $\delta_0^n(f^\psi\circ \exp; \xi) = d\rho(\xi^n)\psi$ for every $n \in \N$, \Fref{prop: Taylor_expansion} implies that $f^\psi(e^\xi) = \sum_{n=0}^{\infty}{1\over n!}d\rho(\xi^n)\psi$ on some $0$-neighborhood in $V$. As $(2)$ and hence $(1)$ hold by assumption, it follows that $\sum_{n=0}^\infty f_n^\psi$ is an entire function extending $\xi \mapsto \rho(e^\xi)\psi$. Thus $(3)$ holds true. Suppose conversely that $(3)$ is valid, so that $f^\psi \circ \exp$ extends to an entire function $F : \g_\C \to \H_\rho$. By \Fref{prop: Taylor_expansion} and using that $\delta_0^n(f^\psi\circ \exp; \xi) = d\rho(\xi^n)\psi$ for $n \in \N$, we find that $F(\xi) = \sum_{n=0}^\infty {1\over n!}d\rho(\xi^n)\psi$ for every $\xi \in \g_\C$. Thus $(1)$ holds true. We have shown $(1) \iff (2) \iff (3)$. Next we show $(2) \implies (4)$. Let $B \subseteq \g_\C$ be compact. As $\g_\C$ is complete, the closed convex hull of $B$ is again compact \citep[p.\ 67]{Treves_tvs}. Thus we may assume that $B$ is convex. Replacing $B$ further by its balanced hull, we may assume that $B$ is balanced. Then $B + \cdots + B \;(n \text{ times}) \subseteq nB$. From \fref{eq: estimate_symk_cts_in_Pk} it follows that 
		$$ 
		\sup_{\xi_i \in B}\|\beta_n^\psi(\xi_1, \ldots, \xi_n)\| \leq {2^n \over n!} \sup_{\xi \in nB}\|f_n^\psi(\xi)\| = {(2n)^n \over n!} \sup_{\xi \in B}\|f_n^\psi(\xi)\|.
		$$
		Choose some $t > 2e$. Since $\sum_{n=0}^\infty \sup_{\xi \in B}\|f_n^\psi(\xi)\| < \infty$ for every compact $B$, it follows (by considering $tB$) that there exists some $C > 0$ s.t.\ $\sup_{\xi \in B} \|f_n^\psi(\xi)\| \leq C t^{-n}$ for every $n \in \N_{\geq 0}$. Then 
		$$\sum_{n=0}^\infty \sup_{\xi_i \in B}\|\beta_n^\psi(\xi_1, \ldots, \xi_n)\| \leq C \sum_{n=0}^\infty {1\over n!}\bigg({2n \over t}\bigg)^n  < \infty.$$
		The implication $(4) \implies (2)$ is trivial.
	\end{proof}
	
	\begin{remark}
		The characterization $(4)$ of entire vectors in \Fref{prop: entire_vectors_seminorm_finite} makes the difference between entire and strongly-entire vectors clear, namely whether one considers the symmetric $n$-linear maps $\beta_n^\psi$ or their non-symmetric analogues $(\xi_1, \ldots, \xi_n)\mapsto {1\over n!}d\rho(\xi_1\cdots \xi_n)\psi$. Analogous to \citep[Rem.\ 3.7]{Neeb_analytic_vectors}, it is in general not known whether or not any entire vector is in fact strongly-entire. In the case where $\g$ is finite-dimensional, this follows immediately from \citep[Thm. I.3, Rem.\ I.7]{Penney_holomorphic_extensions}.
	\end{remark}

	\begin{corollary}\label{cor: entire_vectors_seminorm_finite}
		$\H_\rho^{\mcO} \subseteq \H_\rho^\omega \subseteq \H_\rho^\infty$.
	\end{corollary}
	\begin{proof}
		Any strongly-entire vector is entire. Consequently, the first inclusion follows by combining \Fref{prop: entire_vectors_seminorm_finite} and \Fref{prop: analytic_vectors_equivalent_definitions}. The second one follows from the fact that if the orbit map $f^\psi : G \to \H_\rho$ is real-analytic, then it is smooth by \Fref{prop: analytic_implies_smooth}.
	\end{proof}
	
	\noindent
	The space $\H_\rho^{\mcO}$ of strongly-entire vectors will be considered in more detail in \Fref{sec: entire_vectors} below.
	
	\subsection{Positive energy and ground state representations.}
	Let $G$ be a locally convex Lie group with Lie algebra $\g$. If $\H$ is a Hilbert space and $S\subseteq \H$ is a subset, we write $\llbracket S \rrbracket\subseteq \H$ for the closed linear span of $S$.
	
	\begin{theorem}[Borchers--Arveson {\citep[Thm.\ 3.2.46]{bratelli_robinson_1}, \citep[Lem.\ 4.17]{beltita_neeb_covariant_reps}}]\label{thm: borchers_arveson}~\\
		Let $\M \subseteq \B(\H)$ be a von Neumann algebra on the Hilbert space $\H$. Let $(U_t)_{t \in \R}$ be a strongly continuous unitary one-parameter group satisfying $U_t\M U_t^{-1} \subseteq \M$ for all $t \in \R$. Assume that $U_t = e^{itH}$ with $H \geq 0$. Define $\alpha: \R \to \Aut(\M)$ by $\alpha_t(x) := \Ad_{U_t}(x) := U_t x U_t^{-1}$ for $t \in \R$ and $x \in \M$. Denote by $\M^\alpha(S)\subseteq \M$ the Arveson spectral subspace for $S\subseteq \R$. Then
		\begin{enumerate}
			\item There exists a strongly continuous unitary one-parameter group $V_t = e^{it H_0}$ in $\M$ with $H_0 \geq 0$ and $\Ad_{V_t} = \alpha_t$ for every $t \in \R$. 
			\item $\bigcap_{t > 0}\llbracket \M^\alpha[t, \infty)\H\rrbracket = \{0\}$. 
			\item $V_t$ is uniquely determined by the additional requirement that for any other such $V_t^\prime = e^{itH_0^\prime}$, we have $H_0^\prime \geq H_0$. In this case, the spectral projection $P$ corresponding to $V_t$ is determined uniquely by
			$$ P[t, \infty)\H = \bigcap_{s < t}\llbracket \M^\alpha[s, \infty)\H \rrbracket.$$
		\end{enumerate}
	\end{theorem}
	
	\begin{definition}
		Consider the setting of \Fref{thm: borchers_arveson}. A unitary one-parameter group $V_t = e^{it H_0}$ satisfying the conditions of \Fref{thm: borchers_arveson}$(1)$ is called a \textit{positive inner implementation} of $\alpha : \R \to \Aut(\M)$ on $\H$. If $V_t$ additionally satisfies the condition in \Fref{thm: borchers_arveson}$(3)$ then it is said to be the \textit{minimal positive inner implementation} of $\alpha$ on $\H$.
	\end{definition}
	
	\begin{definition}~\\
		A smooth unitary representation $(\rho, \H_\rho)$ of $G$ is of \textit{positive energy} (p.e.) at $\xi \in \g$ if $-i\Spec( \overline{d\rho(\xi)}) \geq 0$. If additionally $\H_\rho(0) := \ker \overline{d\rho(\xi)}$ is cyclic for $G$, then $(\rho, \H_\rho)$ is said to be \textit{ground state} at $\xi \in \g$.
	\end{definition}
	
	\begin{definition}\label{def: extended_pe_reps}
		Let $\alpha : \R \to \Aut(G)$ be a homomorphism for which the corresponding action $\R \times G \to G$ is smooth. Define $G^\sharp := G \rtimes_\alpha \R$ and $\g^\sharp := \bm{\mrm{L}}(G^\sharp) = \g \rtimes_D \R \bm{d}$, where $\bm{d} := 1 \in \R$. Let $(\rho, \H_\rho)$ be a smooth unitary representation of $G$. We write $\M := \rho(G)^{\prime \prime}$ for the von Neumann algebra generated by $\rho(G)$.
		\begin{enumerate}
			\item An \textit{extension} of $\rho$ to $G^\sharp$ is a smooth unitary representation $\widetilde{\rho}$ of $G^\sharp$ on $\H_\rho$ such that $\restr{\widetilde{\rho}}{G} = \rho$.
			\item We say that $\rho$ is of positive energy w.r.t.\ $\alpha$ if there exists an extension $\widetilde{\rho}$ of $\rho$ to $G^\sharp$ which is of p.e.\ at $\bm{d} \in \g^\sharp$. In this case $\widetilde{\rho}$ is called a \textit{positive extension} of $\rho$.
			\item Assume that $\rho$ is of positive energy w.r.t.\ $\alpha$. A \textit{minimal positive extension} $\widetilde{\rho}$ of $\rho$ is a positive extension $\widetilde{\rho}$ of $\rho$ to $G^\sharp$ such that $V_t := \widetilde{\rho}(e,t)$ is the minimal positive inner implementation of the automorphism group $\R \to \Aut(\M), \, t\mapsto \Ad_{V_t}$. Then in particular $V_t \in \M$ for every $t \in \R$.
			\item A unitary representation $\rho$ of $G$ that is of p.e.\ w.r.t.\ $\alpha$ is said to be \textit{ground state} if it has a minimal positive extension that is ground state at $\bm{d} \in \g^\sharp$.
		\end{enumerate}
	\end{definition}
	
	\begin{definition}
		Let $\alpha : \R \to \Aut(G)$ be an $\R$-action on $G$ for which the corresponding map $\R \times G \to G$ is smooth. Let $\widehat{G}$ denote the set of equivalence classes of irreducible unitary representations of $G$ that are smooth. Define $$\widehat{G}_{\pos(\alpha)} := \set{ \rho \in \widehat{G} \st \rho \text{ is of p.e.\ w.r.t.\ } \alpha}.$$
	\end{definition}
	
	\begin{proposition}\label{prop: positive_extensions_of_repr}~\\
		Consider the setting of \Fref{def: extended_pe_reps}. Assume that $\rho$ is of positive energy w.r.t.\ $\alpha$. 
		\begin{enumerate}
			\item There exists a unique minimal positive extension $\widetilde{\rho}_0$ of $\rho$ to $G^\sharp$.
			\item If $\widetilde{\rho}$ is any other positive extension of $\rho$ to $G^\sharp$, there exists a strongly continuous unitary $1$-parameter group $(U_t)$ in $\M^\prime$ such that $\widetilde{\rho}(t) = \widetilde{\rho}_0(t)U_t$. In this case $\widetilde{\rho}_0(G^\sharp)^{\prime \prime} = \rho(G)^{\prime \prime}$. In particular, $\rho$ is irreducible if and only if $\widetilde{\rho}_0$ is.
			\item Assume that $\alpha_T = \id_G$ for some $T > 0$. Then $\widetilde{\rho}_0(T) = \id_{\H_\rho}$ and $\rho$ is ground state w.r.t $\alpha$.
			\item Let $P$ denote the spectral measure associated to $t \mapsto \widetilde{\rho}_0(t)$. Let $\epsilon > 0$. Then the projection $P[0, \epsilon)$ has central support $1_{\M} = \id_{\H_\rho} \in \mc{Z}(\M)$. In particular $P[0, \epsilon)\H_\rho$ is cyclic for $\M$. 
		\end{enumerate}
	\end{proposition}
	\begin{proof}
		The first three assertions follow by \citep[Cor.\ 3.9]{BasNeeb_PE_reps_I} and the last by \citep[Lem.\ 4.17]{beltita_neeb_covariant_reps}.
	\end{proof}

	\section{The space $\H_\rho^{\mcO}$ of strongly-entire vectors}\label{sec: entire_vectors}
	\noindent
	Let $G$ be a BCH Fr\'echet--Lie group with Lie algebra $\g$. Let $(\rho, \H_\rho)$ be a smooth unitary representation of $G$. In this section, we extend some results of \citep{Goodman_analytic_entire_vectors} concerning the space of strongly-entire vectors $\H_\rho^{\mcO}$ from the case where $G$ is finite-dimensional to the present setting.
	
	\subsection{Necessary conditions for the existence of strongly-entire representations}
	
	\noindent
	We first show that when $\dim(\g) < \infty$, the definition for $\H_\rho^{\mcO}$ (\Fref{def: smooth_analytic_entire_vectors}) agrees with the one used in \citep[p.61]{Goodman_analytic_entire_vectors}. The existence of a dense set of strongly-entire vectors is well-understood for continuous unitary representations of finite-dimensional Lie groups, yielding immediate necessary conditions for the existence of strongly-entire representations in the infinite-dimensional setting. This will turn out to be quite restrictive.\\
	
	\noindent
	Assume that $\dim(\g) < \infty$. Let us recall the definition used in \citep[p.61]{Goodman_analytic_entire_vectors}. Let $\{e_\mu\}_{\mu=1}^d$ be a basis of $\g$. For $v \in \H_\rho^\infty$, we define
	$$ E_s(v) := \sum_{n=0}^\infty {s^n \over n!}\sup_{1 \leq \mu_k \leq d} \| d\rho(e_{\mu_1}\cdots e_{\mu_n})v\| \in [0, \infty].$$
	Set $\H_\rho^{\omega_t} := \set{v \in \H_\rho^\infty \st E_s(v) < \infty \text{ for all } 0 < s < t}$ for $t > 0$. We now define $\H_\rho^{{\mcO}^\prime} := \bigcap_{t > 0} \H_\rho^{\omega_t}$. Equip $\H_\rho^{{\mcO}^{\prime}}$ with the locally convex topology defined by the seminorms $E_s$ for $s > 0$.
	
	\begin{lemma}
		$\H_\rho^{\mcO_b} = \H_\rho^{\mcO} = \H_\rho^{{\mcO}^\prime}$ as an equality of locally convex vector spaces.
	\end{lemma}
	\begin{proof}
		Define for $s>0$ the compact subsets 
		$$B_s := \set{\sum_{\mu=1}^d c_\mu e_\mu \st c_\mu \in \C, \; |c_\mu| \leq s \quad \forall \mu\in\{1, \ldots, d\}} \subseteq \g_\C.$$
		Let $s > 0$. As $se_\mu \in B_s$ for any $\mu \in \{1, \ldots, d\}$, it is immediate that $E_s(v) \leq q_{B_s}(v)$. Conversely, take $\xi_j \in B_s$ for $j \in \{1, \ldots, n\}$. Then $\xi_j = \sum_{\mu_j = 1}^d c_{\mu_j} e_{\mu_j} \in B_s$ for some $c_{\mu_j} \in \C$ with $|c_{\mu_j}| \leq s$. So
		\begin{align*}
			d\rho(\xi_{j_1}\cdots \xi_{j_n})v
			&= \sum_{\mu_{1}, \ldots \mu_{n}=1}^d c_{\mu_{1}}\cdots c_{\mu_{n}}d\rho(e_{\mu_1}\cdots e_{\mu_n})v.
		\end{align*}
		Consequently 
		\begin{align*}
			\|d\rho(\xi_{j_1}\cdots \xi_{j_n})v\| \leq s^n \sum_{\mu_{1}, \ldots \mu_{n}=1}^d \|d\rho(e_{\mu_1}\cdots e_{\mu_n})v\| \leq s^n d^n \sup_{1\leq \mu_k \leq d} \|d\rho(e_{\mu_1}\cdots e_{\mu_n})v\|. 
		\end{align*}
		Hence $E_s(v) \leq q_{B_s}(v) \leq E_{sd}(v)$ for any $s > 0$. This shows that $\H_\rho^{\mcO} = \H_\rho^{{\mcO}^\prime}$ as locally convex vector spaces. Since $\dim(\g_\C) < \infty$, it is moreover clear that $\H_\rho^{\mcO_b} = \H_\rho^{\mcO}$, because the closure of any bounded set in $\g_\C$ is compact.
	\end{proof}
	
	\noindent
	Following \citep[p.\ 128]{Auslander_Moore_unireps_solvable}, \citep[p.\ 115]{Jenkins_growth_rate} and \citep{Penney_holomorphic_extensions}, we define:
	\begin{definition}~
		\begin{itemize}
			\item A finite-dimensional Lie group $G$ is said to be of \textit{type $R$} if $\Spec(\Ad_g) \subseteq \T$ for every $g\in G$, where $\T\subseteq \C$ is the unit-circle.
			\item A finite-dimensional Lie algebra $\g$ is said to be of \textit{type $R$} if $\Spec(\ad_\xi) \subseteq i\R$ for every $\xi \in \g$.
		\end{itemize}
	\end{definition}
	
	\begin{remark}
		Lie algebras of type $R$ are by some authors also called \textit{weakly elliptic} \citep[Def.\ II.1]{Neeb_open_problems}.
	\end{remark}
	
	\begin{proposition}[{\citep[Prop.\ 1.3]{Jenkins_growth_rate}}]\label{prop: type_R_gp_lie_alg}~\\
		Let $G$ be a finite-dimensional connected Lie group with Lie algebra $\g$. Then $G$ is of type $R$ if and only if $\g$ is of type $R$.
	\end{proposition}
	
	\begin{proposition}[{\citep[Lem.\ on p.\ 120]{Penney_holomorphic_extensions}}]\label{prop: classification_type_R_LA}~\\
		A finite-dimensional Lie algebra $\g$ is of type $R$ if and only if it is the semi-direct product $\s \rtimes \k$ of a compact semisimple Lie algebra $\k$ and a solvable Lie algebra $\s$ that is of type $R$.
	\end{proposition}
	
	\begin{theorem}[{\citep[Cor.\ II.5]{Penney_holomorphic_extensions}}]\label{thm: hol_extensions_of_reps}~\\
		Let $G$ be a finite-dimensional Lie group and $\rho$ a continuous unitary representation of $G$. Then $\H_\rho^{\mcO}$ is dense if and only if $\rho$ factors through a Lie group of type $R$.
	\end{theorem}
	
	\noindent
	In the setting where $G$ is a possibly infinite-dimensional BCH Fr\'echet--Lie group, this yields:
	
	\begin{corollary}
		Let $G$ be a possibly infinite-dimensional BCH Fr\'echet--Lie group. Suppose that $(\rho, \H_\rho)$ is a strongly-entire unitary representation of $G$. If $\rho$ is injective, then any finite-dimensional Lie subgroup of $G$ is of type $R$.
	\end{corollary}
	\begin{proof}
		Let $H$ be a finite-dimensional Lie subgroup of $G$. Then $\pi := \restr{\rho}{H}$ is a continuous unitary $H$-representation on $\H_\pi := \H_\rho =: \H$. Since $\H_\rho^{\mcO} \subseteq \H_\pi^{\mcO}$, $\H_\pi^{\mcO}$ is dense in $\H$. As $\rho$ is injective, it follows by \Fref{thm: hol_extensions_of_reps} that $H$ is of type $R$.
	\end{proof}
	
	\noindent
	As an illustration: If $\rho$ is injective and $\H_\rho^{\mcO}$ is dense, then $G$ can not contain a single copy of the $ax + b$ group. On the other hand, \Fref{thm: hol_extensions_of_reps} provides ample examples of continuous representations that admit a dense set of strongly-entire vectors. Indeed, simply take any continuous unitary representation of a finite-dimensional Lie group of type $R$. The following examples show that also infinite-dimensional Lie groups may admit a dense set of strongly-entire vectors. 
	
	\begin{example}[Norm-continuous representations]\label{ex: norm-continuous}~\\
		Let $G$ be a BCH Fr\'echet--Lie group and let $\rho : G \to \U(\H_\rho)$ a unitary representation of $G$ which is continuous w.r.t.\ the norm-topology on $\U(\H_\rho)$. Equipped with the norm topology, $\U(\H_\rho)$ is a Banach--Lie group with Lie algebra 
		$$\mf{u}(\H_\rho) := \set{T \in \B(\H_\rho) \st T^\ast = -T},$$
		and the continuous homomorphism $\rho : G \to \U(\H_\rho)$ is automatically analytic by \citep[Thm.\ IV.1.18]{neeb_towards_lie}. This implies that $\H_\rho^\omega = \H_\rho$. Let us show that we even have $\H_\rho^{\mcO_b} = \H_\rho$. As the representation $d\rho : \g \to \mf{u}(\H_\rho)$ is continuous, there exist a continuous seminorm $p$ on $\g$ s.t.\ $\|d\rho(\xi)\| \leq p(\xi)$ for all $\xi \in \g$ \citep[Ch.\ I.7, Prop.\ 7.7]{Treves_tvs}. So $\|d\rho(\xi_1)\cdots d\rho(\xi_n)\psi\| \leq p(\xi_1)\cdots p(\xi_n) \|\psi\|$, where $\xi_j \in \g$ for $j \in \N$ and $\psi \in \H_\rho$. So if $B \subseteq \g$ is bounded, then with $M := \sup p(B)<\infty$ we get that
		$$ q_B(\psi) := \sum_{n=0}^\infty {1\over n!}\sup_{\xi_i \in B}\|d\rho(\xi_1)\cdots d\rho(\xi_n)\psi\| \leq \sum_{n=0}^\infty {M^n\over n!}\|\psi\| = e^{M} \|\psi\| < \infty. $$
		Using \Fref{lem: entire_vectors_complexified_seminorms}, this proves that $\H_\rho^{\mcO_b} = \H_\rho$.
	\end{example}
	
	\begin{example}[Positive energy representations of Heisenberg groups]\label{ex: pe_reprs_of_Heisenberg_groups}~\\
		We recall the construction of positive energy representations of Heisenberg groups, and show that they admit a dense set of b-strongly-entire vectors. Let $V$ be a real Fr\'echet space and $\omega$ a non-degenerate continuous skew bilinear form $V \times V \to \R$. Let $G := \Heis(V, \omega)$ be the corresponding Heisenberg group, so its underlying set is $\T \times V$ and it has multiplication $(z_1, v_1)\cdot (z_2, v_2) := (z_1z_2e^{-i\omega(v_1, v_2)}, v_1 + v_2)$. As $V$ is a Fr\`echet space, it is Mackey complete by \citep[Thm.\ I.4.11]{michor_convenient}. Using \citep[Thm.\ V.1.8]{neeb_towards_lie}, this implies that $G$ is regular. Let $G_\C := \Heis(V_\C, \omega)$ be the corresponding complexification. Let $\mc{J}$ be a\ compatible positive complex structure on $V$, meaning that $\mc{J}^\ast \omega = \omega$ and $\omega(v, \mc{J}v)> 0$ for any non-zero $v \in V$. The positive-definite sesquilinear form $\langle v, w\rangle := \omega(v, \mc{J}w) + i \omega(v,w)$ makes $V$ into a complex pre-Hilbert space, whose completion we denote by $V_\mc{J}$. Notice that the inclusion $V \to V_{\mc{J}}$ is continuous. Equip the symmetric algebra $\mrm{S}^\bullet(V_{\mc{J}})$ with the inner product satisfying
		\begin{equation}\label{eq: ip_fock_space}
			\langle v_1\cdots v_n, w_1\cdots w_n\rangle = \sum_{\sigma \in S_n} \prod_{j=1}^n \langle v_j, w_{\sigma_j}\rangle, \qquad \text{ for } v_j, w_j \in V_{\mc{J}}.	
		\end{equation}
		Let $\H_\rho$ be the corresponding Hilbert space completion of $\mrm{S}^\bullet(V_{\mc{J}})$. Then $\H_\rho$ contains and is generated by the \dquotes{coherent states} $e^v := \sum_{n=0}^\infty {1\over n!}v^n \in \H_\rho$ for $v \in V_{\mc{J}}$, and there is a unitary representation $\rho$ of $\Heis(V, \omega)$ on $\H_\rho$ satisfying (\citep[Sec.\ 9.5]{Segal_Loop_Groups})
		$$\rho(z,v)e^w = ze^{-{1\over 2}\|v\|^2 - \langle v, w\rangle}e^{v + w}, \qquad \text{ for } v,w \in V \text{ and }z \in \T.$$
		A direct computation verifies the equation $\rho(v_1)\rho(v_2) = e^{-i \omega(v_1,v_2)}\rho(v_1 + v_2)$ for $v_1, v_2 \in V$. Let $\Omega\in \H_\rho$ be the vacuum vector. The map 
		$$G \to \C, \qquad \; (z,v) \mapsto \langle \Omega, \rho(z,v)\Omega \rangle = ze^{-{1\over 2}\|v\|^2}$$
		is smooth, so it follows from \citep[Thm.\ 7.2]{Neeb_diffvectors} that $\H_\rho^\infty$ contains the cyclic vector $\Omega$ and is therefore dense in $\H_\rho$. So $\rho$ is smooth. The infinitesimal $\g$-action $d\rho$ satisfies $d\rho(v)\psi = (\mf{c}(v) - \a(v))\psi$ for any $v,w \in V$ and $\psi \in \mrm{S}^\bullet(V_{\mc{J}})$, where $\mf{c}(v)\psi = v\psi$ is the creation operator with core $\mrm{S}^\bullet(V_{\mc{J}})$ and $\a(v) := \mf{c}(v)^\ast$ is its adjoint, the annihilation operator. From $\mf{c}(\mc{J}v) = i \mf{c}(v)$ and $\a(\mc{J}v) = -i \a(v)$ we obtain that the $\C$-linear extension of $d\rho$ to $\g_\C$ satisfies $d\rho(v + iw) = \mf{c}(v + \mc{J}w) - \a(v - \mc{J}w)$ for $v,w \in V$. \\
		
		\noindent
		To see that $\H_\rho^{\mcO_b}$ is dense in $\H_\rho$, it suffices to show that it contains the cyclic vector $\Omega$, because $\H_\rho^{\mcO_b}$ is $G$-invariant (cf.\ \Fref{lem: strongly-entire-vectors_invariant} below). Let $B$ be the open unit-ball in $V_{\mc{J}}$. Let $K \subseteq V$ be a bounded subset of the real Fr\'echet space $V$. Then $K$ is also bounded as subspace of $V_\mc{J}$, and is thus contained in $sB$ for some $s > 0$. If $v \in B$, then (\citep[p.\ 9]{bratelli_robinson_2})
		$$\|\restr{\mf{c}(v)}{\mrm{S}^n(V_{\mc{J}})} \| = \|\restr{\a(v)}{\mrm{S}^{n+1}(V_{\mc{J}})}\| < \sqrt{n+1}.$$
		So if $(v_j)_{j \in \N}$ is a sequence in $B$, then $\sup_{v_1, \ldots, v_n\in B}\|d\rho(v_n)\cdots d\rho(v_1)\Omega\| < 2^{n}\sqrt{n!}$ for any $n \in \N$. Consequently,
		$$ q_K(\Omega) \leq q_{sB}(\Omega) = \sum_{n=0}^\infty {s^n\over n!} \sup_{v_j\in B}\|d\rho(v_n)\cdots d\rho(v_1)\Omega\| < \sum_{n=0}^\infty {(2s)^n\over \sqrt{n!}} < \infty.$$
		It follows using \Fref{lem: entire_vectors_complexified_seminorms} that $\Omega \in \H_\rho^{\mcO_b}$. Hence $\H_\rho^{\mcO_b}$ is dense in $\H_\rho$ and $\rho$ is b-strongly-entire. 
	\end{example}

	\subsection{Properties of $\H_\rho^{\mcO}$ and holomorphic extensions}

	\noindent
	Let $(\rho, \H_\rho)$ be a smooth unitary representation of $G$. Throughout this section, we assume in addition that $G$ is a \textit{regular} Lie group. Various properties of the locally convex space $\H_\rho^{\mcO}$ are summarized below:

	\begin{restatable}{theorem}{propertiesofentirevectors}\label{thm: properties_of_space_of_entire_vectors}
		The locally convex space $\H_\rho^{\mcO}$ has the following properties:
		\begin{enumerate}
			\item The inclusion $\H_\rho^{\mcO} \hookrightarrow \H_\rho^\infty$ is continuous w.r.t.\ the weak topology on $\H_\rho^\infty$.
			\item $\H_\rho^{\mcO}$ is Hausdorff and complete.
			\item $\H_\rho^{\mcO}$ is both $G$- and $\g$-invariant.
			\item The series $\sum_{m=0}^\infty {1\over m!}d\rho(\eta^m)\psi $ converges in $\H_\rho^{\mcO}$ for every $\psi \in \H_\rho^\mcO$ and $\eta \in \g_\C$. The corresponding map
			\begin{equation}\label{eq: orbit_map_locally}
				\g_\C \times \H_\rho^{\mcO} \to \H_\rho^{\mcO}, \qquad (\eta, \psi) \mapsto \sum_{m=0}^\infty {1\over m!}d\rho(\eta^m)\psi =: \widetilde{\rho}_\C(\eta)\psi
			\end{equation}
			is separately continuous and extends the map $\g \times \H_\rho^{\mcO} \to \H_\rho^{\mcO}, (\eta, \psi)\mapsto \rho(e^\eta)\psi$. In particular, the function $\g_\C \to \H_\rho^\mcO, \; \eta \mapsto \widetilde{\rho}_\C(\eta)\psi$ is entire for every $\psi \in \H_\rho^\mcO$.
			\item For any $\psi \in \H_\rho^\mcO$, the orbit map $G \to \H_\rho^{\mcO}, \; g \mapsto \rho(g)\psi$ is real-analytic.\\
		\end{enumerate}
	\end{restatable}
	
	\noindent
	Before proceeding with the proof of \Fref{thm: properties_of_space_of_entire_vectors}, we first mention some important corollaries and related remarks.
	
	\begin{corollary}\label{cor: local_representation_of_complexified_group}
		Assume that $\H_\rho^\mcO$ is dense in $\H_\rho$. Define the map
		\begin{equation}\label{eq: local_repr_cplx}
			\widetilde{\rho}_\C : \g_\C \to \B(\H_\rho^{\mcO}), \qquad \widetilde{\rho}_\C(\eta)v := \sum_{n=0}^\infty {1\over n!}d\rho(\eta^n)v.
		\end{equation}
		Let $U \subseteq \g_\C$ be open and convex. Assume that $U \cap \g$ is non-empty and open in $\g$. Suppose that the BCH series defines a complex-analytic map $\ast \st U \times U \to \g_\C$. Then $\widetilde{\rho}_\C(\xi \ast \eta) = \widetilde{\rho}_\C(\xi)\widetilde{\rho}_\C(\eta)$ for any $(\xi, \eta) \in U \times U$.
	\end{corollary}
	\begin{proof}
		Define $U_\R := U \cap \g$. Let $v,w \in \H_\rho^\mcO \subseteq \H_\rho^\infty$. Recall from \Fref{rem: involution_rep} that $d\rho(x)^\dagger = d\rho(x^\ast)$ in $\L^\dagger(\H_\rho^\infty)$ for any $x \in \mc{U}(\g_\C)$. Using \Fref{thm: properties_of_space_of_entire_vectors}$(4)$ and the fact that compositions of analytic maps are again analytic \citep[Thm.\ 6.4]{Bochnak_Siciak_2}, it follows that the two maps $U^2 \to \C$ given by
		\begin{align*}
			(\xi, \eta) &\mapsto \langle v, \widetilde{\rho}_\C(\xi)\widetilde{\rho}_\C(\eta)w\rangle = \langle \widetilde{\rho}_\C(\xi^\ast) v, \widetilde{\rho}_\C(\eta)w\rangle,\\
			(\xi, \eta) &\mapsto \langle v, \widetilde{\rho}_\C(\xi\ast \eta)w \rangle
		\end{align*}
		are both complex-analytic. They agree on the real subspace $U_\R^2$, on which they both equal $(\xi, \eta) \mapsto \langle v, \rho(e^\xi)\rho(e^\eta)w\rangle = \langle v,\rho(e^\xi e^\eta)w\rangle$. It follows using \Fref{prop: identity_theorem} that they must be equal everywhere. Since $\H_\rho^\mcO$ is dense in $\H_\rho$, we find with $\xi, \eta \in U$ that $\widetilde{\rho}_\C(\xi)\widetilde{\rho}_\C(\eta)w = \widetilde{\rho}_\C(\xi\ast \eta)w$ for every $w\in \H_\rho^\mcO$, and therefore that $\widetilde{\rho}_\C(\xi)\widetilde{\rho}_\C(\eta) = \widetilde{\rho}_\C(\xi\ast \eta)$.
	\end{proof}
	
	\begin{corollary}\label{cor: hol_repr_on_strentire_vectors}
		Let $(\rho, \H_\rho)$ be a strongly-entire unitary $G$-representation and define $\widetilde{\rho}_\C : \g_\C \to \B(\H_\rho^{\mcO})$ by \fref{eq: local_repr_cplx}. Let $G_\C$ be a regular $1$-connected complex BCH Fr\'echet--Lie group with $\bm{\mrm{L}}(G_\C) = \g_\C$. Then there is a representation 
		$$\rho_\C : G_\C \to \B(\H_\rho^{\mcO})^\times$$
		of $G_\C$ on $\H_\rho^\mcO$ satisfying $\rho_\C(e^\xi) = \widetilde{\rho}_\C(\xi)$ for all $\xi \in \g_\C$, and for which the orbit map $G_\C \to \H_\rho^\mcO, \; g \mapsto \rho_\C(g)\psi$ is complex-analytic for every $\psi \in \H_\rho^\mcO$. 
	\end{corollary}
	\begin{proof}
		As $G_\C$ is a complex BCH Lie group, there are open symmetric convex $0$-neighborhoods $U, U^\prime \subseteq \g_\C$ such that $U \subseteq U^\prime$, $U \cap \g$ is open in $\g$ and the BCH series $\ast$ defines a complex-analytic map $\ast \st U \times U \to U^\prime \subseteq \g_\C$. Shrinking $U$ and $U^\prime$ if necessary, we may further assume that the restriction of $\exp_{G_\C}$ to $U^\prime$ is biholomorphic onto some open $1$-neighborhood $V$ of $G_\C$. Define the function $f : V \to \B(\H_\rho^{\mcO})$ by $f(e^\xi) := \widetilde{\rho}_\C(\xi)$. In view of \Fref{cor: local_representation_of_complexified_group}, $f$ satisfies 
		\begin{equation}\label{eq: local_gp_hom}
			f(e^\xi e^\eta) = f(e^{\xi \ast \eta}) = \widetilde{\rho}_\C(\xi \ast \eta) = \widetilde{\rho}_\C(\xi)\widetilde{\rho}_\C(\eta) = f(e^\xi)f(e^\eta), \qquad \forall \xi, \eta \in U,
		\end{equation}
		where the first equality follows from \citep[Thm.\ IV.2.8]{neeb_towards_lie} and \Fref{prop: jet_vanishes_then_function_trivial}. In particular $f(e^\xi) \in \B(\H_\rho^{\mcO})^\times$ and $f(e^\xi)^{-1} = f(e^{-\xi})$ for any $\xi \in U$. As $G_\C$ is a 1-connected topological group, \eqref{eq: local_gp_hom} further implies that there is a group homomorphism $\rho_\C : G_\C \to \B(\H_\rho^{\mcO})^\times$ extending $f$ (cf.\ \citep[Proposition C.2.1]{Neeb_Glockner_Book}). Let $\psi \in \H_\rho^\mcO$. As $\exp_{G_\C}$ restricts to a biholomorphic map $U^\prime \to V$, it follows using \Fref{thm: properties_of_space_of_entire_vectors}$(4)$ that the map 
		$$ V\to \H_\rho^{\mcO}, \qquad e^\xi \mapsto \rho_\C(e^\xi)\psi = f(e^\xi)\psi = \widetilde{\rho}_\C(\xi)\psi, \qquad \xi \in U' $$
		is complex-analytic. As $G_\C$ is a complex-analytic Lie group and $V \subseteq G_\C$ is an open $1$-neighborhood, this implies that orbit map $g\mapsto\rho_\C(g)\psi$ is complex-analytic $G_\C \to \H_\rho^\mcO$. The map $\g_\C \to \H_\rho^\mcO, \; \xi \mapsto \rho_\C(e^\xi)\psi$ is therefore complex-analytic, and it agrees by construction with the entire map $\xi \mapsto \widetilde{\rho}(\xi)\psi$ on the open $0$-neighborhood $U'$ of $\g_\C$. It follows by \Fref{prop: identity_theorem} that they are equal everywhere, so $\rho_\C(e^\xi)\psi = \widetilde{\rho}(\xi)\psi$ for every $\xi \in \g_\C$. Since this holds for all $\psi \in \H_\rho^\mcO$, we conclude that $\rho_\C(e^\xi) = \widetilde{\rho}(\xi)$ for all $\xi \in \g_\C$.
	\end{proof}
	
	\noindent
	One might wonder whether or not the map in \fref{eq: orbit_map_locally} is also jointly continuous. The following example shows that this is generally false:

	\begin{example}\label{ex: not_jointly_cts}
		Consider the Fr\`echet-Lie group $G = \R^\N$, equipped with the product topology. Let $\g = \R^\N$ be its Lie algebra. Notice that $G$ is also regular and BCH. Consider the unitary representation of $G$ on $\H_\rho := \ell^2(\N, \C)$ defined by $(\rho(g)\psi)(k) := e^{ig(k)}\psi(k)$ for $g \in G$ and $\psi \in \ell^2(\N, \C)$. Letting $\C^{(\N)}$ denote the space of sequences in $\C$ with only finitely many non-zero components, the space $\H_\rho^\infty$ of smooth vectors is $\H_\rho^\infty = \C^{(\N)}$ \citep[Ex.\ 4.8]{Neeb_diffvectors}. Notice that $(d\rho(\xi)\psi)(k) = i \xi(k)\psi(k)$ for $\xi \in \g_\C$ and $\psi \in \C^{(\N)}$. The weak and strong topologies on $\H_\rho^\infty$ agree, and they both coincide with the locally convex inductive limit topology on $\C^{(\N)}$ \citep[Ex.\ 11]{BasNeeb_ProjReps}. This is the strongest locally convex topology on $\C^{(\N)}$ for which the inclusion $\C^N \hookrightarrow \C^{(\N)}$ is continuous for every $N \in \N$. We claim in addition that $\H_\rho^{\mc{O}_b} = \H_\rho^{\mc{O}} = \C^{(\N)}$ as locally convex spaces.\\
		
		\noindent
		Let $\pi_k : \C^\N \to \C, \; \psi \mapsto \psi(k)$ be the projection onto the k-th component for $k \in \N$. Let $\psi \in \C^{(\N)}$. Then there exists $N \in \N$ s.t.\ $\psi(k) = 0$ for all $k > N$. Let $B \subseteq \g_\C$ be a compact subset, and notice that $B$ is contained in the compact set $B^\prime := \prod_{k=1}^\infty B_k$, where $B_k := \pi_k(B)$. Set $M~:=~\sup \{|\xi(k)| \st \xi \in B^\prime, \; 1\leq k\leq N\} < \infty$. Then 
		$$ \| d\rho(\xi_1 \cdots \xi_n) \psi\|^2_{\ell^2} = \sum_{k=1}^N |\xi_1(k) \cdots \xi_n(k) \psi(k)|^2 \leq M^{2n} \|\psi\|^2_{\ell^2}, \qquad \forall \xi_1, \cdots, \xi_n \in B^\prime. $$
		We thus obtain that
		\begin{equation}\label{eq: estimate_fin_sequences}
			q_B(\psi) \leq q_{B^\prime}(\xi) = \sum_{n=0}^\infty {1\over n!} \sup_{\xi_j \in B^\prime} \| d\rho(\xi_1 \cdots \xi_n) \psi\|_{\ell^2} \leq e^M \|\psi\|_{\ell^2} < \infty.
		\end{equation}
		Hence $\psi \in \H_\rho^{\mc{O}}$. So $\C^{(\N)} \subseteq \H_\rho^{\mc{O}}$. We also have $\H_\rho^{\mc{O}} \subseteq \H_\rho^\infty = \C^{(\N)}$, so $\H_\rho^{\mc{O}} = \C^{(\N)}$. Noticing that the constant $M$ in \eqref{eq: estimate_fin_sequences} depends only on $N$ and $B$, the estimate \eqref{eq: estimate_fin_sequences} moreover shows that the inclusion $\C^N \hookrightarrow \H_\rho^{\mc{O}}$ is continuous for every $N \in \N$. The locally convex inductive limit topology on $\C^{(\N)}$ is therefore finer than that of $\H_\rho^{\mc{O}}$. As the inclusion $\H_\rho^{\mc{O}} \hookrightarrow \H_\rho^\infty = \C^{(\N)}$ is moreover continuous w.r.t.\ the weak topology on $\H_\rho^\infty$, by \Fref{thm: properties_of_space_of_entire_vectors}(1), it follows that $\H_\rho^{\mc{O}} = \H_\rho^\infty = \C^{(\N)}$ as locally convex vector spaces. Since $\prod_{k=1}^\infty B_k \subseteq \R^\N$ is bounded whenever every $B_k \subseteq \R$ is so, it is similarly shown that $\H_\rho^{\mcO_b} = \C^{(\N)}$ as locally convex vector spaces.\\
		
		\noindent
		Now, it is shown in \citep[Ex.\ 4.8]{Neeb_diffvectors} that the action 
		$$\g \times \C^{(\N)} \to \ell^2(\N, \C), \qquad (\xi, \psi) \mapsto d\rho(\xi)\psi$$
		is not jointly continuous (w.r.t.\ \textit{any} locally convex topology on $\C^{(\N)}$). This also implies that the map
		$$ \g_\C \times \C^{(\N)} \to \C^{(\N)}, \qquad (\xi, \psi) \mapsto \sum_{k=0}^\infty {1\over k!} d\rho(\xi^k)\psi$$
		can not be jointly continuous.\\
	\end{example}
	
	\begin{remark}\label{rem: b-strongly-entire}
		\noindent
		Replacing `compact' by `bounded' and $\H_\rho^\mcO$ by $\H_\rho^{\mcO_b}$ in the proof of \Fref{thm: properties_of_space_of_entire_vectors} (given shortly), one verifies that all statements in \Fref{thm: properties_of_space_of_entire_vectors} remain true if $\H_\rho^\mcO$ is replaced by the locally convex space $\H_\rho^{\mcO_b}$. The inclusion $\H_\rho^{\mcO_b} \hookrightarrow \H_\rho^\infty$ is moreover evidently continuous w.r.t.\ the strong topology on $\H_\rho^\infty$. \\
	\end{remark}
	\begin{remark}\label{rem: b-strongly-entire_2}
		It may in certain situations be desirable to consider $\H_\rho^{\mcO_b}$ instead of $\H_\rho^\mcO$. For example, if $G$ is a Banach--Lie group, then the locally convex topology on $\H_\rho^{\mcO_b}$ is Fr{\'e}chet. Since the map 
		\begin{equation}\label{eq: extension_b_strongly_entire_inrem}
			\g_\C \times \H_\rho^{\mcO_b} \to \H_\rho^{\mcO_b}, \qquad (\xi, \psi) \mapsto \sum_{m=0}^\infty {1\over m!}d\rho(\eta^m)\psi = \widetilde{\rho}(\xi)\psi
		\end{equation}
		is separately continuous and linear in the $\H_\rho^{\mcO_b}$-variable, this implies using \citep[Prop.\ 5.1]{Neeb_diffvectors} that the function in \eqref{eq: extension_b_strongly_entire_inrem} is \textit{jointly} continuous, and is therefore entire. \\
		
		\noindent
		In this case we consequently obtain a stronger analogue of \Fref{cor: hol_repr_on_strentire_vectors}. Indeed, take $U \subseteq \g_\C$ as in \Fref{cor: local_representation_of_complexified_group}, and suppose that the BCH series defines a complex-analytic map $\ast : U \times U \to \g_\C$. The functions $(\xi, \eta) \mapsto \widetilde{\rho}_\C(\xi)\widetilde{\rho}_\C(\eta)v$ and $(\xi, \eta)\mapsto \widetilde{\rho}_\C(\xi\ast \eta)v$ are then both complex-analytic $U \times U \to \H_\rho^{\mcO_b}$ for every $v \in \H_\rho^{\mcO_b}$. As they agree on $U_\R \times U_\R$, the must be equal on $U \times U$. Consequently, with $G_\C$ as in \Fref{cor: hol_repr_on_strentire_vectors}, we obtain a representation $\rho_\C : G_\C \to \B(\H_\rho^{\mcO_b})^\times$ that satisfies $\rho_\C(e^\xi) = \widetilde{\rho}_\C(\xi)$ for all $\xi$ in some $0$-neighborhood of $\g_\C$, and for which the corresponding action $G_\C \times \H_\rho^{\mcO_b} \to \H_\rho^{\mcO_b}$ is complex-analytic. We will come back to this point in \Fref{sec: geometric_hol_induction} below. Notice also that \Fref{ex: not_jointly_cts} above shows that \eqref{eq: extension_b_strongly_entire_inrem} need not be jointly continuous if $G$ is only assumed to be Fr\'echet.	
	\end{remark}

	\subsubsection*{The proof of \Fref{thm: properties_of_space_of_entire_vectors}}
	
	
	\begin{lemma}\label{lem: cty_of_embedding_strongly_entire_smooth}
		Let $B \subseteq \g_\C$ be compact and let $\psi \in \H_\rho^{\mcO}$. Then ${1\over n!}p_B^n(\psi) \leq q_B(\psi)$ for any $n \in \N$. In particular, the inclusion $\H_\rho^{\mcO} \hookrightarrow \H_\rho^\infty$ is continuous w.r.t.\ the weak topology on $\H_\rho^\infty$.
	\end{lemma}
	\begin{proof}
		Let $\psi \in \H_\rho^{\mcO}$. It is trivial that ${1\over n!}p_B^n(\psi) \leq q_B(\psi)$. For the final statement, consider the continuous seminorm $p_{\bm{\xi}}(\psi) := \|d\rho(\xi_1\cdots \xi_n)\psi\|$ on $\H_\rho^\infty$ for some $\bm{\xi} = (\xi_1, \ldots, \xi_n) \in \g^n$. Taking for $B$ the finite set $B := \{\xi_1, \ldots, \xi_n\} \subseteq \g_\C$, we obtain that ${1\over n!}p_{\bm{\xi}}(\psi) \leq {1\over n!}p_{B}^n(\psi) \leq q_{B}(\psi)$.
	\end{proof}
	
	\begin{lemma}
		$\H_\rho^{\mcO}$ is both Hausdorff and complete.
	\end{lemma}
	\begin{proof}
		It is clear that $\H_\rho^{\mcO}$ is Hausdorff, because $\H_\rho^\infty$ is so. Let us show that it is complete. Let $(\psi_\alpha)_{\alpha \in I}$ be a Cauchy net in $\H_\rho^{\mcO}$. Then it is also a Cauchy net in $\H_\rho^\infty$. The latter is complete \citep[Prop.\ 3.19]{BasNeeb_ProjReps}, where we use that $G$ is a regular Fr\'echet--Lie group. Thus $\psi_\alpha \to \psi$ in $\H_\rho^\infty$ for some $\psi \in \H_\rho^\infty$. We must show that $\psi \in \H_\rho^{\mcO}$ and $\psi_\alpha \to \psi$ in $\H_\rho^{\mcO}$. Fix a compact set $B \subseteq \g$. Let $\epsilon > 0$. Choose $\epsilon_0 > 0$ such that $\epsilon_0(1 + \epsilon_0) < \epsilon$. Let $t > 1$ be such that ${t \over t -1} < 1 + \epsilon_0$. As $(\psi_\alpha)_{\alpha \in I}$ is a Cauchy net in $\H_\rho^{\mcO}$, there exists $\gamma \in I$ such that $q_{tB}(\psi_\alpha - \psi_\beta) < \epsilon_0$ whenever $\alpha, \beta \geq \gamma$. In particular ${1\over k!}p_{B}^k(\psi_\alpha - \psi_\beta) < \epsilon_0 t^{-k}$ for any $\alpha, \beta \geq \gamma$ and $k \in \N_{\geq 0}$. Consequently, for any $\xi_i \in B$ with $i \in \{1, \ldots, k\}$ we have (using that $\psi_\alpha \to \psi$ in $\H_\rho^\infty$):
		$$ {1\over k!}\|d\rho(\xi_1\cdots \xi_k)(\psi - \psi_\beta)\| =  {1\over k!}\lim_{\alpha}\|d\rho(\xi_1\cdots \xi_k)(\psi_\alpha - \psi_\beta)\| \leq \epsilon_0 t^{-k} \qquad \text{ for }\beta \geq \gamma.$$
		Thus ${1\over k!}p_B^k(\psi - \psi_\beta) \leq \epsilon_0 t^{-k}$ for any $\beta \geq \gamma$. Hence
		$$ q_B(\psi - \psi_\beta) = \sum_{k=0}^\infty{1\over k!}p_B^k(\psi - \psi_\beta) \leq \epsilon_0\sum_{k=0}^\infty t^{-k} = {t \over t -1}\epsilon_0 \leq \epsilon_0(1 + \epsilon_0) < \epsilon, \qquad \forall \beta \geq \gamma$$
		This shows that $q_B(\psi) \leq q_B(\psi - \psi_\beta) + q_B(\psi_\beta) < \infty$ and that $q_B(\psi - \psi_\beta) < \epsilon$ for all $\beta \geq \gamma$. As $B$ and $\epsilon$ were arbitrary, we conclude (using the proof of \Fref{lem: entire_vectors_complexified_seminorms}) that $\psi \in \H_\rho^{\mcO}$ and $\psi_\alpha \to \psi$ in $\H_\rho^{\mcO}$.
	\end{proof}
	
	\begin{lemma}\label{lem: invariance_of_entire_vectors_by_la_action}
		Let $B, B_0 \subseteq \g_\C$ be compact subsets and let $t > 1$. Then there exists a compact subset $B^\prime \subseteq \g_\C$ and some $C > 0$, both depending on $B, B_0$ and $t$, such that $B \subseteq B^\prime$ and
		\begin{equation}\label{eq: inv_of_entire_vectors_main_eq}
			{1\over m!}\sum_{n=0}^\infty {1\over n!}\sup_{\eta_j \in B_0}p^n_{B}(d\rho(\eta_1\cdots \eta_m)\psi) \leq C t^{-m}q_{B^\prime}(\psi), \qquad \forall m\in \N_{\geq 0}, \; \forall \psi \in \H_\rho^{\mcO}.
		\end{equation}
		In particular, we have 
		$${1\over m!}q_B(d\rho(\eta^m)\psi) \leq C t^{-m}q_{B^\prime}(\psi)$$
		for any $\psi \in \H_\rho^{\mcO}$, $\eta \in B_0$ and $m \in \N_{\geq 0}$.
	\end{lemma}
	\begin{proof}
		We may assume that $B_0$ and $B$ are both balanced. Define $B^{\prime \prime} := B \cup B_0$, which is again compact and balanced in $\g_\C$. For any $\eta_1, \ldots, \eta_m \in B_0$ and $\psi \in \H_\rho^{\mcO}$ we have
		$$p^n_{B}(d\rho(\eta_1\cdots \eta_m)\psi) \leq p_{B^{\prime \prime}}^{n+m}(\psi) = t^{-(n+m)}p_{tB^{\prime \prime}}^{n+m}(\psi).$$
		Hence $\sup_{\eta_j \in B_0}p^n_{B}(d\rho(\eta_1\cdots \eta_m)\psi) \leq t^{-(n+m)}p_{tB^{\prime \prime}}^{n+m}(\psi)$.
		It follows that
		\begin{align*}
			\sum_{n=0}^\infty {1\over n!}\sup_{\eta_j \in B_0}p^n_{B}(d\rho(\eta_1\cdots \eta_m)\psi) 
			&\leq t^{-m}\sum_{n=0}^\infty {t^{-n}\over n!}p_{tB^{\prime \prime}}^{n+m}(\psi) \\
			&\leq t^{-m}\bigg(\sum_{n=0}^\infty t^{-n}\bigg) \bigg( \sum_{n=0}^\infty {1\over n!}p_{tB^{\prime \prime}}^{n+m}(\psi)\bigg)  \\
			&= {t^{-m} \over 1 - t^{-1}} \sum_{n=0}^\infty {1\over n!}p_{tB^{\prime \prime}}^{n+m}(\psi)
		\end{align*}
		Let $s > 2$. Notice that $\sum_{n=0}^\infty{(n+m)!\over n!} s^{-(n+m)} < \infty$, and that
		\begin{align*}
			\sum_{n=0}^\infty {1\over n!}p_{tB^{\prime \prime}}^{n+m}(\psi) 
			&= \sum_{n=0}^\infty \bigg({(n+m)!\over n!} s^{-(n+m)} \cdot {1\over (n+m)!}p_{stB^{\prime \prime}}^{n+m}(\psi)\bigg)\\
			&\leq \bigg(\sum_{n=0}^\infty{(n+m)!\over n!} s^{-(n+m)}\bigg) \cdot q_{stB^{\prime \prime}}(\psi).
		\end{align*}
		Consequently, with $C_m := \sum_{n=0}^\infty{(n+m)!\over m!n!} s^{-(n+m)} > 0$ and $B^\prime := stB^{\prime\prime}$, we have:
		\begin{equation}\label{eq: inv_of_entire_vectors_1}
			{1\over m!}\sum_{n=0}^\infty {1\over n!}\sup_{\eta_j \in B_0}p^n_{B}(d\rho(\eta_1\cdots \eta_m)\psi) \leq C_m {t^{-m}\over 1 - t^{-1}} q_{B^{\prime}}(\psi).
		\end{equation}
		Using $\sum_{k=0}^N {N \choose k} = 2^N$, notice that $\sum_{m=0}^\infty C_m = \sum_{N=0}^\infty \left(2 \over s\right)^{N} < \infty$, and hence the sequence $(C_m)_{m \in \N_{\geq 0}}$ is bounded. So there exists $C> 0$ with $C_m \leq (1 - t^{-1})C$ for all $m \in \N_{\geq 0}$. Now simply observe using \eqref{eq: inv_of_entire_vectors_1} that \eqref{eq: inv_of_entire_vectors_main_eq} holds for this $C$ and $B^\prime$. Notice also that $B \subseteq B^\prime$.
	\end{proof}

	\begin{lemma}\label{lem: strongly-entire-vectors_invariant}
		$\H_\rho^{\mcO}$ is both $G$- and $\g$-invariant.
	\end{lemma}
	\begin{proof}
		Let $\psi \in \H_\rho^{\mcO}$ and let $B \subseteq \g_\C$ be compact. As the adjoint action of $G$ on $\g_\C$ is continuous, $\Ad_g(B)$ is again compact in $\g_\C$ for every $g \in G$. Since $\rho(g)$ is unitary, we find that $q_B(\rho(g)\psi) = q_{\Ad_{g^{-1}}(B)}(\psi) < \infty$. Thus $\rho(g)\psi \in \H_\rho^{\mcO}$ and so $\H_\rho^{\mcO}$ is $G$-invariant. The $\g$-invariance of $\H_\rho^{\mcO}$ is immediate from \Fref{lem: invariance_of_entire_vectors_by_la_action}.
	\end{proof}
	
	
	\begin{lemma}\label{lem: series_entire}
		Define for every $m \in \N_{\geq 0}$ the function
		$$ f_m : \g_\C \times \H_\rho^\mcO \to \H_\rho^\mcO, \qquad f_m(\xi, \psi) := {1\over m!}d\rho(\xi^m)\psi.$$
		The series $\sum_{m=0}^\infty f_m(\xi, \psi)$ converges in $\H_\rho^{\mcO}$ for every $\xi \in \g_\C$ and $\psi \in \H_\rho^{\mcO}$. It defines a separately continuous map 
		\begin{equation}\label{eq: orbit_map_local_charts_inlem}
			f : \g_\C \times \H_\rho^{\mcO} \to \H_\rho^{\mcO}, \qquad f(\xi, \psi) := \sum_{m=0}^\infty f_m(\xi, \psi).
		\end{equation}
		In particular, the map $\g_\C \to \H_\rho^\mcO, \; \xi \mapsto f(\xi, \psi)$ is entire for every $\psi \in \H_\rho^\mcO$.
	\end{lemma}
	\begin{proof}
		\noindent
		Let $\xi \in \g_\C$, $t>1$ and let $B \subseteq \g_\C$ be a compact subset. According to \Fref{lem: invariance_of_entire_vectors_by_la_action}, there is a constant $C>0$ and a compact subset $B^\prime \subseteq \g_\C$ s.t.\ ${1\over m!}q_B(d\rho(\xi^m)\psi) \leq C t^{-m} q_{B^\prime}(\psi)$ for every $m \in \N_{\geq 0}$ and $\psi \in \H_\rho^\mcO$. So for $\psi \in \H_\rho^\mcO$ we have
		$$ \sum_{m=0}^\infty {1\over m!} q_B(d\rho(\xi^m)\psi) \leq C \bigg(\sum_{m=0}^\infty t^{-m}\bigg) q_{B^\prime}(\psi) = {C \over 1 - t^{-1}} q_{B^\prime}(\psi) < \infty. $$
		We thus find that the series $\sum_{m=0}^\infty f_m(\xi, \psi)$ converges in $\H_\rho^\mcO$, and that 
		$$q_B(f(\xi, \psi)) \leq  {C \over 1 - t^{-1}} q_{B^\prime}(\psi), \qquad \forall \psi \in \H_\rho^\mcO.$$
		In particular, the linear map $\H_\rho^\mcO \to \H_\rho^\mcO, \; \psi \mapsto f(\xi, \psi)$ is continuous. \\
		
		\noindent
		We now show that $f$ is also separately continuous in the $\xi$-variable. Take $\psi \in \H_\rho^{\mcO}$ and define $f^\psi(\xi) := f(\xi, \psi)$ for $\xi \in \g_\C$. Let $B_0 \subseteq \g_\C$ be a compact subset. Consider the functions $h_m : B_0 \to \H_\rho^\mcO$ defined by $h_m(\xi) := {1\over m!} d\rho(\xi^m)\psi$ for $m \in \N_{\geq 0}$. Define also $h := \restr{f^\psi}{B_0}$. So $h(\xi) = \sum_{m=0}^\infty h_m(\xi) = f^\psi(\xi)$ for $\xi \in B_0$. We show that $h_m$ is continuous for every $m \in \N_{\geq 0}$ and that $\sum_{m=0}^N h_m \to h$ uniformly on $B_0$ as $N \to \infty$. This will imply that $h$ is continuous. \\
		
		\noindent
		Let $m \in \N_{\geq 0}$. Assume that $(\eta_n)_{n \in \N}$ is a sequence in $B_0$ with $\eta_n \to \eta$ for some $\eta \in B_0$. Let $B \subseteq \g_\C$ be a compact subset and let $t > 1$. Using \Fref{lem: invariance_of_entire_vectors_by_la_action}, there exists a constant $C > 0$ and a compact subset $B^\prime \subseteq \g_\C$ containing $B$ such that \eqref{eq: inv_of_entire_vectors_main_eq} holds true. Let $m \in \N_{\geq 0}$. Notice that
		\begin{equation}\label{eq: continuity_of_terms_inlem}
			q_B(h_m(\eta_n) - h_m(\eta)) = {1\over m!}q_B(d\rho(\eta_n^m - \eta^m)\psi) = {1\over m!} \sum_{k=0}^\infty {1\over k!} p_B^k(d\rho(\eta_n^m - \eta^m)\psi).
		\end{equation}
		Since the multilinear map
		$$ \g_\C^m \to \H_\rho^\infty, \qquad (\xi_1, \cdots, \xi_m) \mapsto d\rho(\xi_1, \cdots, \xi_m)\psi $$
		is continuous w.r.t.\ the strong topology on $\H_\rho^\infty$ \citep[Lem.\ 3.22]{BasNeeb_ProjReps} and $p_B^k$ is a continuous seminorm on $\H_\rho^\infty$, it follows that the function
		$$ \tau_N : B_0 \to [0, \infty), \qquad \tau_N(\xi) := \sum_{k=0}^N {1\over k!} p_B^k(d\rho(\xi^m - \eta^m)\psi) $$
		is continuous for every $N \in \N$. Moreover, in view of \eqref{eq: inv_of_entire_vectors_main_eq} we have that
		$$\sum_{k=0}^\infty {1\over k!}\sup_{\xi \in B_0} p_B^k(d\rho(\xi^m - \eta^m)\psi) \leq 2 \sum_{k=0}^\infty {1\over k!} \sup_{\zeta\in B_0}p_B^k(d\rho(\zeta^m)\psi) \leq 2C{m! \over t^m}q_{B^\prime}(\psi)< \infty. $$
		It follows that the continuous functions $\tau_N$ converge uniformly to 
		$$ \tau : B_0 \to [0, \infty), \qquad \tau(\xi) := \sum_{k=0}^\infty {1\over k!} p_B^k(d\rho(\xi^m - \eta^m)\psi).$$
		It follows that $\tau$ is continuous. In particular we obtain that $\tau(\eta_n) \to \tau(\eta) = 0$ as $n \to \infty$, which in view of \fref{eq: continuity_of_terms_inlem} implies that
		$$ q_B(h_m(\eta_n) - h_m(\eta)) \to 0 \text{ as } n \to \infty.$$
		We can thus conclude that $h_m$ is continuous for every $m \in \N_{\geq 0}$.\\
		
		\noindent
		With $t>1$, $B, B^\prime$ and $C > 0$ as above, notice using \Fref{lem: invariance_of_entire_vectors_by_la_action} that
		$$ \sum_{m=0}^\infty \sup_{\xi \in B_0} q_B(h_m(\xi)) = \sum_{m=0}^\infty {1\over m!}\sup_{\xi \in B_0} q_B(d\rho(\xi^m)\psi) \leq {C \over 1 - t^{-1}}q_{B^\prime}(\psi) < \infty. $$
		Hence $\sum_{m=0}^N h_m \to h$ uniformly on $B_0$ as $N \to \infty$. Thus $h = \restr{f^\psi}{B_0}$ is continuous. \\
		
		\noindent
		We have shown that $\restr{f^\psi}{B_0} : B_0 \to \H_\rho^\mcO$ is continuous for every compact subset $B_0 \subseteq \g_\C$. Since $\g_\C$ is Fr\'echet and thus first-countable, it is compactly generated in the sense that a set $S \subseteq \g_\C$ is open if and only if $S \cap B_0$ is open in $B_0$ for every compact $B_0 \subseteq \g_\C$. From the fact that $\restr{f^\psi}{B_0}$ is continuous for every compact subset $B_0 \subseteq \g_\C$, it therefore follows that $f^\psi$ is continuous. So $f$ is indeed separately continuous.
	\end{proof}
	
	\begin{lemma}\label{lem: series_function_extends_local_repr}
		Consider the function $f$ in \eqref{eq: orbit_map_local_charts_inlem}. For any $\xi \in \g$ and $\psi \in \H_\rho^{\mcO}$, we have $f(\xi, \psi) = \rho(e^\xi)\psi$.
	\end{lemma}
	\begin{proof}
		Let $\psi \in \H_\rho^{\mcO}$. Then $\psi$ is an analytic vector for $\rho$, by \Fref{cor: entire_vectors_seminorm_finite}. In view of \Fref{lem: series_entire}, we find that the two maps $\xi \mapsto \rho(e^\xi)\psi$ and $\xi \mapsto f(\xi, \psi)$ are both real analytic as maps $\g \to \H_\rho$. They moreover have the same image under the jet-projection $j^\infty_0 : C^\omega(\g, \H_\rho) \to P(\g, \H_\rho)$. Using \Fref{prop: jet_vanishes_then_function_trivial}, we conclude that $\rho(e^\xi)\psi = f(\xi, \psi)$ for all $\xi \in \g$.\qedhere
	\end{proof}
	
	\begin{lemma}
		The orbit map $G \to \H_\rho^{\mcO}, \; g \mapsto \rho(g)v$ is analytic for any $\psi \in \H_\rho^\mcO$.
	\end{lemma}
	\begin{proof}
		Let $\psi \in \H_\rho^\mcO$. As the Lie group $G$ is BCH, \Fref{lem: series_entire} and \Fref{lem: series_function_extends_local_repr} together imply that the map $G \to \H_\rho^{\mcO}, \; g \mapsto \rho(g)\psi$ is analytic on $U$ for some $1$-neighborhood $U \subseteq G$, which implies the assertion.
	\end{proof}
	
	\section{A general approach to holomorphic induction}\label{sec: general_hol_induction}
	
	\noindent
	In this section, we define and study a notion of holomorphic induction for unitary representations of Lie groups. The presented definition and results extend that of \citep{Neeb_hol_reps}, by removing the requirement of norm-continuity of the representation being induced. The precise setting we consider is as described below.\\
	
	\noindent
	Let $G$ be a connected BCH Fr\'echet--Lie group with Lie algebra $\g$. Let $H \subseteq G$ be a connected and closed subgroup of $G$, and assume that $H$ is a locally exponential Lie subgroup of $G$ (cf.\ \citep[Def.\ IV.3.2]{neeb_towards_lie}). Let $\h := \bm{\mrm{L}}(H)$ be its Lie algebra, which we identify as Lie subalgebra of $\g$ using the pushforward of the inclusion $H \hookrightarrow G$. We then have $\exp_G(\xi) = \exp_H(\xi)$ for all $\xi \in \h \subseteq \g$. We furthermore have:
	\begin{lemma}\label{lem: subgroup_embedded_bch}
		There exists an open neighborhood $U$ of $\g$ s.t.\ $\restr{\exp_G}{U}$ is an analytic diffeomorphism onto an open neighborhood in $G$ and such that 
		$$\exp_G(U \cap \h)~=~\exp_G(U) \cap H.$$
		In particular, $H$ is BCH and an analytic embedded Lie subgroup of $G$.
	\end{lemma}
	\begin{proof}
		Since $H$ is a locally exponential Lie subgroup of $G$ by assumption, it follows using \citep[Thm.\ IV.3.3]{neeb_towards_lie} that there exists an open neighborhood $U \subseteq \g$ such that $\restr{\exp_G}{U}$ is a diffeomorphism onto an open $1$-neighborhood in $G$ and $\exp_G(U \cap \h)= \exp_G(U) \cap H$. Consequently, $H$ is an analytic embedded Lie subgroup of $G$, and the exponential map $\exp_G$ can be used to obtain analytic slice charts for $H$. Since $G$ is BCH, it follows that also $H$ is BCH.
	\end{proof}
	
	\noindent
	Let $\theta : \g_\C \to \g_\C$ be the conjugation defined by $\theta(\xi + i\eta) = \xi - i\eta$ for $\xi, \eta \in \g$. We assume given a triangular decomposition 
	$$\g_\C = \n_- \oplus \h_\C \oplus \n_+,$$
	where $\n_\pm$ and $\h_\C$ are closed Lie subalgebras of $\g_\C$ satisfying $\theta(\n_\pm) \subseteq \n_\mp$, $\theta(\h_\C) \subseteq \h_\C$ and $[\h_\C, \n_\pm] \subseteq \n_\pm$. Set $\b_\pm := \n_\pm \rtimes \h_\C$. \\
	
	\noindent
	The structure of this chapter is as follows. In \Fref{sec: cplx_bundles} we establish some notation and preliminary definitions, in particular specifying a certain space of functions on $G$ that takes the role usually taken by the holomorphic sections of a complex homogeneous vector bundle over $G/H$. In \Fref{sec: hol_induced} we present the definition of holomorphically induced representations and establish an equivalent characterization. We then proceed in \Fref{sec: uniqueness}, \Fref{sec: commutants} and \Fref{sec: holomorphic_induction_in_stages} to study the most important properties enjoyed by holomorphically induced representations.\\
	
	\noindent
	As the theory of this section no longer has a clear interpretation in terms of holomorphic maps, we present in \Fref{sec: geometric_hol_induction} a stronger notion that involves complex geometry. The approach presented there depends crucially on the availability of a dense set of b-strongly-entire vectors in the representation being induced.

	\subsection{A substitute for holomorphic sections}\label{sec: cplx_bundles}
	
	\noindent
	Let $(\sigma, V_\sigma)$ be an analytic unitary representation of $H$. Let us establish some notation and preliminary definitions.
	
	\begin{definition}\label{def: lie_derivatives}
		For $\xi \in \g$, define the differential operators $\L_{\bm{v}(\xi)}$ and $\L_{\bm{v}(\xi)^r}$ on $C^\infty(G, V_\sigma)$ by
		\begin{align*}
			(\L_{\bm{v}(\xi)}f)(g) &:= \restr{d\over dt}{t=0}f(ge^{t\xi}),\\
			(\L_{\bm{v}(\xi)^r}f)(g) &:= \restr{d\over dt}{t=0}f(e^{-t\xi}g), \qquad \forall g \in G, \; f \in C^\infty(G, V_\sigma).
		\end{align*}
		Extend both $\xi \mapsto \L_{\bm{v}(\xi)}$ and $\xi \mapsto \L_{\bm{v}(\xi)^r}$ $\C$-linearly to $\g_\C$ and further to algebra homomorphisms on $\mc{U}(\g_\C)$, so we have e.g.\ $\L_{\bm{v}(\xi_1\cdots \xi_n)^r} = \L_{\bm{v}(\xi_1)^r}\cdots \L_{\bm{v}(\xi_n)^r}$ for all $\xi_k \in \g_\C$ and $k\in \{1,\ldots, n\}$. 
	\end{definition}
	
	\begin{remark}
		We thus adopt the convention that for $\xi \in \g$, $\bm{v}(\xi)$ denotes the \textit{left}-invariant vector field on $G$ associated to $\xi \in \g$ whereas $\bm{v}(\xi)^r$ denotes the \textit{right}-invariant one.
	\end{remark}
	
	\begin{definition}\label{def: extension}
		Let $\mcD \subseteq V_\sigma^\omega$ be a subspace that is dense in $V_\sigma$.
		\begin{itemize}
			\item An \textit{extension} of $d\sigma$ to $\b_\pm$ with domain $\mcD$ is a Lie algebra homomorphism $\chi : \b_\pm \to \L(\mcD)$ such that $\chi(\xi) = \restr{d \sigma(\xi)}{\mcD}$ for all $\xi \in \h_\C$. We call $(\sigma, \chi)$ an \textit{$(H, \b_-)$-extension pair with domain $\mcD$}.
			\item The \textit{trivial extension} of $d\sigma$ to $\b_\pm$ with domain $\mcD$ is defined by letting $\n_\pm$ act trivially on $\mcD$.
		\end{itemize}
	\end{definition}

	\begin{definition}
		For $k\in \{1,2\}$, let $(\sigma_k, \chi_k)$ be an $(H, \b_-)$-extension pair with domain $\mcD_k$. We say that $(\sigma_1, \chi)$ and $(\sigma_2, \chi_2)$ are \textit{unitarily equivalent} if there is a unitary isomorphism $U : V_{\sigma_1} \to V_{\sigma_2}$ of $H$-representation such that $U \mcD_1 = \mcD_2$ and $U\chi_1(\xi)v = \chi_2(\xi)Uv$ for all $\xi \in \b_-$ and $v \in \mcD_1$. In this case we write $(\sigma_1, \chi_1) \cong (\sigma_2, \chi_2)$.
	\end{definition}

	\begin{definition}
		For $k\in \{1,2\}$, let $(\sigma_k, \chi_k)$ be an $(H, \b_-)$-extension pair with domain $\mcD_k$. Define the direct sum $(\sigma_1, \chi_1) \oplus (\sigma_2, \chi_2) := (\sigma_1 \oplus \sigma_2, \chi_1 \oplus \chi_2)$, where $\chi_1 \oplus \chi_2$ is defined by
		\begin{align*}
			\chi_1\oplus \chi_2& : \b_- \to \L(\mcD_1 \oplus \mcD_2), \qquad (\chi_1 \oplus \chi_2)(\xi)(v_1, v_2) = (\chi_1(\xi)v_1, \chi_2(\xi)v_2).
		\end{align*}
	\end{definition}
	\begin{definition}
		Let $(\sigma, \chi)$ be an $(H, \b_-)$-extension pair with domain $\mcD$. We say that $(\sigma, \chi)$ is \textit{decomposable} if $(\sigma, \chi) \cong (\sigma_1, \chi_1) \oplus (\sigma_2, \chi_2)$ for some non-trivial $(H, \b_-)$-extension pairs $(\sigma_1, \chi_1)$ and $(\sigma_2, \chi_2)$. We say that $(\sigma, \chi)$ is \textit{indecomposable} if it is not decomposable.\\
	\end{definition}
	
	\noindent
	Recall the definition of the involutions $\tau, \theta$ and $(\--)^\ast$ on $\mc{U}(\g_\C)$, specified in \Fref{def: involutions}. Recalling that $\theta(\xi + i\eta) = \theta(\xi - i\eta)$ for $\xi, \eta \in \g$, the involutions $\tau$ and $(\--)^\ast$ satisfy $\tau(\xi) = -\xi$ and $\xi^\ast = -\theta(\xi)$ for $\xi \in \g_\C$. Extensions are used to specify a suitable $G$-subrepresentation of $C^\omega(G, V_\sigma)^H$:
	
	\begin{definition}\label{def: hol_by_extension}
		Let $(\sigma, \chi)$ be an $(H, \b_-)$-extension pair with domain $\mcD$. Define
		\begin{align*}
			C^\omega(G, V_\sigma)^{H} &:= \set{ f \in C^\omega(G, V_\sigma) \st f(gh) = \sigma(h)^{-1}f(g), \qquad \forall \; g \in G, h \in H }\\
			C^\omega(G, V_\sigma)^{H, \chi} &:= \set{ f \in C^\omega(G, V_\sigma)^H \st \langle v, \L_{\bm{v}(\xi)}f\rangle = -\langle \chi(\xi^\ast)v, f\rangle, \qquad \forall \xi \in \b_+, v \in \mcD}.\\
		\end{align*}
	\end{definition}
	
	\noindent
	Let us next record an important property regarding $C^\omega(G, V_\sigma)^{H, \chi}$:
	
	\begin{proposition}\label{prop: equivariance_property_qhol_fns}
		Let $(\sigma, \chi)$ be an $(H, \b_-)$-extension pair with domain $\mcD$. Let $f \in C^\omega(G, V_\sigma)^{H, \chi}$. Then 
		$$f(g) \in \dom(\chi(x^\ast)^\ast) \quad\text{ and }\quad (\L_{\bm{v}(\tau(x))} f)(g) = \chi(x^\ast)^\ast f(g), \qquad \forall x \in \mc{U}(\b_+), \; \forall  g \in G.$$
	\end{proposition}
	\begin{proof}
		Let $v \in \mcD$. Suppose that $x = \xi_1\cdots \xi_n$ for $n\in \N$ and $\xi_i \in \b_+$. Observe that $f(g) \in \dom(\chi(\eta^\ast)^\ast)$ and $(\L_{\bm{v}(\eta)}f)(g) = -\chi(\eta^\ast)^\ast f(g)$ for any $g \in G$ and $\eta \in \b_+$, as a consequence of \Fref{def: hol_by_extension}. It follows by induction on $n \in \N$ that $\langle v, \L_{\bm{v}(\xi_1\cdots \xi_n)}f\rangle = (-1)^n\langle \chi(\xi_1^\ast \cdots \xi_n^\ast)v, f\rangle$. This implies $\langle v, \L_{\bm{v}(\tau(x))}f\rangle = \langle \chi(x^\ast)v, f\rangle$ for any $x \in \mc{U}(\b_+)$ and $v \in \mcD$. The assertion follows.
	\end{proof}
	
	\subsection{Holomorphically induced representations}\label{sec: hol_induced}
	
	\noindent
	We now define holomorphically induced representations. Fix throughout the section an $(H, \b_-)$-extension pair $(\sigma, \chi)$ with a domain $\mcD_\chi \subseteq V_\sigma^\omega$ that is dense in $V_\sigma$. Let $(\rho, \H_\rho)$ be a unitary representation of $G$.\\

	\begin{remark}
		The theory of holomorphic induction presented in the upcoming section makes use of reproducing kernel Hilbert spaces. For more details thereon, one may refer e.g. to \citep[Ch.\ I-II]{Neeb_book_hol_conv}. The most relevant properties are recalled in \Fref{sec: reproducing_hspaces} below.\\
	\end{remark}


	\begin{definition}\label{def: hol_induced}
		We say that $(\rho, \H_\rho)$ is \textit{holomorphically induced} from $(\sigma, \chi)$ if there exists a $G$-equivariant injective linear map $\Phi : \H_\rho \hookrightarrow \Map(G, V_\sigma)^H$ satisfying the following conditions:
		\begin{enumerate}
			\item The point evaluation $\mc{E}_x : \H_\rho \to V_\sigma, \; \mc{E}_x(\psi) := \Phi_\psi(x)$ is continuous for every $x \in G$.
			\item $\mc{E}_x\mc{E}_x^\ast = \id_{V_\sigma}$ for every $x \in G$.
			\item $\mcD_\chi = \set{v \in V_\sigma \st \Phi(\mc{E}_{e}^\ast v) \in C^\omega(G, V_\sigma)^{H, \chi}}$.
		\end{enumerate}
	\end{definition}
	
	\begin{remark}
		The first condition entails that $(\rho, \H_\rho)$ is unitarily equivalent to the natural $G$-representation on the reproducing kernel Hilbert space $\H_Q$, where $Q \in C(G \times G, \B(V_\sigma))^{H \times H}$ is the positive definite and $G$-invariant kernel defined by $Q(x,y) := \mc{E}_x\mc{E}_y^\ast$, cf.\ \Fref{thm: properties_repr_kernel} and \Fref{prop: homog_repr_h_space_unitary_repr} below.\\
	\end{remark}


	\noindent
	We have the following equivalent characterization, whose proof comprises the remainder of the section:
	
	\begin{restatable}{theorem}{holinduced}\label{thm: hol_induced_equiv_char}
		The following assertions are equivalent. 
		\begin{enumerate}
			\item The $G$-representation $(\rho, \H_\rho)$ is holomorphically induced from the $(H, \b_-)$-extension pair $(\sigma, \chi)$.
			\item There is a closed $H$-invariant subspace $V \subseteq \H_\rho$ with the following properties:
			\begin{enumerate}
				\item $V$ is cyclic for the $G$-representation $\H_\rho$.
				\item $\mcD_{\widetilde{\chi}} := V \cap \H_\rho^\omega$ satisfies $d\rho(\n_-)\mcD_{\widetilde{\chi}} \subseteq \mcD_{\widetilde{\chi}}$. 
				\item $(\sigma, \chi)$ is unitarily equivalent to $(\widetilde{\sigma}, \widetilde{\chi})$, where $(\widetilde{\sigma}, \widetilde{\chi})$ is the $(H, \b_-)$-extension pair defined by
				\begin{alignat*}{2}
					\widetilde{\sigma} : H &\to \U(V),& \qquad \widetilde{\chi} : \b_- &\to \L(\mcD_{\widetilde{\chi}}),\\
					\widetilde{\sigma}(h) &:= \restr{\rho(h)}{V}, &													\widetilde{\chi}(\xi) &:= \restr{d\rho(\xi)}{\mcD_{\widetilde{\chi}}}.
				\end{alignat*} 
				In particular, $\mcD_{\widetilde{\chi}}$ is dense in $V$.
			\end{enumerate}
		\end{enumerate}
		If these equivalent assertions are satisfied, then $\rho$ is an analytic $G$-representation.\\
	\end{restatable}

	\noindent
	We proceed with the proof of \Fref{thm: hol_induced_equiv_char}. We have the following simple but important observation:
	
	\newpage
	\begin{lemma}\label{lem: hol_induced_1}
		Let $\Phi : \H_\rho \hookrightarrow \Map(G, V_\sigma)^H$ be a $G$-equivariant injective linear map. Assume that the point evaluation $\mc{E}_x(\psi) := \Phi_\psi(x)$ is continuous for every $x \in G$. Define $f_v := \Phi(\mc{E}_e^\ast v) \in \Map(G, V_\sigma)^H$ for $v \in V_\sigma$. Then:
		\begin{enumerate}
			\item $\mc{E}_g = \mc{E}_e \rho(g)^{-1}$ for any $g \in G$.
			\item $\Phi_\psi(g) = \mc{E}_e \rho(g)^{-1}\psi$ for any $\psi \in \H_\rho$. In particular, $f_v(g) = \mc{E}_e \rho(g)^{-1}\mc{E}_e^\ast v$ for $v \in V_\sigma$.
			\item Let $v \in V_\sigma$. Then 
			$$\mc{E}_e^\ast v \in \H_\rho^\omega \iff f_v \in C^\omega(G, V_\sigma)^H \iff \langle v, f_v\rangle \in C^\omega(G, \C).$$
		\end{enumerate}
	\end{lemma}
	\begin{proof}~
		\begin{enumerate}
			\item As $\Phi$ is $G$-equivariant, we have $\mc{E}_g \psi = \Phi_\psi(g) = \Phi_{\rho(g)^{-1}\psi}(e) = \mc{E}_e \rho(g)^{-1}\psi$ for any $\psi \in \H_\rho$.
			\item This is immediate from the first assertion.
			\item Let $v \in V_\sigma$. If $\mc{E}_e^\ast v \in \H_\rho^\omega$, then the orbit map $g \mapsto \rho(g)\mc{E}_e^\ast v$ is analytic $G \to \H_\rho$. It follows that $f_v(g) = \mc{E}_e \rho(g)^{-1}\mc{E}_e^\ast v$ is analytic $G \to V_\sigma$, which in turn implies that $\langle v, f_v\rangle \in C^\omega(G, \C)$. Assume that $\langle v, f_v\rangle \in C^\omega(G, \C)$. Then $g \mapsto \langle \mc{E}_e^\ast v , \rho(g)\mc{E}_e^\ast v\rangle_{\H_\rho} = \langle v, f_v(g^{-1})\rangle_V$ is analytic. As $G$ is a BCH Fr\'echet--Lie group, this implies using \citep[Thm.\ 5.2]{Neeb_analytic_vectors} that $\mc{E}_e^\ast v \in \H_\rho^\omega$.\qedhere
		\end{enumerate}
	\end{proof}
	
	\noindent
	We first prove that $\textrm{(1)} \implies \textrm{(2)}$ in \Fref{thm: hol_induced_equiv_char}. Assume that $\rho$ is holomorphically induced from $(\sigma, \chi)$. Let the map $\Phi : \H_\rho \to \Map(G, V_\sigma)^H$ satisfy the conditions in \Fref{def: hol_induced}. Let $\mc{E}_x := \ev_x \circ \Phi$ be the point evaluation at $x \in G$. We write $f_v := \Phi(\mc{E}_e^\ast v) \in C^\omega(G, V_\sigma)^{H, \chi}$ for $v \in \mcD_\chi$. \\
	
	\noindent
	We show that the $H$-invariant subspace $W := \mc{E}_e^\ast V_\sigma \subseteq \H_\rho$ satisfies the conditions in \Fref{thm: hol_induced_equiv_char}. Define $\mc{D}_{\widetilde{\chi}} := \mc{E}_e^\ast \mc{D}_\chi \subseteq W$. By \Fref{thm: properties_repr_kernel} we know that $\rho(G) W = \bigcup_{g \in G}\mc{E}_g^\ast V_\sigma$ is total in $\H_\rho$, so that $W$ is cyclic for $\rho$. It is moreover immediate from \Fref{lem: hol_induced_1} that $\mcD_{\widetilde{\chi}} \subseteq \H_\rho^\omega$. Because $\mcD_{\widetilde{\chi}}$ is dense in the cyclic subspace $W$ and $\H_\rho^\omega$ is $G$-invariant, we obtain that $\H_\rho^\omega$ is dense in $\H_\rho$. Hence $\rho$ is analytic.

	\begin{lemma}\label{lem: second_properties_hol_induced}
		Let $v \in \mcD_\chi$. The following assertions hold true: 
		\begin{enumerate}
			\item $\mc{E}_e \rho(g) \mc{E}_e^\ast v \in \dom(\chi(x^\ast)^\ast)$ and $\mc{E}_e d\rho(x)\rho(g)\mc{E}_e^\ast v = \chi(x^\ast)^\ast \mc{E}_e \rho(g) \mc{E}_e^\ast v$ for any $x\in \mc{U}(\b_+)$ and $g \in G$.
			\item $d\rho(\b_-)\mcD_{\widetilde{\chi}} \subseteq \mcD_{\widetilde{\chi}}$ and $d\rho(x)\mc{E}_e^\ast v  = \mc{E}_e^\ast \chi(x) v$ for any $x \in \mc{U}(\b_-)$.
		\end{enumerate}
	\end{lemma}
	\begin{proof}~
		\begin{enumerate}
			\item Let $x\in \mc{U}(\b_+)$. Since $f_v \in C^\omega(G, V_\sigma)^{H, \chi}$, we obtain from \Fref{prop: equivariance_property_qhol_fns} that $f_v(e) \subseteq \dom(\chi(x^\ast)^\ast)$ and that $\L_{\bm{v}(\tau(x))}f_v = \chi(x^\ast)^\ast f_v$. On the other hand, notice using the formula $f_v(g) = \mc{E}_e \rho(g)^{-1} \mc{E}_e^\ast v$ that $(\L_{\bm{v}(\tau(x))}f_v)(g) = \mc{E}_e d\rho(x)\rho(g)^{-1}\mc{E}_e^\ast v$ holds true for any $g \in G$, say by induction on the degree of $x$. We thus obtain that $\mc{E}_e d\rho(x)\rho(g)^{-1}\mc{E}_e^\ast v = \chi(x^\ast)^\ast f_v(g) = \chi(x^\ast)^\ast \mc{E}_e \rho(g)^{-1} \mc{E}_e^\ast v$ for any $g \in G$.
			\item Let $x \in \mc{U}(\b_-)$. Recall from \Fref{lem: hol_induced_1} that $\mcD_{\widetilde{\chi}} \subseteq \H_\rho^\omega$. Let $\psi \in \rho(G)\mcD_{\widetilde{\chi}}$. Using the first assertion, observe that $$\langle \mc{E}_e^\ast \chi(x)v, \psi\rangle = \langle v, \chi(x)^\ast\mc{E}_e\psi\rangle = \langle v, \mc{E}_e d\rho(x^\ast)\psi\rangle = \langle d\rho(x)\mc{E}_e^\ast v, \psi\rangle.$$
			As $\mcD_{\widetilde{\chi}}$ is cyclic for $G$ in $\H_\rho$, it follows that $\mc{E}_e^\ast \chi(x)v = d\rho(x)\mc{E}_e^\ast v$. In particular $d\rho(\b_-)\mcD_{\widetilde{\chi}} \subseteq \mcD_{\widetilde{\chi}}$.\qedhere
		\end{enumerate}
	\end{proof}
	
	\noindent
	Define the unitary $H$-action $\widetilde{\sigma}$ on $W$ by $\widetilde{\sigma}(h) = \restr{\rho(h)}{W}$. Consider the extension $\widetilde{\chi}(\xi) := \restr{d\rho(\xi)}{\mcD_{\widetilde{\chi}}}$ of $d\widetilde{\sigma}$ to $\b_-$, whose domain is $\mcD_{\widetilde{\chi}}$. By \Fref{lem: second_properties_hol_induced}, $\mc{E}_e^\ast$ defines a unitary equivalence between $(\sigma, \chi)$ and $(\widetilde{\sigma}, \widetilde{\chi})$. In particular, it follows that $\mcD_{\widetilde{\chi}}$ is dense in $W$, because $\mcD_\chi$ is dense in $V_\sigma$ by assumption.

	\begin{lemma}\label{lem: an_vectors}
		$\mcD_{\widetilde{\chi}} = W\cap \H_\rho^\omega$.
	\end{lemma}
	\begin{proof}
		The inclusion $\mcD_{\widetilde{\chi}} \subseteq W\cap \H_\rho^\omega$ follows from \Fref{lem: hol_induced_1}. Let $w \in W \cap \H_\rho^\omega$. Then $w = \mc{E}_e^\ast v$ for some $v \in V_\sigma$. We must show that $v \in \mcD_\chi$. \Fref{lem: hol_induced_1} implies that $f_v \in C^\omega(G, V_\sigma)^H$. Let $v_2 \in \mcD_\chi$ and $\xi \in \b_+$. Using \Fref{lem: second_properties_hol_induced} and the formula $f_v(g) = \mc{E}_e \rho(g)^{-1}\mc{E}_e^\ast v$ we obtain:
		\begin{align*}
			\langle v_2, (\L_{\bm{v}(\xi)}f_v)(g)\rangle 
			&= -\langle d\rho(\xi^\ast)\mc{E}_e^\ast v_2,\rho(g)^{-1}\mc{E}_e^\ast v \rangle =  -\langle \chi(\xi^\ast)v_2, \mc{E}_e \rho(g)^{-1}\mc{E}_e^\ast v \rangle
			= -\langle \chi(\xi^\ast)v_2, f_v(g) \rangle.
		\end{align*}
		It follows that $f_v \in C^\omega(G, V_\sigma)^{H, \chi}$. By the third property in \Fref{def: hol_induced}, this means that $v \in \mcD_\chi$.
	\end{proof}
	
	\noindent
	This completes the proof of $(1) \implies (2)$ in \Fref{thm: hol_induced_equiv_char}. The converse is \Fref{lem: recognize_hol_induced} below:

	%
	
	\noindent

	\begin{lemma}\label{lem: recognize_hol_induced}
		Let $(\rho, \H_\rho)$ be a unitary representation of $G$. Let $V \subseteq \H_\rho$ be a closed $H$-invariant subspace. Define an $H$-representation $\sigma$ on $V$ by $\sigma(h):= \restr{\rho(h)}{V}$. Set $\mcD_{\chi}  = V \cap \H_\rho^\omega$. Assume that $\rho(G)V$ is total in $\H_\rho$, that $\mcD_{\chi}$ is dense in $V_\sigma$ and that $d\rho(\n_-)\mcD_{\chi} \subseteq \mcD_{\chi}$. Define the extension $\chi(\xi)v := d\rho(\xi) v$ of $d\sigma$ to $\b_-$ with domain $\mcD_{\chi}$, where $\xi \in \b_-$ and $v \in \mcD_\chi$. Then $\rho$ is holomorphically induced from $(\sigma, \chi)$.
	\end{lemma}
	\begin{proof}
		Let $p_V : \H_\rho \to V_\sigma$ denote the orthogonal projection onto $V_\sigma$. For $\psi \in \H_\rho$, define $\Phi_\psi(g) := p_V \rho(g)^{-1}\psi$. Consider the linear map $\Phi : \H_\rho \to C(G, V_\sigma)^H$ defined by $\psi \mapsto \Phi_\psi$. It is clear that $\Phi$ is $G$-equivariant and that the point-evaluation $\mc{E}_g = p_V \rho(g)^{-1}$ is continuous for any $g\in G$. The map $\Phi$ is injective because $\Phi_\psi = 0$ is equivalent to $\psi \perp \rho(G)V$, and $\rho(G)V$ is total in $\H_\rho$ by assumption. Notice next that $\mc{E}_g^\ast v = \rho(g)v$ for any $v \in V$ and so $\mc{E}_g \mc{E}_g^\ast = \id_{V}$. Define $V^0 := \{ v \in V \st \Phi_v \in C^\omega(G, V_\sigma)^{H, \chi}\}$. It remains to show that $\mcD_{\chi} = V^0$. It is immediate from the third assertion in \Fref{lem: hol_induced_1} that $V^0 \subseteq \mcD_{\chi}$. Suppose conversely that $v\in \mcD_{\chi}$. By \Fref{lem: hol_induced_1}, we have $\Phi_v \in C^\omega(G, V_\sigma)^H$. Let $\xi \in \b_+$ and $w \in \mcD_{\chi}$. Using $\L_{\bm{v}(\xi)}\Phi_v(g) = -p_Vd\rho(\xi)\rho(g)^{-1}v$, we find that
		$$ \langle w, \L_{\bm{v}(\xi)}\Phi_v(g)\rangle = -\langle d\rho(\xi^\ast)w, \rho(g)^{-1}v\rangle = -\langle \chi(\xi^\ast)w, \rho(g)^{-1}v\rangle = -\langle \chi(\xi^\ast)w, \Phi_v(g)\rangle.$$
		Thus $\Phi_v \in C^\omega(G, V_\sigma)^{H, \chi}$, which means that $v \in V^0$. Therefore $V^0 = \mcD_{\chi}$.
	\end{proof}
	
	\subsection{Uniqueness}\label{sec: uniqueness}
	
	\noindent
	In the following, we determine that there is up to unitary equivalence at most one unitary $G$-representation that is holomorphically induced from a given $(H, \b_-)$-extension pair. Let $(\sigma, \chi)$ be such an $(H, \b_-)$-extension pair, whose domain $\mcD_\chi~\subseteq~V_\sigma^\omega$ is dense in $V_\sigma$. Let $(\rho, \H_\rho)$ be a unitary representation of $G$.
	
	\begin{definition}
		We say that $(\sigma, \chi)$ is \textit{holomorphically inducible} to $G$ if there is a unitary $G$-representation which is holomorphically induced from $(\sigma, \chi)$. 
	\end{definition}
	
	\begin{proposition}\label{prop: reproducing_function}
		Assume that $\rho$ is holomorphically induced from $(\sigma, \chi)$. Let the map $\Phi : \H_\rho \hookrightarrow \Map(G, V_\sigma)^H$ satisfy the conditions in \Fref{def: hol_induced} and let $\mc{E}_x~:=~\ev_x \circ \Phi$ be the point evaluation at $x \in G$. Define:
		$$F : G \to \B(V_\sigma), \quad F(g) := \mc{E}_e \rho(g)\mc{E}_e^\ast.$$
		Then $F$ satisfies the following properties:
		\begin{enumerate}
			\item $F(e) = \id_{V_\sigma}$.
			\item The function $Q : G \times G \to \B(V_\sigma), \; Q(x,y) := F(x^{-1}y)$ is positive definite (cf.\ \Fref{def: pos_kernel}).
			\item $\mcD_\chi = \set{v \in V_\sigma \st g \mapsto \langle v, F(g) v\rangle \text{ is real-analytic } G\to \C}$.
			\item For all $v,w \in \mcD_\chi$ and $g\in G$, $\xi \in \b_+$ we have:
			\begin{align}\label{eq: defining_reproducing_function}
				\big[\L_{\bm{v}(\xi)^r}\langle w, F v\rangle\big](g) &= -\langle \chi(\xi^\ast) w, F(g)v\rangle.
			\end{align}
		\end{enumerate}
		Finally, $\rho$ is unitarily equivalent to the $G$-representation on the reproducing kernel Hilbert space $\H_Q$.
	\end{proposition}
	\begin{proof}
		Define the $\widetilde{Q} : G \times G \to \B(V_\sigma)$ by $\widetilde{Q}(x,y) := \mc{E}_x\mc{E}_y^\ast$, which is positive-definite by \Fref{thm: properties_repr_kernel}. In view of the first assertion in \Fref{lem: hol_induced_1}, we have $\widetilde{Q}(x,y)v = \mc{E}_e \rho(x^{-1}y)\mc{E}_e^\ast v =F(x^{-1}y)v = Q(x,y)v$ for any $v \in V_\sigma$. Thus $\widetilde{Q} = Q$. In particular, $Q$ is positive definite and $F(e) = Q(e,e) = \id_{V_\sigma}$. Let $v \in V_\sigma$. Writing $f_v := \Phi(\mc{E}_e^\ast v)$, notice that $f_v(g) = F(g^{-1})v$ for $g \in G$. We find that $\mc{E}_e^\ast v\in \H_\rho^\omega$ if and only if $\langle v, F v\rangle \in C^\omega(G, \C)$, using \Fref{lem: hol_induced_1}. Then 
		$$D_\chi = \set{v \in V_\sigma \st \mc{E}_e^\ast v \in \H_\rho^\omega} = \set{v \in V_\sigma \st \langle v, Fv\rangle \in C^\omega(G, \C)},$$
		where we used \Fref{lem: an_vectors} in the first equality. Finally, notice that $\langle w, F(g)v\rangle = \langle \mc{E}_e^\ast w, \rho(g)\mc{E}_e^\ast v\rangle$ for $v,w \in \mcD_\chi$ and $g \in G$. It thus follows from \Fref{lem: second_properties_hol_induced} that $F$ satisfies \eqref{eq: defining_reproducing_function} for all $g \in G$ and $\xi \in \b_+$. The final statement is immediate from \Fref{prop: homog_repr_h_space_unitary_repr}.
	\end{proof}

	\noindent
	The next result, \Fref{thm: inducibility}, gives a characterization of $(\sigma, \chi)$ being holomorphically inducible in terms $\B(V_\sigma)$-valued positive-definite functions on $G$. 
	\newpage
	\begin{theorem}\label{thm: inducibility}
		The following assertions are equivalent:
		\begin{enumerate}
			\item $(\sigma, \chi)$ is holomorphically inducible.		
			\item There is a function $F : G \to \B(V_\sigma)$ satisfying the properties in \Fref{prop: reproducing_function}.
		\end{enumerate}
		Assume that these assertions are valid. Let $F : G \to \B(V_\sigma)$ satisfy the conditions in \Fref{prop: reproducing_function}. Then $F(g)^\ast = F(g^{-1})$ for all $g \in G$. Moreover, for $v \in \mcD_\chi$ and $w \in V_\sigma$ we have:
		\begin{align}
			\big[\L_{\bm{v}(x_+)^r}\L_{\bm{v}(x_-)}\langle w, F v\rangle\big](g) &= \langle w, \chi(\tau(x_+)^\ast)^\ast F(g)\chi(x_-)v\rangle \label{eq: pde_repr_function_1},\\
			\big[\L_{\bm{v}(x_+x_-)}\langle w, F v\rangle\big](e) &= \langle w, \chi(x_+^\ast)^\ast\chi(x_-)v\rangle, \label{eq: pde_repr_function_2}
		\end{align}
		for all $g \in G$ and $x_\pm \in \mc{U}(\b_\pm)$. Finally, the function $F : G \to \B(V_\sigma)$ is unique.
	\end{theorem}
	\begin{proof}
		\noindent
		The implication $(1) \implies (2)$ is immediate from \Fref{prop: reproducing_function}. Conversely, let $F : G \to \B(V_\sigma)$ be a function satisfying the conditions in \Fref{prop: reproducing_function}. Define $Q(x,y) := F(x^{-1}y)$ for $x,y \in G$. Let $\H_\rho$ be the corresponding reproducing kernel Hilbert space. Using \Fref{prop: homog_repr_h_space_unitary_repr} we obtain a unitary representation $\rho$ of $G$ on $\H_\rho$ and a $G$-equivariant injective linear map $\Phi : \H_\rho \to \Map(G, V_\sigma)^H$ for which the point evaluation $\mc{E}_x := \ev_x \circ \Phi$ is continuous and satisfies $\mc{E}_x = \mc{E}_e\rho(x)^{-1}$ for every $x \in G$. From $F(e) = \id_{V_\sigma}$ it follows that $Q(x,x) = \id_{V_\sigma}$ for every $x \in G$. Define $f_v := \Phi(\mc{E}_e^\ast v)$ for $v \in V_\sigma$. \\
		
		\noindent
		To see that $(1)$ holds true, it remains only to show that $\mcD_\chi = \set{v \in V_\sigma \st f_v \in C^\omega(G, V_\sigma)^{H, \chi}}$. Let $x \in G$ and $v \in V_\sigma$. From the equations $f_v(x) = \mc{E}_{x} \mc{E}_e^\ast v = Q(x,e)v = F(x^{-1})v$ and $\mc{E}_{x} \mc{E}_e^\ast v = \mc{E}_e \rho(x) \mc{E}_e^\ast v$, we conclude that $F(x)v = \mc{E}_e \rho(x) \mc{E}_e^\ast v = f_v(x^{-1})$. It follows that 
		$$\mcD_\chi = \set{v \in V_\sigma \st  \langle v, F v\rangle \in C^\omega(G, \C)} = \set{v \in V_\sigma \st f_v \in C^\omega(G, V_\sigma)^H},$$
		where \Fref{lem: hol_induced_1} was used in the second equality. Assume that $f_v \in C^\omega(G, V_\sigma)^H$. Let $w \in \mcD_\chi$ and $\xi \in \b_+$. From the equation $F(g)v = f_v(g^{-1})$ we obtain that $\L_{\bm{v}(\xi)}f_v(g) = \big[\L_{\bm{v}(\xi)^r}Fv\big](g^{-1})$ for any $g \in G$. Using \Fref{eq: defining_reproducing_function} we find:
		\begin{align*}
			\big\langle w, \L_{\bm{v}(\xi)} f_v(g)\big\rangle 
			&= \big[\L_{\bm{v}(\xi)^r}\langle w, F v\rangle\big](g^{-1}) = -\langle \chi(\xi^\ast)w, F(g^{-1})v\rangle = -\langle \chi(\xi^\ast)w, f_v(g)\rangle, \qquad \forall g \in G. 
		\end{align*}
		Hence $f_v \in C^\omega(G, V_\sigma)^{H, \chi}$. Thus $\mcD_\chi = \set{ v \in V_\sigma \st f_v \in C^\omega(G, V_\sigma)^{H, \chi} }$. We conclude that $(\rho, \H_\rho)$ is holomorphically induced from $(\sigma, \chi)$. So $(1) \iff (2)$. \\
		
		\noindent
		Assume these equivalent assertions are satisfied. From $F(g) = \mc{E}_e \rho(g)\mc{E}_e^\ast$ it is immediate that $F(g^{-1}) = F(g)^\ast$ for all $g \in G$. We next show \eqref{eq: pde_repr_function_1} and \eqref{eq: pde_repr_function_2}. Let $v \in \mcD_\chi$. Notice using $F(g) = \mc{E}_e \rho(g)\mc{E}_e^\ast$ that for any $x,y \in \mc{U}(\g_\C)$ we have
		\begin{equation}\label{eq: equiv_inducible_1}
			\big[\L_{\bm{v}(y)^r}\L_{\bm{v}(x)}F v\big](g) = \mc{E}_ed\rho(\tau(y))\rho(g)d\rho(x)\mc{E}_e^\ast v, \qquad \forall g \in G.	
		\end{equation}
		Thus, for $x_\pm \in \mc{U}(\b_\pm)$ we obtain using \eqref{eq: equiv_inducible_1} and \Fref{lem: second_properties_hol_induced} that
		\begin{align}
				\big[\L_{\bm{v}(x_+)^r}\L_{\bm{v}(x_-)}F v\big](g) &= \mc{E}_ed\rho(\tau(x_+))\rho(g)d\rho(x_-)\mc{E}_e^\ast v = \chi(\tau(x_+)^\ast)^\ast\mc{E}_e\rho(g)\mc{E}_e^\ast \chi(x_-) v,\label{eq: equiv_inducible_2b}\\
				\big[\L_{\bm{v}(x_+x_-)}F v\big](e) &= \mc{E}_ed\rho(x_+x_-)\mc{E}_e^\ast v = \chi(x_+^\ast)^\ast\chi(x_-)v.\label{eq: equiv_inducible_2a}.
		\end{align}
		From \eqref{eq: equiv_inducible_2b} we conclude that $\big[\L_{\bm{v}(x_+)^r}\L_{\bm{v}(x_-)}F v\big](g) = \chi(\tau(x_+)^\ast)^\ast F(g) \chi(x_-) v$ for all $g \in G$, which implies \eqref{eq: pde_repr_function_1}. On the other hand, \eqref{eq: pde_repr_function_2} is implied by \eqref{eq: equiv_inducible_2a}. Finally, assume that $F_1$ and $F_2$ are two functions satisfying the conditions in \Fref{prop: reproducing_function}. Let $v \in \mcD_{\chi}$. The functions $g \mapsto F_1(g)v$ and $g\mapsto F_2(g)v$ are both analytic and satisfy \eqref{eq: equiv_inducible_2a}. As $\mc{U}(\g_\C)$ is spanned by $\mc{U}(\n_+)\mc{U}(\b_-)$ by the PBW Theorem, it follows that $j^\infty_e(F_1v) = j^\infty_e(F_2v)$. As $G$ is connected, it follows from \Fref{prop: jet_vanishes_then_function_trivial} that $F_1(g)v = F_2(g)v$ for all $g \in G$ and $v \in \mcD_\chi$. For any fixed $g \in G$, the map $v \mapsto (F_1(g) - F_2(g))v$ is continuous and vanishes on the dense subset $\mcD_\chi \subseteq V_\sigma$. Hence $F_1 = F_2$.
	\end{proof}
	
	\noindent
	Combining \Fref{prop: reproducing_function} with the uniqueness of $F : G \to \B(V_\sigma)$ in \Fref{thm: inducibility}, we obtain the desired uniqueness of the holomorphically induced representation up to unitary equivalence:
	
	\begin{theorem}\label{thm: uniqueness}
		Let $\rho_1$ and $\rho_2$ be unitary $G$-representations which are both holomorphically induced from $(\sigma, \chi)$. Then $\rho_1 \cong \rho_2$ as unitary $G$-representations.\\
	\end{theorem}

	\noindent
	Finally, we focus our attention on the important special case where $\chi$ is a trivial extension. Using the PBW Theorem, notice that we have the vector space decomposition $\mc{U}(\g_\C) = \mc{U}(\h_\C) \oplus (\n_+\mc{U}(\g_\C) + \mc{U}(\g_\C)\n_-)$.
	
	\begin{definition}\label{def: cond_exp_univ_env_alg}
		Let $E_0 : \mc{U}(\g_\C) \to \mc{U}(\h_\C) \cong \mc{U}(\g_\C)/(\n_+\mc{U}(\g_\C) + \mc{U}(\g_\C)\n_-)$ be the quotient map.\\
	\end{definition}
	
	\begin{lemma}\label{lem: f_v_factors_through_h}
		Assume that $\rho$ is holomorphically induced from $(\sigma, \chi)$, where $\chi$ is the trivial extension of $d\sigma$ to $\b_-$ with domain $\mcD\subseteq V_\sigma$. Let $v \in \mcD$ and $x \in \mc{U}(\g_\C)$. Then $d\sigma(E_0(x^\ast))v = d\sigma(E_0(x))^\ast v$. Moreover for all $w \in V_\sigma^\infty$ we have 
		\[
			\langle w, d\rho(x)v\rangle = \langle w, d\sigma(E_0(x))v\rangle \quad \text{ and } \quad \langle d\rho(x^\ast)v, w\rangle = \langle v, d\sigma(E_0(x))w\rangle.
		\]
	\end{lemma}
	\begin{proof}
		By \Fref{thm: hol_induced_equiv_char} we may assume that $V_\sigma \subseteq \H_\rho$ is a closed subspace, that $\mcD_\chi = V \cap \H_\rho^\omega$, $\sigma(h) = \restr{\rho(h)}{V_\sigma}$ and $\chi(\xi) = \restr{d\rho(\xi)}{\mcD_\chi}$ for every $h\in H$ and $\xi \in \b_-$. Let $p_V : \H_\rho \to V_\sigma$ be the orthogonal projection onto $V_\sigma$. Take $v\in\mcD_\chi$, $\xi_+\in\n_+$, $x\in\mc{U}(\g_\C)$ and $\xi_-\in\n_-$. From \Fref{lem: second_properties_hol_induced} we obtain $p_Vd\rho(x\xi_-)v = p_Vd\rho(x)\chi(\xi_-)v = 0$ and $p_V d\rho(\xi_+x)v = \chi(\xi_+^\ast)^\ast p_Vd\rho(x)v = 0$. Thus 
		$$p_Vd\rho(x) v = p_Vd\rho(E_0(x)) v = d\sigma(E_0(x))v, \qquad \forall x \in \mc{U}(\g_\C).$$
		Let $w \in \mcD_\chi$. Recall from \Fref{lem: an_vectors} that $\mcD_\chi \subseteq \H_\rho^\omega$. We have:
		\[
			\langle d\sigma(E_0(x^\ast))v, w\rangle 
			= \langle d\rho(x^\ast)v, w\rangle = \langle v, d\rho(x)w\rangle = \langle v, d\sigma(E_0(x))w\rangle
			= \langle d\sigma(E_0(x))^\ast v, w\rangle.
		\]
		As $\mcD_\chi$ is dense in $V_\sigma$ we conclude $d\sigma(E_0(x^\ast))v = d\sigma(E_0(x))^\ast v$. Consequently, if $w \in V_\rho^\infty$ then 
		$$ \langle d\rho(x^\ast)v, w\rangle = \langle d\sigma(E_0(x^\ast))v, w\rangle =  \langle d\sigma(E_0(x))^\ast v, w\rangle = \langle v, d\sigma(E_0(x))w\rangle. \qedhere$$
	\end{proof}
	
	\vspace{.2cm}
	
	\noindent
	We complement \Fref{thm: inducibility} with the following result, regarding the uniqueness of the domain:
	\begin{proposition}\label{prop: unique_domain_triv_ext}
		Let $\sigma$ be an analytic unitary representation of $H$. Assume that there exists a subspace $\mcD_\chi \subseteq V_\sigma^\omega$ dense in $V_\sigma$ for which $(\sigma, \chi)$ is holomorphically inducible, where $\chi$ $\b_- \to \L(\mcD_\chi)$ is the trivial extension of $d\sigma$ to $\b_-$ with domain $\mcD_\chi$. Then $\mcD_\chi$ is unique with this property. 
	\end{proposition}
	\begin{proof}
		Suppose that $\mcD_1$ and $\mcD_2$ are two such domains. For $k\in \{1,2\}$, let $\chi_k$ denote the trivial extension of $d\sigma$ to $\b_-$ with domain $\mcD_k$. By assumption $(\sigma, \chi_k)$ is holomorphically inducible. Let $F_k : G \to \B(V_\sigma)$ satisfy the conditions in \Fref{prop: reproducing_function} for $(\sigma, \chi_k)$. Let $v_k \in \mcD_k$. Observe using \Fref{lem: f_v_factors_through_h} that for any $k \in \{1,2\}$, $x\in \mc{U}(\b_-)$ and $v \in \mcD_k$ we have $\chi_k(x)v = d\sigma(E_0(x))v$ and $\chi_k(x)^\ast v = d\sigma(E_0(x))^\ast v = d\sigma(E_0(x^\ast))v$. Consider the functions $a,b : G \to \C$ defined by $a(g) := \langle v_1, F_1(g)v_2\rangle$ and $b(g) := \langle v_1, F_2(g)v_2\rangle$. Notice that both $a$ and $b$ are analytic, where we remark that $a(g) = \langle F_1(g^{-1})v_1, v_2\rangle$. Let $x_\pm \in \mc{U}(\b_\pm)$. Using \eqref{eq: pde_repr_function_2} we obtain:
		\[
			(\L_{\bm{v}(x_+x_-)}b)(e) 
			= \langle v_1, \chi_2(x_+^\ast)^\ast \chi_2(x_-)v_2\rangle = \langle v_1, d\sigma(E_0(x_+))d\sigma(E_0(x_-))v_2\rangle
			= \langle d\sigma(E_0(x_+^\ast))v_1, d\sigma(E_0(x_-))v_2\rangle.
		\]
		We next compute $(\L_{\bm{v}(x_+x_-)}a)(e)$. Let $\iota: G \to G, g \mapsto g^{-1}$ denote the inversion on $G$ and $\Sigma : \C \to \C, z \mapsto \overline{z}$ the conjugation on $\C$. Define 
		$$h : G \to \C, \qquad h(g) = \langle v_2, F_1(g)v_1\rangle,$$
		so that $a = \Sigma \circ h \circ \iota$. For any $x \in \mc{U}(\g_\C)$ and $f \in C^\infty(G, \C)$, we have 
		\begin{align*}
			\big[\L_{\bm{v}(x)}(f \circ \iota)\big](e) &= (\L_{\bm{v}(\tau(x))}f)(e), \\
			\big[\L_{\bm{v}(x)}(\Sigma \circ f)\big](e) &= \Sigma \left[\L_{\bm{v}(\theta(x))}f\right](e).
		\end{align*}
		Using these equations we obtain that $(\L_{\bm{v}(x)}a)(e) = \Sigma \left[\L_{\bm{v}(x^\ast)}h\right](e)$ for any $x\in\mc{U}(\g_\C)$. By \fref{eq: pde_repr_function_2} we have
		\begin{align*}
			\left[\L_{\bm{v}(x_-^\ast x_+^\ast)}h\right](e) 
			&= \langle v_2, \chi(x_-)^\ast \chi(x_+^\ast)v_1\rangle = \langle v_2, d\sigma(E_0(x_-))^\ast d\sigma(E_0(x_+^\ast))v_1\rangle \\
			&= \langle d\sigma(E_0(x_-)) v_2, d\sigma(E_0(x_+^\ast))v_1\rangle.
		\end{align*}
		Thus
		\[
			(\L_{\bm{v}(x_+x_-)}a)(e) 
			= \Sigma \left[\L_{\bm{v}(x_-^\ast x_+^\ast)}h\right](e) = \langle d\sigma(E_0(x_+^\ast))v_1, d\sigma(E_0(x_-))v_2\rangle 
			= (\L_{\bm{v}(x_+x_-)}b)(e). 
		\]
		As $\mc{U}(\g_\C)$ is spanned by elements in $\mc{U}(\n_+)\mc{U}(\b_-)$ by the PBW Theorem, it follows that $j^\infty_e(a) = j^\infty_e(b)$. Since $G$ is connected, it follows from \Fref{prop: jet_vanishes_then_function_trivial} that $a = b$. Thus $\langle v_1, F_1(g) v_2\rangle = \langle v_1, F_2(g) v_2\rangle$ for all $g \in G$, $v_1 \in \mcD_1$ and $v_2 \in \mcD_2$. As both $\mcD_1$ and $\mcD_2$ are dense, it follows that $F_1 = F_2 =: F$. From the third property in \Fref{prop: reproducing_function}, we conclude that 
		$$\mcD_1 = \mcD_2 = \set{v \in V_\sigma \st g \mapsto \langle v, F(g) v\in C^\omega(G, \C)}.$$
	\end{proof}

	\noindent
	\Fref{thm: uniqueness} and \Fref{prop: unique_domain_triv_ext} justify the following notation:
	
	\begin{definition}\label{def: notation_hol_induced}
		We write $\rho = \HolInd_H^G(\sigma, \chi)$ if $\rho$ is holomorphically induced from $(\sigma, \chi)$. If additionally $\chi$ is the trivial extension of $d\sigma$ to $\b_-$ on some necessarily unique domain $\mcD_\chi\subseteq V_\sigma^\omega$, we simply write $\rho = \HolInd_H^G(\sigma)$.\\
	\end{definition}

	\begin{remark}
		For $k \in \{1,2\}$, let $(\sigma_k, \chi_k)$ be an $(H, \b_-)$-extension pair and let $\rho_k$ be a unitary $G$-representation with $\rho_k = \HolInd_H(\sigma_k, \chi_k)$. In view of \Fref{thm: uniqueness}, one might wonder whether or not $\rho_1 \cong \rho_2$ implies $(\sigma_1, \chi_1) \cong (\sigma_2, \chi_2)$. This turns out to be false. For an explicit and simple counterexample, consider $G = \SU(3)$. Let $H \subseteq G$ be the subgroup consisting of diagonal matrices and let $\b_- \subseteq \sl(3,\C)$ consist of upper-triangular matrices. The defining representation $\rho$ of $G$ on $\C^3$ is holomorphically induced from the two $(H, \b_-)$-extension pairs obtained by restricting $\restr{\rho}{H}$ and $\restr{d\rho}{\b_-}$ to either $V_{\sigma_1} := \C e_1$ or $V_{\sigma_2} := \C e_1 \oplus \C e_2$, as is quickly verified using \Fref{thm: hol_induced_equiv_char}. These are not unitary equivalent.
	\end{remark}

	\subsection{Commutants}\label{sec: commutants}
	
	\noindent
	Suppose that $\rho = \HolInd_H^G(\sigma, \chi)$.
	
	\begin{definition}
		Let $T \in \B(V_\sigma)$. We say that $T$ \textit{commutes with} $(\sigma, \chi)$ if $T~\in~\B(V_\sigma)^H$, $T \mcD_\chi \subseteq \mcD_\chi$ and $T\chi(\xi)v = \chi(\xi)Tv$ for every $\xi \in \b_-$ and $v \in \mcD_\chi$. Define the \textit{$\ast$-closed commutant $\B(V_\sigma)^{H, \chi}$ of} $(\sigma, \chi)$ by 
		$$\B(V_\sigma)^{H, \chi} := \set{T \in \B(V_\sigma)^H \st \text{ both $T$ and $T^\ast$ commute with } (\sigma, \chi) }.$$
	\end{definition}
	
	\begin{remark}\label{rem: projections_in_commutant}
		Orthogonal projections in $\B(V_\sigma)^{H, \chi}$ correspond to direct sum decompositions of $(\sigma, \chi)$. To see this, suppose $p_1 \in \B(V_\sigma)^{H, \chi}$ is an orthogonal projection. Let $p_2 := 1-p_1$. For $k \in \{1,2\}$, define $V_k := p_k V_\sigma$ and $\mcD_k := p_k \mcD_{\chi}\subseteq \mcD_\chi$. Define the $(H, \b_-)$-extension pair $(\sigma_k, \chi_k)$ by $\sigma_k(h) := \restr{\sigma(h)}{V_k}$ and $\chi_k(\xi) := \restr{\chi(\xi)}{\mcD_k}$, where $h \in H$ and $\xi \in \b_-$. Then $(\sigma, \chi) \cong (\sigma_1, \chi_1) \oplus (\sigma_2, \chi_2)$.\\
	\end{remark}
	
	\noindent
	The main results of this section are \Fref{thm: hol_induced_then_commutants_iso} and \Fref{thm: commutant_trivial_extension} below:
	
	\begin{restatable}{theorem}{commutants}\label{thm: hol_induced_then_commutants_iso}
		Suppose that $\rho = \HolInd_H^G(\sigma, \chi)$. Let $V_\sigma$ be a closed subspace of $\H_\rho$ satisfying the conditions in \Fref{thm: hol_induced_equiv_char}$\textit{.2}$. Let $q_V \in \B(\H_\rho)$ be the orthogonal projection onto $V_\sigma$. Then
		\begin{enumerate}
			\item $\B(V_\sigma)^{H, \chi}$ is a von Neumann algebra.
			\item Assume that $q_V \in \rho(G)^{\prime \prime}$. Then 
			$$r : \B(\H_\rho)^G \to \B(V_\sigma)^{H, \chi}, \;r(T) := \restr{T}{V_\sigma}$$
			defines a $\ast$-isomorphism of von Neumann algebras. In particular, $\rho$ is irreducible if and only if $(\sigma, \chi)$ is indecomposable.
		\end{enumerate}
	\end{restatable}
	
	\begin{restatable}{theorem}{commutanttrivialextension}\label{thm: commutant_trivial_extension}
		Consider the setting of \Fref{thm: hol_induced_then_commutants_iso}. Let $\chi : \b_- \to \L(\mcD_\chi)$ denote the trivial extension of $d\sigma$ to $\b_-$ with domain $\mcD_\chi$. The following assertions are valid:
		\begin{enumerate}
			\item $\B(V_\sigma)^{H, \chi} = \B(V_\sigma)^H$.
			\item Assume that $q_V \in \rho(G)^{\prime \prime}$. Then $\B(\H_\rho)^G \cong \B(V_\sigma)^{H}$. In particular, $\rho$ is irreducible if and only if $\sigma$ is.
		\end{enumerate}
	\end{restatable}
	
	\begin{remark}
		In the context of positive energy representations, the case where $\chi$ is a trivial extension is of central importance. In that setting we can typically guarantee that $q_V \in \rho(G)^{\prime \prime}$. The relation between positive energy representations and holomorphic induction is considered \Fref{sec: pe_and_hol_ind}.
	\end{remark}

	\subsubsection*{Proof of \Fref{thm: hol_induced_then_commutants_iso} and \Fref{thm: commutant_trivial_extension}}

	\noindent
	Assume throughout the following that $\rho$ is holomorphically induced from $(\sigma, \chi)$. In view of \Fref{thm: hol_induced_equiv_char}, we may and do assume that $V_\sigma \subseteq \H_\rho$ is a closed subspace, $\sigma(h) = \restr{\rho(h)}{V_\sigma}$ for all $h \in H$, that $\mcD_\chi = V_\sigma \cap \H_\rho^\omega$, $d\rho(\b_-)\mcD_\chi \subseteq \mcD_\chi$ and that $\chi(\xi)v = d\rho(\xi)v$ for all $\xi \in \b_-$ and $v \in \mcD_\chi$. We may further assume that the map $\Phi : \H_\rho \to \Map(G, V_\sigma)^H$ satisfying the conditions in \Fref{def: hol_induced} is given by $\Phi_\psi(g) = p_V\rho(g)^{-1}\psi$. In particular $\mc{E}_e = p_V$ is the orthogonal projection $p_V : \H_\rho \to V_\sigma$ and $\mc{E}_e^\ast = \iota_V$ is the inclusion $\iota_V : V_\sigma \hookrightarrow \H_\rho$. We also have $q_V = \iota_V p_V$.
	
	\begin{lemma}\label{lem: commute_w_bil_form}
		Let $T \in \B(\H_\rho)^{H, \chi}$, $x \in \mc{U}(\g_\C)$ and $v,w \in \mcD_\chi$. Then 
		$$\langle v, Td\rho(x)w\rangle = \langle v, d\rho(x)T w\rangle.$$
	\end{lemma}
	\begin{proof}
		Using the PBW Theorem, it suffices to consider the case where $x = x_+x_-$ for some $x_+ \in \mc{U}(\n_+)$ and $x_- \in \mc{U}(\b_-)$. In that case we obtain using \Fref{lem: second_properties_hol_induced} and the fact that $T \in \B(\H_\rho)^{H, \chi}$:
		\begin{align*}
			\langle v, Td\rho(x)w\rangle 
			&= \langle \chi(\xi_+^\ast)  T^\ast v, \chi(x_-)w\rangle = \langle \chi(\xi_+^\ast)  v, T \chi(x_-)w\rangle \\
			&= \langle \chi(\xi_+^\ast)  v, \chi(x_-)T w\rangle = \langle v, d\rho(x)T w\rangle. \qedhere
		\end{align*}
	\end{proof}
	
	\begin{lemma}\label{lem: suff_cond_to_leave_D_inv}
		Let $T \in \B(V_\sigma)$. Assume that $\langle v, T\rho(e^\xi)w \rangle = \langle v, \rho(e^\xi)Tw \rangle$ for all $v,w \in \mcD_\chi$ and all $\xi$ in some $0$-neighborhood in $\g$. Then $T \mcD_\chi \subseteq \mcD_\chi$ and 
		\begin{equation}\label{eq: PTC}
			\langle w, T\rho(g)v\rangle = \langle w, \rho(g)Tv\rangle, \qquad \forall g \in G, \; \forall v,w \in V_\sigma.	
		\end{equation}
	\end{lemma}
	\begin{proof}
		Let $v,w \in \mcD_\chi$. Both $g\mapsto \langle w, T\rho(g)v\rangle$ and $g\mapsto \langle w, \rho(g)Tv\rangle$ are real-analytic $G \to \C$. As $G$ is BCH, so in particular locally exponential, these functions agree on some $1$-neighborhood in $G$ by assumption. As $G$ is connected, it follows from \Fref{prop: identity_theorem} that they are equal everywhere. We thus obtain that $\langle w, T\rho(g)v\rangle = \langle w, \rho(g)Tv\rangle$ for all $g \in G$. As $\mcD_\chi$ is dense, \fref{eq: PTC} follows. Let $v\in \mcD_\chi$. Then using \eqref{eq: PTC} we find that $\langle T v, \rho(g) Tv \rangle = \langle T v, T\rho(g) v \rangle$ for all $g \in G$. The right-hand side defines a real-analytic function $G \to \C$ because $v \in \H_\rho^\omega$. Thus also $g \mapsto \langle T v, \rho(g) Tv \rangle$ is real-analytic. Recalling that $G$ is a BCH Fr\'echet--Lie group, we conclude using \citep[Thm.\ 5.2]{Neeb_analytic_vectors} that $Tv \in \H_\rho^\omega$. Thus $Tv \in \H_\rho^\omega \cap V_\sigma = \mcD_\chi$.
	\end{proof}

	\begin{lemma}\label{lem: commutant_vna}
		$\B(\H_\rho)^{H, \chi}$ is a von Neumann algebra. Moreover we have
		\begin{equation}\label{eq: commutant_vna}
			\langle w, T\rho(g)v\rangle = \langle w, \rho(g)Tv\rangle, \qquad \forall T \in \B(\H_\rho)^{H, \chi},\; \forall g \in G, \;  \forall v,w \in V_\sigma.	
		\end{equation}
	\end{lemma}
	\begin{proof}
		Let $\mc{N} \subseteq \B(V_\sigma)^H$ denote the von Neumann algebra in $\B(V_\sigma)$ generated by $\B(\H_\rho)^{H, \chi}$. We show $\mc{N} = \B(\H_\rho)^{H, \chi}$. It only remains to show $\mc{N} \subseteq \B(\H_\rho)^{H, \chi}$. As $\mc{N}$ is $\ast$-closed, it suffices to show that $T \mcD_\chi \subseteq \mcD_\chi$ and that $T\chi(\xi)v = \chi(\xi)Tv$ for all $T \in \mc{N}$, $\xi\in \b_-$ and $v \in \mcD_\chi$. Let $T \in \mc{N}$. Let $(T_\lambda)$ be a net in $\B(\H_\rho)^{H, \chi}$ such that $T_\lambda \to T$ strongly. Let $v,w \in \mcD_\chi$ and $x \in \mc{U}(\g_\C)$. Using \Fref{lem: commute_w_bil_form} we have:
		\begin{align}\label{eq: commutant_vna_1}
			\begin{split}
				\langle v, Td\rho(x) w\rangle 
				&= \lim_{\lambda} \; \langle v, T_\lambda d\rho(x) w\rangle
				= \lim_{\lambda} \; \langle v, d\rho(x) T_\lambda w\rangle
				= \lim_{\lambda} \; \langle d\rho(x^\ast)v, T_\lambda w\rangle\\
				&= \langle d\rho(x^\ast)v, T w\rangle
			\end{split}
		\end{align}
		As $v,w \in \mcD_\chi \subseteq H_\rho^\omega$, the orbit maps $g \mapsto \rho(g)v$ and $g \mapsto \rho(g)w$ are both real-analytic $G \to \H_\rho$. We obtain using \eqref{eq: commutant_vna_1} for all $\xi \in \g$ in a small-enough $0$-neighborhood in $\g$ that:
		\begin{align}\label{eq: commutant_vna_1b}
			\begin{split}
				\langle w, T\rho(e^\xi) v\rangle
				&= \sum_{n=0}^\infty{1\over n!}\langle w, Td\rho(\xi^n) v\rangle 
				= \sum_{n=0}^\infty{(-1)^n\over n!}\langle d\rho(\xi^n)w, Tv\rangle 
				= \langle \rho(e^{-\xi})w, T v\rangle \\
				&= \langle w, \rho(e^{\xi}) T v\rangle.
			\end{split}
		\end{align}
		It follows from \Fref{lem: suff_cond_to_leave_D_inv} that $T \mcD_\chi \subseteq \mcD_\chi$ and that \fref{eq: PTC} is valid for $T$. Thus $\mc{N}\mc{D}_\chi \subseteq \mcD_\chi$. Differentiating \eqref{eq: PTC} at the identity $e \in G$ we find that $\langle w, Td\rho(\xi)v\rangle = \langle w, d\rho(\xi)Tv\rangle$ for all $\xi \in \g_\C$ and $w \in \mc{D}_\chi$. Suppose $\xi \in \b_-$. Using that $T \mcD_\chi \subseteq \mcD_\chi$, we obtain
		$$ \langle w, T\chi(\xi)v\rangle = \langle w, Td\rho(\xi)v\rangle=  \langle w, d\rho(\xi)Tv\rangle = \langle w, \chi(\xi)Tv\rangle, \qquad \forall w \in \mcD_\chi,$$
		where \Fref{lem: second_properties_hol_induced} was used in the first and last equality. As $\mcD_\chi$ is dense in $V_\sigma$, it follows for every $\xi \in \b_-$ and $v\in \mcD_\chi$ that $T \chi(\xi)v = \chi(\xi)T v$. Thus $T \in \B(\H_\rho)^{H, \chi}$. Hence $\mc{N} = \B(\H_\rho)^{H, \chi}$.
	\end{proof}
	
	\noindent
	Combined with \Fref{lem: commutant_vna}, \Fref{lem: iso_of_vna} below completes the proof of \Fref{thm: hol_induced_then_commutants_iso}.

	\begin{lemma}\label{lem: iso_of_vna}
		Assume that $q_V\in \rho(G)^{\prime \prime}$. Then the map 
		$$r : \B(\H_\rho)^G \to \B(V_\sigma)^{H, \chi}, \quad r(T) := \restr{T}{V_\sigma}$$
		defines an isomorphism of von Neumann algebras.
	\end{lemma}
	\begin{proof}
		We know using \Fref{lem: commutant_vna} that $\B(V_\sigma)^{H, \chi}$ is a von Neumann algebra. Notice that the assumption $q_V \in \rho(G)^{\prime \prime}$ is equivalent with $T V_\sigma \subseteq V_\sigma$ for every $T \in \B(\H_\rho)^G$. Let $T \in \B(\H_\rho)^G$. Then $T \H_\rho^\omega \subseteq \H_\rho^\omega$ and $TV_\sigma \subseteq V_\sigma$. Recalling that $\mcD_\chi = V_\sigma \cap \H_\rho^\omega$, it follows that $T \mcD_\chi \subseteq \mcD_\chi$. Since both $T$ and $T^\ast$ are in $\B(\H_\rho)^G$, it follows that $r(T) \in \B(V_\sigma)^{H, \chi}$, where we recall that $\restr{\rho(h)}{V_\sigma} = \sigma(h)$ and $\restr{d\rho(\xi)}{\mcD_\chi} = \chi(\xi)$ for $h \in H$ and $\xi \in \b_-$. It is clear that $r$ is a weakly-continuous unital $\ast$-preserving homomorphism. It is also injective, because $r(T) =~0$ implies $ T\rho(G)V_\sigma = \rho(G)TV_\sigma = \{0\}$, which in turn implies that $T = 0$ because $V_\sigma$ is cyclic for the $G$-representation $\H_\rho$. As $r$ is weakly continuous, its image is weakly closed \citep[Thm.\ 4.3.4]{Murphy_operator_algs}. Thus, to see that $r$ is surjective, it suffices to show that its image contains all orthogonal projections in $\B(V_\sigma)^{H, \chi}$. Let $p_1 \in \B(V_\sigma)^{H, \chi}$ be an orthogonal projection and let $p_2 := 1 - p_1$. For $k \in \{1,2\}$, define $V_k, \mcD_k, \sigma_k$ and $\chi_k$ as in \Fref{rem: projections_in_commutant}, so that $(\sigma, \chi) \cong (\sigma_1, \chi_1) \oplus (\sigma_2, \chi_2)$. Let $\H_1$ and $\H_2$ be the closed $G$-invariant subspaces of $\H_\rho$ generated by $V_1$ and $V_2$, respectively. It suffices to show that $\H_1 \perp \H_2$. Let $v_1 \in V_1$ and $v_2 \in V_2$. As $p_1 \in \B(V_\sigma)^{H, \chi}$, it follows from \fref{eq: commutant_vna} that 
		$$\langle v_1, \rho(g)v_2\rangle = \langle v_1, p_1\rho(g)v_2\rangle = \langle v_1, \rho(g)p_1v_2\rangle = 0 , \qquad \forall g \in G.$$
		It follows that $V_1 \perp \rho(G)V_2$, which in turn implies $\H_1 \perp \H_2$. 
	\end{proof}
	
	\noindent
	Finally, it remains to prove \Fref{thm: commutant_trivial_extension}:
	
	\begin{proof}[Proof of \Fref{thm: commutant_trivial_extension}:]
		It remains only to prove the first item. The second will follow using \Fref{thm: hol_induced_then_commutants_iso}. It is clear that $\B(V_\sigma)^{H, \chi} \subseteq \B(V_\sigma)^H$. Conversely, take $T \in \B(V_\sigma)^H$. Let $v,w \in \mcD_\chi$. Then in particular the orbit maps $G \to \H_\rho, g \mapsto \rho(g)v$ and $g \mapsto \rho(g)w$ are real-analytic. Notice that $T \mcD_\chi \subseteq V_\sigma^\infty$ and similarly $T^\ast \mcD_\chi \subseteq V_\sigma^\infty$. Using \Fref{lem: f_v_factors_through_h}, we obtain for all $\xi$ in a small-enough $0$-neighborhood that
		\begin{align*}
			\langle w, T\rho(e^\xi)v\rangle 
			&= \sum_{n=0}^\infty{1\over n!}\langle w, Td\rho(\xi^n)v\rangle
			= \sum_{n=0}^\infty{1\over n!}\langle w, Td\sigma(E_0(\xi^n))v\rangle\\
			&= \sum_{n=0}^\infty{1\over n!}\langle w, d\sigma(E_0(\xi^n))Tv\rangle
			= \sum_{n=0}^\infty{(-1^n)\over n!}\langle d\rho(\xi^n) w, Tv\rangle\\
			&= \langle \rho(e^{-\xi})w, Tv\rangle 
			= \langle w, \rho(e^{\xi})Tv\rangle.
		\end{align*}
		Using \Fref{lem: suff_cond_to_leave_D_inv}, it follows that $T\mcD_\chi\subseteq \mcD_\chi$. Hence $\B(V_\sigma)^H\mcD_\chi \subseteq \mcD_\chi$ and in particular $T^\ast \mcD_\chi\subseteq \mcD_\chi$. Suppose that $v \in \mcD_\chi, \xi_0 \in \h_\C$ and $\xi_- \in \n_-$. Then 
		$$T\chi(\xi_0 + \xi_-)v = Td\sigma(\xi_0)v = d\sigma(\xi_0)Tv = \chi(\xi_0 + \xi_-)Tv.$$
		Therefore, $T \chi(\xi)v = \chi(\xi)Tv$ for all $\xi\in \b_-$ and $v \in \mcD_\chi$. Thus $T \in \B(V_\sigma)^{H, \chi}$. We conclude that $\B(V_\sigma)^{H, \chi} = \B(V_\sigma)^H$.
	\end{proof}	
	
	\subsection{Holomorphic induction in stages}\label{sec: holomorphic_induction_in_stages}
	
	\noindent
	Let us next consider holomorphic induction in stages. We specialize to the context of trivial extensions. Recall from \Fref{sec: holomorphic_induction_in_stages} that $\g_\C = \n_- \oplus \h_\C \oplus \n_+$ and that $H \subseteq G$ is a connected Lie subgroup with $\bm{\mrm{L}}(H) = \h$. Assume similarly that $\h_\C = \a_- \oplus \t_\C \oplus \a_+$, where $\a_\pm$ and $\t_\C$ are closed subalgebras with $\theta(\t_\C) \subseteq \t_\C$, $\theta(\a_\pm) \subseteq \a_\mp$ and $[\t_\C, \a_\pm] \subseteq \a_\pm$. Let $T$ be a connected and closed locally exponential Lie subgroup of $H$ integrating $\t \subseteq \h$. Using the notation of \Fref{def: notation_hol_induced}, we have:
	
	\begin{proposition}[Induction In Stages]\label{prop: induction_in_stages}
		Let $(\rho, \H_\rho)$, $(\sigma, \H_\sigma)$ and $(\nu, \H_\nu)$ be analytic unitary representations of $G$, $H$ and $T$, respectively. Then
		\begin{enumerate}
			\item $\rho = \HolInd_T^G(\nu)$ and $\sigma = \HolInd_T^H(\nu) \;\implies \; \rho = \HolInd_H^G(\sigma)$.
			\item Suppose that $\sigma = \HolInd_T^H(\nu)$ and $\rho = \HolInd_H^G(\sigma)$. Assume w.l.o.g. that $\H_\nu \subseteq\H_\sigma \subseteq \H_\rho$ using \Fref{thm: hol_induced_equiv_char}, the inclusions being $T$- and $H$-equivariant, respectively. If $\H_\nu \cap \H_\rho^\omega$ is dense in $\H_\nu$, then $\rho = \HolInd_T^G(\nu)$.
		\end{enumerate}
	\end{proposition}
	\begin{proof}
		These observations follow from a repeated application of \Fref{thm: hol_induced_equiv_char}. 
		\begin{enumerate}
			\item In view of \Fref{thm: hol_induced_equiv_char}, we may assume that $\H_\nu \subseteq \H_\rho$ as $T$-representations and that $\H_\nu \cap \H_\rho^\omega$ is dense in $\H_\nu$ and killed by $d\rho(\n_- \oplus \a_-)$. Let $(\pi, \H_\pi)$ denote the unitary $H$-representation in $\H_\rho$ generated by $\H_\nu \cap \H_\rho^\omega \subseteq \H_\rho$. Using \Fref{thm: hol_induced_equiv_char} it follows that $\pi = \HolInd_T^H(\nu)$. By \Fref{thm: uniqueness}, it follows that $\pi \cong \sigma$ as unitary $H$-representations. Thus we may assume $\H_\sigma = \H_\pi \subseteq \H_\rho$, the last inclusion being $H$-equivariant. The $H$-orbit of $\H_\nu \cap \H_\rho^\omega$ under $\restr{\rho}{H}$ in $\H_\sigma$ is contained in $\H_\sigma \cap \H_\rho^\omega$ and is trivially total for $\H_\sigma$. Thus $\H_\sigma \cap \H_\rho^\omega$ is dense in $\H_\sigma$. As $\H_\nu \cap \H_\rho^\omega$ is already cyclic for $(\rho, \H_\rho)$, so is the larger space $\H_\sigma \cap \H_\rho^\omega$. To see that $\rho = \HolInd_H^G(\sigma)$, it just remains to show that $\H_\sigma \cap \H_\rho^\omega$ is killed by $d\rho(\n_-)$. As $\H_\nu \cap \H_\rho^\omega$ is killed by $d\rho(\n_-)$ and $\Ad_H(\n_-) \subseteq \n_-$, it follows that $d\rho(\n_-)\rho(H)\psi \subseteq \rho(H)d\rho(\n_-)\psi = \{0\}$ for any $\psi \in \H_\nu \cap \H_\rho^\omega$. Thus $d\rho(\n_-)$ kills $\rho(H)(\H_\nu \cap \H_\rho^\omega)$. As $\rho(H)(\H_\nu \cap \H_\rho^\omega)$ is total in $\H_\sigma$, it follows that $d\rho(\n_-)$ kills $\H_\sigma \cap \H_\rho^\omega$. Having shown all conditions of \Fref{thm: hol_induced_equiv_char}, we conclude that $\rho = \HolInd_H^G(\sigma)$.			
			\item As $\sigma = \HolInd_T^H(\nu)$ we may assume that $\H_\nu \subseteq \H_\sigma$ as $T$-representations and that $\H_\sigma^\omega \cap \H_\nu$ is dense in $\H_\nu$, cyclic for the $H$-representation $\H_\sigma$, and killed by $d\sigma(\a_-)$. Similarly, as $\rho = \HolInd_H^G(\sigma)$ we may assume that $\H_\sigma \subseteq \H_\rho$ as $H$ representations and moreover that $\H_\rho^\omega \cap \H_\sigma$ is dense in $\H_\sigma$, cyclic for the $G$-representation $\H_\rho$ and killed by $d\rho(\n_-)$. Then $\H_\nu \subseteq \H_\sigma \subseteq \H_\rho$ the inclusions being $T$- and $H$ equivariant, respectively. By assumption $\H_\nu \cap \H_\rho^\omega$ is dense in $\H_\nu$. Since $\H_\nu$ is cyclic for $(\sigma, \H_\sigma)$ and $\H_\sigma$ for $(\rho, \H_\rho)$, it follows that $\H_\nu \cap \H_\rho^\omega$ is cyclic for $(\rho, \H_\rho)$. For any $\psi \in \H_\nu \cap \H_\rho^\omega \subseteq \H_\sigma \cap \H_\rho^\omega$ we have $d\rho(\a_- \oplus \n_-)\psi \subseteq d\sigma(\a_-)\psi + d\rho(\n_-)\psi = \{0\}$. Thus $\H_\nu \cap \H_\rho^\omega$ is killed by $d\rho(\a_- \oplus \n_-)$. By \Fref{thm: hol_induced_equiv_char} it follows that $\rho = \HolInd_T^G(\nu)$.\qedhere
		\end{enumerate}
	\end{proof}
	
	\section{A geometric approach to holomorphic induction}\label{sec: geometric_hol_induction}

	\setcounter{subsection}{0}
	\setcounter{theorem}{0}
	
	\noindent
	In this section, a definition of holomorphically induced representations is presented which ensures that $\H_\rho^\infty$ embeds in a space of holomorphic mappings. Contrary to \Fref{sec: general_hol_induction}, this approach involves complex geometry. It is not as generally applicable, and in particular requires access to a dense set of b-strongly-entire vectors in the representation being induced, a condition that is well-understood for finite-dimensional Lie groups but barely studied for infinite-dimensional ones. We first clarify the precise setting, after which the homogeneous vector bundle $G \times_{H}V_\sigma^{\mcO_b}$ is equipped with a suitable complex-analytic bundle structure in \Fref{sec: complex_bdl_structure}. Here, we closely follow the construction of \citep[Thm.\ 2.6]{Neeb_hol_reps}. We then proceed in \Fref{sec: geom_hol_ind} to define geometric holomorphic induction and compare the notion with the one studied in \Fref{sec: general_hol_induction}.\\

	\noindent
	We consider the setting of \Fref{sec: general_hol_induction}, so $G$ denotes a connected BCH Fr\'echet--Lie group with Lie algebra $\g$ and $H \subseteq G$ is a closed and connected locally exponential Lie subgroup of $G$, in the sense of \citep[Def.\ IV.3.2]{neeb_towards_lie}. In this case, $H$ is necessarily BCH and an analytic embedded Lie subgroup of $G$ (cf.\ \Fref{lem: subgroup_embedded_bch}). Let $\h := \bm{\mrm{L}}(H)$ be the Lie algebra of $H$, which we identify as Lie subalgebra of $\g$ using the pushforward of the inclusion $H \hookrightarrow G$. We moreover assume given a decomposition $\g_\C = \n_- \oplus \h_\C \oplus \n_+$, where $\n_\pm$ and $\h_\C$ are closed Lie subalgebras of $\g_\C$ satisfying $\theta(\n_\pm) \subseteq \n_\mp$, $\theta(\h_\C) \subseteq \h_\C$ and $[\h_\C, \n_\pm] \subseteq \n_\pm$. We assume further that $G_\C$ is a connected complex BCH Fr\'echet--Lie group with Lie algebra $\bm{\mrm{L}}(G_\C) = \g_\C$, and the existence of an embedding $\eta : G \hookrightarrow G_\C$ with $\bm{\mrm{L}}(\eta) : \g \hookrightarrow \g_\C$ being the inclusion. Observe in this setting that $\Ad_H(\b_\pm) \subseteq \b_\pm$. Define the closed complement $\p := (\n_- \oplus \n_+) \cap \g$ of $\h$ in $\g$, so that $\p \cong \g/\h$. Let $M$ denote the homogeneous space $M := G/H$. \\

	\noindent
	Following \citep[Appendix C]{Neeb_semibounded_hilbert_loop}, we assume in addition that there exist open symmetric convex $0$-neighborhoods
	$$ U_\p \subseteq \p\cap U_\C, \quad U_\h \subseteq \h \cap U_\C, \quad U_{\n_\pm} \subseteq \n_\pm \cap U_\C \quad \text{ and } U_{\b_-} \subseteq {\b_-} \cap U_\C$$
	such that the following maps are analytic diffeomorphisms onto an open subset of their codomain, where $x \ast y$ is defined by the BCH series:
	\begin{alignat}{2}
		U_\p \times U_\h &\to \g,\quad &(x,y) &\mapsto x \ast y, \label{eq: condition_A1}\tag{A1}\\
		U_\p \times U_{\b_-} &\to \g_\C, &(x,y) &\mapsto x \ast y, \label{eq: condition_A2}\tag{A2}\\
		U_{\n_+} \times U_{\b_-} &\to \g_\C, &(x,y) &\mapsto x \ast y\label{eq: condition_A3}\tag{A3}.
	\end{alignat}
	\begin{remark}\label{rem: complex_str_hom_space}~
		\begin{enumerate}
			\item As mentioned in \citep[Appendix C]{Neeb_semibounded_hilbert_loop}, \eqref{eq: condition_A1} ensures that $M$ carries the structure of a real-analytic manifold turning $G \xrightarrow{q} G/H$ into a real-analytic principal $H$-bundle, and for which the left $G$-action is real-analytic $G \times M \to M$. The conditions \eqref{eq: condition_A2} and \eqref{eq: condition_A3} furthermore ensure that $M$ carries a compatible complex manifold structure for which the map $l_g : M \to M, xH \mapsto gxH$ is holomorphic for any $g \in G$, and for which the $\C$-linear extension of the map $\g \to T_{eH}(M), \; \xi \mapsto \restr{d\over dt}{t=0}q(e^{t\xi})$ descends to a $\C$-linear isomorphism $\g_\C/\b_- \cong T_{eH}(M)$. Condition \eqref{eq: condition_A3} is also needed to equip $G \times_H V_\sigma^{\O_b}$ with the structure of a complex vector bundle over $M$, as we shall see shortly.
			\item In \citep[Example C.4]{Neeb_semibounded_hilbert_loop}, sufficient conditions are discussed that guarantee that \eqref{eq: condition_A1}, \eqref{eq: condition_A2} and \eqref{eq: condition_A3} are satisfied. Using the Inverse Function Theorem, this is in particular the case if $G$ is a simply connected Banach--Lie group and $G/H$ is a Banach homogeneous space. In \citep[Sec.\ 5.2]{Neeb_semibounded_hilbert_loop}, these conditions are also shown to be satisfied in the context where $G$ is (a central $\T$-extensions of) a (twisted) loop group, and where $H\subseteq G$ is the subgroup of fixed points under a particular $\R$-action on $G$.
		\end{enumerate}		
	\end{remark}

	\noindent
	Henceforth, we endow $M = G/H$ with the $G$-invariant complex structure mentioned in \Fref{rem: complex_str_hom_space}, so that $T_{eH}(M) \cong \g_\C/\b_-$ as complex vector spaces.
	
	\begin{lemma}\label{lem: hol_fns_M}
		Let $f \in C^\infty(M; \C)$ and let $\widetilde{f} : G \to \C$ be its lift to $G$. Then $f \in \O(M) \iff \L_{\bm{v}(\b_-)}\widetilde{f} = \{0\}$. 
	\end{lemma}
	\begin{proof}
		Identify $(T_g G)_\C \cong \g_\C$ and $T_{gH}M \cong \g_\C/\b_-$ using the left $G$-action on $M$. By \Fref{prop: complex_analytic_and_complex_linear_differential}, $f$ is holomorphic if and only if $T(f)$ is fiberwise $\C$-linear, which in turn is equivalent to $\L_{\bm{v}(\b_-)}\widetilde{f} = \{0\}$.
	\end{proof}

	\subsection{Complex structures on $\mbb{E}_\sigma = G \times_H V_\sigma^{\mc{O}_b}$}\label{sec: complex_bdl_structure}
	
	\noindent
	Fix a unitary $H$-representation $(\sigma, V_\sigma)$. Recall from \Fref{def: smooth_analytic_entire_vectors} that $V_\sigma^{\mcO_b}$ denotes the space of b-strongly-entire vectors for the $H$-representation $\sigma$ on $V_\sigma$. Assume that the $H$-action $\sigma$ on $V_\sigma^{\mc{O}_b}$ is real-analytic $H \times V_\sigma^{\mc{O}_b} \to V_\sigma^{\mc{O}_b}$. Since $G \to G/H$ is a real-analytic principal $H$-bundle, we have in this case that the $G$-homogeneous vector bundle $\mbb{E}_\sigma := G \times_H V_\sigma^{\mcO_b}$ over $G/H$ with typical fiber is $V_\sigma^{\mcO_b}$ carries a natural real-analytic bundle structure.\\
	
	\begin{remark}
		If $H$ is actually a Banach--Lie group, so that $\h$ is a Banach space w.r.t.\ the subspace topology inherited from $\g$, then the $H$-action on $V_\sigma^{\mc{O}_b}$ is always real-analytic $H \times V_\sigma^{\mc{O}_b} \to V_\sigma^{\mc{O}_b}$, cf.\ \Fref{rem: b-strongly-entire_2}\\
	\end{remark}

	\noindent
	In the following, we adapt the proof of \citep[Thm.\ 2.6]{Neeb_hol_reps} to endow $\mbb{E}_\sigma$ with a complex-analytic bundle structure, using the notion of entire extensions $\chi : \b_- \to \B(V_\sigma^{\mcO_b})$ of $d\sigma$ to $\b_-$, see \Fref{def: extension_geometric} below. We write $L_g : \mbb{E}_\sigma \to \mbb{E}_\sigma$ for the left $G$-action on $\mbb{E}_\sigma$.
	
	\begin{definition}\label{def: complex_analytic_maps_into_operators}
		Let $W$ be a complete Hausdorff complex (resp.\ real) locally convex vector space and let $F : W \to \B(V_\sigma^{\mcO_b})$ be a function. We say that $F$ is complex-analytic (resp.\ real-analytic, smooth) if the corresponding map 
		$$F^\vee : W~\times~V_\sigma^{\mcO_b}~\to~V_\sigma^{\mcO_b}$$
		is complex-analytic (resp.\ real-analytic, smooth).
	\end{definition}
	\begin{lemma}\label{lem: map_into_operators_hol}
		Consider the setting of \Fref{def: complex_analytic_maps_into_operators}. If $F$ is smooth then $F$ is complex-analytic if and only if the map $T_x(F) : W \to \B(V_\sigma^{\mcO_b})$ is $\C$-linear for every $x \in W$, where $T_x(F)(w)v := \restr{d\over dt}{t=0}F(x + tw)v$ 
	\end{lemma}
	\begin{proof}
		By \Fref{prop: complex_analytic_and_complex_linear_differential}, the map $F^\vee : W \times V_\sigma^{\mcO_b} \to V_\sigma^{\mcO_b}$ is complex-analytic if and only if it is smooth and $T(F^\vee)$ is fiber-wise $\C$-linear. $F^\vee$ is smooth by assumption and $v \mapsto T_x(F^\vee)(0, v)$ is trivially $\C$-linear. Thus $F$ is complex-analytic if and only if $w \mapsto T_x(F^\vee)(w, 0)$ is $\C$-linear for any $x,w \in W$, which is the statement.\qedhere \\
	\end{proof}

	\begin{definition}\label{def: extension_geometric}
		An \textit{entire extension} $\chi$ of $d\sigma : \h_\C \to \B(V_\sigma^{\mcO_b})$ to $\b_-$ is a homomorphism $\chi : \b_- \to \B(V_\sigma^{\mcO_b})$ of Lie algebras such that:
		\begin{enumerate}
			\item $\restr{\chi}{\h_\C} = d\sigma$.
			\item $\chi(\Ad_h(\xi)) = \sigma(h)\chi(\xi)\sigma(h)^{-1}$ for all $h \in H$ and $\xi \in \b_-$.
			\item The series $\sum_{n=0}^\infty {1\over n!}\chi(\xi)^nv$ converges in $V_\sigma^{\mcO_b}$ for all $\xi \in \b_-$ and $v \in V_\sigma^{\mcO_b}$, and defines an entire map 
			$$\b_- \times V_\sigma^{\mcO_b} \to V_\sigma^{\mcO_b}, \qquad (\xi, v) \mapsto \sum_{n=0}^\infty {1\over n!}\chi(\xi)^nv.$$
		\end{enumerate} 
		In this case, we write $e^{\chi(\xi)}\psi := \sum_{n=0}^\infty{1\over n!}\chi(\xi)^n\psi$. 
	\end{definition}	
	
	\noindent
	Notice that an entire extension is in particular an extension in the sense of \Fref{def: extension}. \\
	
	\noindent
	The following example and the discussion in \Fref{rem: b-strongly-entire_2} reflect the main reason for working with the space $V_\sigma^{\mcO_b}$ of $b$-strongly-entire vectors, instead of $V_\sigma^\mcO$:
	\begin{example}\label{ex: banach_trivial_extension_entire}
		Assume that $H$ is a Banach--Lie group. In view of \Fref{rem: b-strongly-entire_2}, we then know that the following map is entire:
		\begin{equation}\label{eq: b-entire-extension_bundles}
			\h_\C \times V_\sigma^{\mcO_b} \to V_\sigma^{\mcO_b}, \qquad (\eta, v) \mapsto \sum_{n=0}^\infty {1\over n!}d\sigma(\eta^n)v
		\end{equation}
		Consequently, the trivial extension $\chi : \b_- \to \B(V_\sigma^{\mcO_b})$ of $d\sigma$ to $\b_-$ with domain $V_\sigma^{\mcO_b}$ is an entire extension.\\
	\end{example}

	\noindent
	Using the notion of entire extensions, \citep[Thm.\ 2.6]{Neeb_hol_reps} adapts straightforwardly to the present setting:
	
	\begin{theorem}\label{thm: complex_bundle_structure}
		Let $\chi : \b_- \to \B(V_\sigma^{\mcO_b})$ be an entire extension of $d\sigma$ to $\b_-$. Then $\mbb{E}_\sigma = G \times_H V_\sigma^{\mcO_b}$ carries a unique complex-analytic bundle structure satisfying the following properties:
		\begin{enumerate}
			\item The left $G$-action $L_g$ is complex-analytic for any $g \in G$. 
			\item The quotient map $G \times V_\sigma \to \mbb{E}_\sigma$ is real-analytic.
			\item Let $U \subseteq G$ be a neighborhood of $g \in G$. A smooth function $f \in C^\infty(UH, V_\sigma^{\mcO_b})^H$ corresponds to a local holomorphic section of $\mbb{E}_\sigma$ if and only if 
			$$\L_{\bm{v}(\xi)}f = -\chi(\xi)f, \qquad \forall \xi \in \n_-.$$
		\end{enumerate}
		If the two entire extensions $\chi_1$ and $\chi_2$ of $d\sigma$ to $\b_-$ define the same complex-bundle structure, then $\chi_1 = \chi_2$.
	\end{theorem}
	
	\begin{definition}
		Let $\chi : \b_- \to \B(V_\sigma^{\mcO_b})$ be an entire extension of $d\sigma$ to $\b_-$. We denote by $\mbb{E}_{(\sigma, \chi)} \to M$ the vector bundle $\mbb{E}_\sigma \to M$ equipped with the unique complex-analytic bundle structure satisfying the conditions in \Fref{thm: complex_bundle_structure}.
	\end{definition}
	
	\begin{proof}[Proof of \Fref{thm: complex_bundle_structure}:]
		This proof essentially follows from trivial adaptations of \citep[Thm.\ 2.6]{Neeb_hol_reps}. Let us indicate the required changes and recall the construction of the local charts, for later use. \\

		\noindent
		Let $q_M : G \to G/H$ denote the quotient map. Let $U_\g \subseteq \g$ and $U_G \subseteq G$ be neighborhoods of $0 \in \g$ and $e \in G$, respectively, s.t.\ $\restr{\exp_G}{U_\g} : U_\g \to U_G$ is an analytic diffeomorphism. Shrinking $U_\g$ if necessary, there exists by \eqref{eq: condition_A1} some $0$-neighborhoods $U_\p \subseteq \p$ and $U_\h \subseteq \h$ s.t.\ the BCH series defines an analytic diffeomorphism $U_\p \times U_\h \to U_\g$. Define $U_P := \exp_G(U_\p)$ and $U_H := \exp_G(U_\h)$, so $U_G = U_P U_H$. (Comparing with the proof of \citep[Thm.\ 2.6]{Neeb_hol_reps}, $U_P$ takes the role of $U_Z$.) Define for any $x \in G$ the open subsets 
		\begin{equation}\label{eq: complex_bdl_str_open_sets}
			U_x := xq_M(U_P)\subseteq M \qquad \text{ and } \qquad \widetilde{U}_x := xU_PH \subseteq G.	
		\end{equation}
		Using \eqref{eq: condition_A3}, and replacing $\beta : \b_- \to \B(V_\sigma)$ in steps $2-4$ of the proof of \citep[Thm.\ 2.6]{Neeb_hol_reps} by the entire extension $\chi : \b_- \to \B(V_\sigma^{\mcO_b})$, we obtain after shrinking $U_\g$ if necessary for each $x \in G$ a smooth function $F_x :  \widetilde{U}_x \to \B(V_\sigma^{\mcO_b})^\times$ satisfying the following properties:
		\begin{enumerate}
			\item $F_x(gh) = \sigma(h)^{-1}F_x(g)$ for all $g \in \widetilde{U}_x$ and $h \in H$.
			\item $\L_{\bm{v}(\xi)}F_x = -\chi(\xi)F_x$ for all $\xi \in \b_-$.
			\item $F_x(x) = \id_{V_\sigma^{\mcO_b}}$.
		\end{enumerate}
		Moreover, using \eqref{eq: condition_A1}, \eqref{eq: condition_A3}, that $e^\chi : \b_- \to \B(V_\sigma^{\mcO_b})$ is complex-analytic and that the action $H \times V_\sigma^{\mcO_b} \to V_\sigma^{\mcO_b}$ is real-analytic, in view of \Fref{rem: b-strongly-entire_2}, observe from its construction that both $F_x$ and $F_x^{-1}$ are actually real-analytic. These functions moreover satisfy $F_{yx}(yg) = F_x(g)$ for any $x,y \in G$ and $g \in \widetilde{U}_x$, as is immediate from their construction. As in \citep[Thm.\ 2.6]{Neeb_hol_reps}, we now define for each $x \in G$ the trivialization
		\begin{equation}\label{eq: chart_cplx_hom_bundle}
			\phi_{x} : U_x \times V_\sigma^{\mcO_b} \to \restr{\mbb{E}_\sigma}{U_x}, \qquad (gH, v) \mapsto [g, F_x(g)v],
		\end{equation}
		so that the transition function $\phi_x^{-1} \circ \phi_y : U_x \cap U_y \times V_\sigma^{\mcO_b} \to U_x \cap U_y \times V_\sigma^{\mcO_b}$ is given by $(gH, v) \mapsto (gH, \phi_{xy}(gH)v)$, where 
		$$ \phi_{xy} : U_x \cap U_y \to \B(V_\sigma^{\mcO_b})^\times, \quad \phi_{xy}(gH) = F_x(g)^{-1}F_y(g).  $$
		Let us check as in \citep[Thm.\ 2.6]{Neeb_hol_reps} that these transition functions are complex-analytic. It suffices by \Fref{lem: map_into_operators_hol} to show that the lift $\widetilde{\phi}_{xy} : \widetilde{U}_x\cap \widetilde{U}_y \to \B(V_\sigma^{\mcO_b})$ of $\phi_{xy}$ satisfies $\L_{\bm{v}(\xi)}\widetilde{\phi}_{xy} = 0$ for all $\xi \in \b_-$. This follows from the three properties of the functions $F_x$ mentioned above:
		\begin{align*}
			\L_{\bm{v}(\xi)}\widetilde{\phi}_{xy} 
			&= \big(\L_{\bm{v}(\xi)}F_x^{-1}\big)F_y + F_x^{-1}\big(\L_{\bm{v}(\xi)}F_y\big)\\
			&= -F_x^{-1}\big(\L_{\bm{v}(\xi)}F_x)F_x^{-1}F_y + F_x^{-1}\big(\L_{\bm{v}(\xi)}F_y\big)\\
			&= F_x^{-1}\chi(\xi)F_y - F_x^{-1}\chi(\xi)F_y \\
			&= 0.
		\end{align*}
		Thus the trivializations $\{\phi_x\}_{x\in G}$ define a complex-analytic bundle structure on $\mbb{E}_\sigma$. Let $\mbb{E}_{(\sigma, \chi)}$ denote the corresponding complex-analytic bundle. We show that the properties $1-3$ in \Fref{thm: complex_bundle_structure} are satisfied:
		\begin{enumerate}
			\item Let $x,g \in G$. In the local charts defined by $\phi_x$ and $\phi_{gx}$, $L_g$ is represented by $l_g \times \id_{V_\sigma^{\mcO_b}}$, which is complex-analytic from the corresponding property of $l_g : M \to M$.
			\item Let $x \in G$. Consider the local coordinates of $\mbb{E}_{(\sigma, \chi)}$ defined by $\phi_x$. In these local coordinates, the quotient map $G \times V_\sigma^{\mcO_b} \to \mbb{E}_{(\sigma, \chi)}$ is represented by the real-analytic function
			$$ \widetilde{U}_x \times V_\sigma^{\mcO_b} \to U_x \times V_\sigma^{\mcO_b}, \qquad (g,v) \mapsto (gH, F_x(g)^{-1}v).$$
			\item Take $f \in C^\infty(UH, V_\sigma^{\mcO_b})^H$. The corresponding local section of $\mbb{E}_{(\sigma, \chi)}$ is obtained by descending the function $\widetilde{f} : UH \to \mbb{E}_\sigma, \widetilde{f}(g) := [g, f(g)]$ to the quotient $q_M(U)$. Let $x \in U$ and define $W_x:= U_x \cap U$ and $\widetilde{W}_x := \widetilde{U}_x \cap UH$. Using the local chart $\phi_x$, the map $\restr{\widetilde{f}}{\widetilde{W}_x}$ is represented by the smooth function 
			$$\overline{f} : \widetilde{W}_x \to U_x \times V_\sigma^{\mcO_b}, \qquad \overline{f}(g) = (gH, F_x(g)^{-1}f(g)),$$
			which is complex-analytic if and only if $\L_{\bm{v}(\xi)}h = 0$ for any $\xi \in \b_-$, where $h$ is given by
			$$h : \widetilde{W}_x \to V_\sigma^{\mcO_b}, \qquad h(g):= F_x(g)^{-1}f(g).$$
			We compute that
			\begin{align*}
				\L_{\bm{v}(\xi)}h
				&= \big(\L_{\bm{v}(\xi)}F_x^{-1}\big)f + F_x^{-1}\big(\L_{\bm{v}(\xi)}f\big)\\
				&= -F_x^{-1}\big(\L_{\bm{v}(\xi)}F_x)F_x^{-1}f + F_x^{-1}\big(\L_{\bm{v}(\xi)}f\big)\\
				&= F_x^{-1}\chi(\xi)f + F_x^{-1}\big(\L_{\bm{v}(\xi)}f\big).
			\end{align*}
			Thus $\L_{\bm{v}(\xi)}h = 0$ if and only if $\L_{\bm{v}(\xi)}f = -\chi(\xi)f$ for any $\xi \in \b_-$. Consequently $f$ corresponds to a holomorphic local section of $\mbb{E}_\sigma \to M$ if and only if $\L_{\bm{v}(\xi)}f = -\chi(\xi)f$ for every $\xi \in \b_-$. The equation is automatically satisfied for any $\xi \in \h_\C$ by the $H$-equivariance of $f$. The conclusion follows. 
		\end{enumerate}
		
		\noindent
		Step 5 in \citep[Thm.\ 2.6]{Neeb_hol_reps} shows that if the two entire extensions $\chi_1$ and $\chi_2$ define the same complex bundle structure, then $\chi_1 = \chi_2$. To see that the complex-bundle structure is unique, we simply remark that if $\mbb{E}_\sigma^1$ and $\mbb{E}_\sigma^2$ denote the vector bundle $\mbb{E}_\sigma$ equipped \`a priori with possibly different complex-analytic bundle structures satisfying the properties $1-3$ in \Fref{thm: complex_bundle_structure}, then by the third property they have the same holomorphic local sections. This implies $\mbb{E}_\sigma^1 = \mbb{E}_\sigma^2$ as complex-analytic vector bundles over $M$.
	\end{proof}
	
	\subsection{Geometric holomorphic induction}\label{sec: geom_hol_ind}
	
	\noindent
	Having the complex-analytic $G$-homogeneous vector bundles $\mbb{E}_{(\sigma, \chi)}$ at hand, we are now in a position to define a stronger notion of holomorphic induction, which guarantees that $\H_\rho^\infty$ actually embeds into a space of holomorphic mappings. We continue under the assumptions of \Fref{sec: complex_bdl_structure}. In particular, $\sigma$ is a unitary representation of $H$ on $V_\sigma$ for which the $H$-action on $V_\sigma^{\mc{O}_b}$ is jointly real-analytic $H \times V_\sigma^{\mc{O}_b} \to V_\sigma^{\mc{O}_b}$. Let $\chi : \b_- \to \B(V_\sigma^{\mcO_b})$ be an entire extension of $d\sigma : \h_\C \to \B(V_\sigma^{\mcO_b})$ to $\b_-$, and let $(\rho, \H_\rho)$ be a unitary $G$-representation.
	
	\begin{definition}\label{def: geometric_hol_induced}
		We say that $(\rho, \H_\rho)$ is \textit{geometrically holomorphically induced} from $(\sigma, \chi)$ if $\sigma$ is b-strongly-entire and there exists a $G$-equivariant injective linear map $\Phi : \H_\rho \hookrightarrow \Map(G, V_\sigma)^H$ satisfying:
		\begin{enumerate}
			\item The point evaluation $\mc{E}_x : \H_\rho \to V_\sigma, \; \mc{E}_x(\psi) := \Phi_\psi(x)$ is continuous for every $x \in G$.
			\item $\mc{E}_x\mc{E}_x^\ast = \id_{V_\sigma}$ for every $x \in G$.
			\item For every $w \in V_\sigma^{\mcO_b}$, the following function is holomorphic:
			$$f_w : \mathbb{E}_{(\sigma, \chi)} \to \C, \qquad f_w([g, v]) := \langle \mc{E}_e^\ast w, \rho(g) \mc{E}_e^\ast v\rangle.$$
		\end{enumerate}
	\end{definition}
	
	\begin{remark}
		The first condition in \Fref{def: geometric_hol_induced} entails that $(\rho, \H_\rho)$ is unitarily equivalent to the natural $G$-representation on the reproducing kernel Hilbert space $\H_Q$, where $Q \in C(G \times G, \B(V_\sigma))^{H \times H}$ is the positive definite and $G$-invariant kernel defined by $Q(x,y) := \mc{E}_x\mc{E}_y^\ast$. Combined with the second property, we additionally have that $\mc{E}_e^\ast$ is an $H$-equivariant isometry and that the subspace $\mc{E}_e^\ast V_\sigma \subseteq H_\rho$ is cyclic for $G$, cf.\ \Fref{thm: properties_repr_kernel} and \Fref{prop: homog_repr_h_space_unitary_repr} below.\\
	\end{remark}
	
	\noindent
	We start with a lemma:
	
	\begin{lemma}\label{lem: functions_on_bundle_holomorphic}
		Assume that $V_\sigma \subseteq \H_\rho$ as unitary $H$-representations and that $V_\sigma$ is cyclic for $G$ in $\H_\rho$. Assume further that $\sigma$ is b-strongly-entire. Then the following assertions are equivalent:
		\begin{enumerate}
			\item $V_\sigma^{\mcO_b} \subseteq \H_\rho^\infty$ and $d\rho(\xi)v = \chi(\xi)v$ for all $\xi \in \b_-$ and $v \in V_\sigma^{\mcO_b}$.
			\item $f_w \in \mcO(\mathbb{E}_{(\sigma, \chi)})$ for every $w \in V_\sigma^{\mcO_b}$, where $f_w([g,v]) := \langle w, \rho(g)v\rangle$.
		\end{enumerate}
		If these assertions are satisfied, then we even have $V_\sigma^{\mcO_b} \subseteq \H_\rho^\omega$. Moreover, we have $f_\psi \in \mcO(\mathbb{E}_{(\sigma, \chi)})$ for any $\psi \in \H_\rho^\infty$, where $f_\psi([g, v]) := \langle \psi, \rho(g)v\rangle$.
	\end{lemma}
	\begin{proof}
		Let $\psi \in \H_\rho^\infty$ and consider the function 
		$$f_\psi : \mathbb{E}_{(\sigma, \chi)} \to \C, \qquad f_\psi([g,v]) = \langle \psi, \rho(g)v\rangle.$$
		Consider its lift to $G \times V_\sigma^{\mcO_b}$, defined by $\widetilde{f}_\psi : G \times V_\sigma^{\mcO_b} \to \C, \;\widetilde{f}_\psi(g,v) = f_\psi([g,v])$. Let $x \in G$. Define the open sets $\widetilde{U}_x \subseteq G$ and $U_x\subseteq M$ as in \eqref{eq: complex_bdl_str_open_sets}, so $U_x$ is an open neighborhood of $xH \in M$. Let $F_x : \widetilde{U}_x \to \B(V_\sigma^{\mcO_b})^\times$ be defined as in the proof of \Fref{thm: complex_bundle_structure}. In particular, $F_x$ satisfies $\L_{\bm{v}(\xi)}F_x = -\chi(\xi)F_x$ for any $\xi \in \b_-$. Let $\phi_x : U_x \times V_\sigma^{\mcO_b} \to \restr{\mathbb{E}_{(\sigma, \chi)}}{U_x}$ be the corresponding chart of the holomorphic vector bundle $\mathbb{E}_{(\sigma, \chi)}$, defined in \eqref{eq: chart_cplx_hom_bundle}. In these local coordinates, $f_\psi$ and $\widetilde{f}_\psi$ are represented by $h_{\psi, x}$ and $\widetilde{h}_{\psi, x}$, respectively, where
		\begin{alignat*}{2}
			h_{\psi, x} : U_x \times V_\sigma^{\mcO_b} &\to \C, &\qquad h_{\psi, x}(gH, v) &= \langle \rho(g)^{-1}\psi, F_x(g)v\rangle,\\
			\widetilde{h}_{\psi, x} : \widetilde{U}_x \times V_\sigma^{\mcO_b} &\to \C, &\qquad \widetilde{h}_{\psi, x}(g,v) &= \langle \rho(g)^{-1}\psi, F_x(g)v\rangle.	
		\end{alignat*}
		As $F_x$ is smooth, and because $\psi \in \H_\rho^\infty$, this shows in particular that $f_\psi$ is smooth for the underlying real manifold structure. Then $h_{\psi, x}$ is complex-analytic if and only if $\L_{\bm{v}(\xi)}\widetilde{h}_{\psi, x} = 0$ for any $\xi \in \b_-$. Let $\xi \in \b_-$. Using $\L_{\bm{v}(\xi)}F_x = -\chi(\xi)F_x$, we compute for any $(g,v) \in \widetilde{U}_x \times V_\sigma^{\mcO_b}$ that
		\begin{align}\label{eq: computation_local_hol_condition}
			(\L_{\bm{v}(\xi)}\widetilde{h}_{\psi, x})(g,v) &= \langle d\rho(\xi^\ast)\rho(g)^{-1}\psi, F_x(g)v\rangle - \langle \rho(g)^{-1}\psi, \chi(\xi)F_x(g)v\rangle.
		\end{align}
		Thus if $(1)$ holds true, then \eqref{eq: computation_local_hol_condition} shows that $\L_{\bm{v}(\xi)}\widetilde{h}_{\psi, x} = 0$ for any $\xi \in \b_-$, so that $h_{\psi, x}$ is complex-analytic for any $x \in G$. We then conclude that $f_\psi \in \mcO(\mathbb{E}_{(\sigma, \chi)})$ for any $\psi \in \H_\rho^\infty$. Since $V_\sigma^{\mcO_b} \subseteq \H_\rho^\infty$ by assumption, we in particular notice that $(2)$ holds true. \\
		
		\noindent
		Assume conversely that $f_w \in \mcO(\mathbb{E}_{(\sigma, \chi)})$ for any $w \in V_\sigma^{\mcO_b}$. Let $v \in V_\sigma^{\mcO_b}$. Then the function $g \mapsto \langle v, \rho(g)v\rangle = f_v([g, v])$ is real-analytic $G \to \C$, where we have used that the quotient map $G \times V_\sigma^{\mcO_b} \to \mathbb{E}_{(\sigma, \chi)}$ is real-analytic. As $G$ is a BCH Fr\'echet--Lie group, this implies by \citep[Thm.\ 5.2]{Neeb_analytic_vectors} that $v \in \H_\rho^\omega$. Hence $V_\sigma^{\mcO_b} \subseteq \H_\rho^\omega$. We know from \Fref{cor: entire_vectors_seminorm_finite} that $\H_\rho^\omega \subseteq \H_\rho^\infty$, so we also obtain that $V_\sigma^{\mcO_b} \subseteq \H_\rho^\infty$. To see that $(1)$ holds true, it remains to show that $d\rho(\xi)v = \chi(\xi)v$ for all $\xi \in \b_-$ and $v \in V_\sigma^{\mcO_b}$. Consider the set $\mcD := \set{\psi \in \H_\rho^\infty \st f_\psi \in \mcO(\mathbb{E}_{(\sigma, \chi)})}$. The preceding shows that $V_\sigma^{\mcO_b} \subseteq \mcD$. The set $\mcD$ is moreover $G$-invariant. Indeed, if $\psi \in \mcD$ and $g \in G$, then $f_{\rho(g)\psi} = f_\psi \circ L_{g^{-1}}$ defines a holomorphic map on $\mathbb{E}_{(\sigma, \chi)}$, because $L_{g^{-1}} : \mathbb{E}_{(\sigma, \chi)} \to \mathbb{E}_{(\sigma, \chi)}$ is holomorphic. As $V_\sigma^{\mcO_b}$ is dense in $V_\sigma$ and $V_\sigma$ is cyclic for $G$, it follows that $\mcD$ is dense in $\H_\rho$. Let $\xi \in \b_-$, $\psi \in \mcD$ and $v \in V_\sigma^{\mcO_b}$. Recall that $F_e(e) = \id_{V_\sigma^{\mcO_b}}$. As $f_\psi \in \mcO(\mathbb{E}_{(\sigma, \chi)})$, we know that $(\L_{\bm{v}(\xi)}\widetilde{h}_{\psi, e})(e) = 0$. Using that $v \in \H_\rho^\infty$, it follows by evaluating \eqref{eq: computation_local_hol_condition} at $(e,v) \in \widetilde{U}_e \times V_\sigma^{\mcO_b}$ that 
		$$ \langle \psi, d\rho(\xi)v\rangle = \langle d\rho(\xi^\ast) \psi, v\rangle = \langle \psi, \chi(\xi)v\rangle.$$
		As $\mcD$ is dense, it follows that $d\rho(\xi)v = \chi(\xi)v$ for all $\xi \in \b_-$ and $v \in V_\sigma^{\mcO_b}$, so that $(1)$ holds true. We have also shown that if these equivalent conditions are satisfied, then $V_\sigma^{\mcO_b} \subseteq \H_\rho^\omega$ and $f_\psi \in \mcO(\mathbb{E}_{(\sigma, \chi)})$ for any $\psi \in \H_\rho^\infty$.
	\end{proof}
	
	\noindent
	The following entails that $\H_\rho^\infty$ can be seen as a space of of holomorphic functions on the complex-analytic bundle $\overline{\mathbb{E}}_{(\sigma, \chi)} \to \overline{M}$ conjugate to $\mathbb{E}_{(\sigma, \chi)} \to M$:
	\begin{proposition}\label{prop: smooth_vectors_embed_into_hol_mappings}
		Assume that $\rho$ is geometrically holomorphically induced from $(\sigma, \chi)$. Then there is an injective $G$-equivariant $\C$-linear map $\H_\rho^\infty \hookrightarrow \mcO(\overline{\mathbb{E}}_{(\sigma, \chi)})$ for which all point evaluations are continuous.
	\end{proposition}
	\begin{proof}
		Assume that $\rho$ is geometrically holomorphically induced from $(\sigma, \chi)$. In particular, this implies that $\sigma$ is b-strongly-entire. Let $\Phi : \H_\rho \hookrightarrow \Map(G, V_\sigma)^H$ satisfy the conditions in \Fref{def: geometric_hol_induced}. We may consider $V_\sigma$ as a subspace of $\H_\rho$ using the $H$-equivariant isometry $\mc{E}_e^\ast$. We know by \Fref{thm: properties_repr_kernel} that $V_\sigma \subseteq \H_\rho$ is cyclic. From \Fref{lem: functions_on_bundle_holomorphic} we obtain that $f_\psi \in \mcO(\mathbb{E}_{(\sigma, \chi)})$ for any $\psi \in \H_\rho^\infty$. The map $\psi \mapsto f_\psi$ defines a $G$-equivariant $\C$-linear map $\H_\rho^\infty \to \overline{\mcO(\mathbb{E}_{(\sigma, \chi)})}$ that has continuous point evaluations, where $\overline{\mcO(\mathbb{E}_{(\sigma, \chi)})}$ denotes the vector space complex conjugate to $\mcO(\mathbb{E}_{(\sigma, \chi)})$, which may be identified with $\mcO(\overline{\mathbb{E}}_{(\sigma, \chi)})$. This map is injective because $f_\psi = 0$ implies that $\psi \perp \rho(G)V_\sigma^{\mcO_b}$, which in turn implies $\psi = 0$ because $V_\sigma^{\mcO_b}$ is cyclic for $\H_\rho$.
	\end{proof}

	\noindent
	Let us next compare the notion of geometric holomorphic induction with \Fref{def: hol_induced}. Recall that $\chi$ is an entire extension of $d\sigma : \h_\C \to \B(V_\sigma^{\mcO_b})$ to $\b_-$.
	
	\begin{theorem}\label{thm: comparing_two_notions_hol_ind}
		Assume that $\sigma$ is b-strongly-entire. The following assertions are equivalent:
		\begin{enumerate}
			\item $(\rho, \H_\rho)$ is geometrically holomorphically induced from $(\sigma, \chi)$. 
			\item There is a subspace $\mcD_{\widetilde{\chi}} \subseteq V_\sigma^\omega$ containing $V_\sigma^{\mcO_b}$ and an extension $\widetilde{\chi}~:~\b_-~\to~\L(\mcD_{\widetilde{\chi}})$ of $d\sigma$ to $\b_-$ such that $\chi(\xi) = \restr{\widetilde{\chi}(\xi)}{V_\sigma^{\mcO_b}}$ for every $\xi \in \b_-$, and such that $\rho = \HolInd_H^G(\sigma, \widetilde{\chi})$.
		\end{enumerate}
		Suppose that $\chi$ is the trivial extension of $d\sigma$ to $\b_-$ with domain $V_\sigma^{\mcO_b}$. Then these assertions are equivalent to:
		\begin{enumerate}\setcounter{enumi}{2}
			\item $\rho = \HolInd_H^G(\sigma)$ and $V_\sigma^{\mcO_b} \subseteq \H_\rho^\infty$, where we considered $V_\sigma$ as a subspace of $\H_\rho$ using \Fref{thm: hol_induced_equiv_char}.
		\end{enumerate}
	\end{theorem}
	\begin{proof}
		Assume that $(\rho, \H_\rho)$ is geometrically holomorphically induced from $(\sigma, \chi)$. Let $\Phi : \H_\rho \hookrightarrow \Map(G, V_\sigma)^H$ satisfy the conditions in \Fref{def: geometric_hol_induced}. Identify $V_\sigma$ with a cyclic subspace of $\H_\rho$ using $\mc{E}_e^\ast$. Define $\mcD_{\widetilde{\chi}} := V_\sigma \cap \H_\rho^\omega$. From \Fref{lem: functions_on_bundle_holomorphic} we obtain that $V_\sigma^{\mcO_b} \subseteq \mcD_{\widetilde{\chi}}$ and that $d\rho(\xi)v = \chi(\xi)v$ for all $\xi \in \b_-$ and $v \in V_\sigma^{\mcO_b}$. As $V_\sigma^{\mcO_b}$ is dense in $V_\sigma$, the latter in particular implies that $d\rho(\b_-)\mcD_{\widetilde{\chi}} \subseteq V_\sigma$, which in turn implies $d\rho(\b_-)\mcD_{\widetilde{\chi}} \subseteq \mcD_{\widetilde{\chi}}$. From \Fref{thm: hol_induced_equiv_char} it follows that $\rho = \HolInd_H^G(\sigma, \widetilde{\chi})$, where $\widetilde{\chi} : \b_- \to \L(\mcD_{\widetilde{\chi}})$ is the extension of $d\sigma$ to $\b_-$ with domain $\mcD_{\widetilde{\chi}}$, defined by $\widetilde{\chi}(\xi)v = d\rho(\xi)v$. This extension satisfies $\restr{\widetilde{\chi}(\xi)}{V_\sigma^{\mcO_b}} = \chi(\xi)$ for any $\xi \in \b_-$, as required. \\
		
		\noindent
		Conversely, let $\widetilde{\chi} : \b_- \to \L(\mcD_{\widetilde{\chi}})$ satisfy the conditions in $(2)$, so in particular $\rho = \HolInd_H^G(\sigma, \widetilde{\chi})$. By \Fref{thm: hol_induced_equiv_char} we may assume that $V_\sigma \subseteq \H_\rho$ as unitary $H$-representations, that $\mcD_{\widetilde{\chi}} = V_\sigma \cap \H_\rho^\omega$ and that $\widetilde{\chi}(\xi)v = d\rho(\xi)v$ for all $\xi \in \b_-$ and $v \in \mcD_{\widetilde{\chi}}$. As $\mcD_{\widetilde{\chi}}$ contains $V_\sigma^{\mcO_b}$ by assumption, it follows in particular that $V_\sigma^{\mcO_b} \subseteq \H_\rho^\omega$. From \Fref{lem: functions_on_bundle_holomorphic}, we obtain that $f_w \in \mcO(\mbb{E}_{(\sigma, \chi)})$ for any $w \in V_\sigma^{\mcO_b}$, where $f_w([g,v]) = \langle w, \rho(g)v\rangle$. So the map
		$$\Phi : \H_\rho \to \Map(G, V_\sigma)^H, \quad \Phi_\psi(g) := p_V \rho(g)^{-1}\psi$$
		satisfies the conditions in \Fref{def: geometric_hol_induced}, where $p_V : \H_\rho \to V_\sigma$ is the orthogonal projection. \\
		
		\noindent
		Assume that $\chi$ is the trivial extension of $d\sigma$ to $\b_-$ with domain $V_\sigma^{\mcO_b}$. Assume that $(2)$ holds true. Let the subspace $\mcD_{\widetilde{\chi}} \subseteq V_\sigma$ and the extension $\widetilde{\chi} : \b_- \to \L(\mcD_{\widetilde{\chi}})$ satisfy the conditions in $(2)$. We may consider $V_\sigma$ as a closed $H$-invariant linear subspace of $\H_\rho$ satisfying the conditions in \Fref{thm: hol_induced_equiv_char}. In particular, we have $V_\sigma^{\mcO_b} \subseteq \mcD_{\widetilde{\chi}} = V_\sigma \cap \H_\rho^\omega$, so certainly $V_\sigma^{\mcO_b} \subseteq V_\sigma^\infty$. As $\chi$ is the trivial extension on $V_\sigma^{\mcO_b}$ and $\chi(\xi) = \restr{\widetilde{\chi}(\xi)}{V_\sigma^{\mcO_b}}$ for every $\xi \in \b_-$, we also have $d\rho(\n_-)V_\sigma^{\mcO_b} = \{0\}$. As $V_\sigma^{\mcO_b}$ is dense in $V_\sigma$, this further implies that $d\rho(\n_-)\mcD_{\widetilde{\chi}} = \{0\}$, so $\widetilde{\chi}$ is the trivial extension on $\mcD_{\widetilde{\chi}}$. Hence $(3)$ holds true. Assume conversely that $(3)$ is valid. Let $\widetilde{\chi}$ denote the trivial extension of $d\sigma$ to $\b_-$ on the domain $\mcD_{\widetilde{\chi}} := V_\sigma \cap \H_\rho^\omega$. By assumption $\rho = \HolInd_H^G(\sigma, \widetilde{\chi})$ and $V_\sigma^{\mcO_b} \subseteq \H_\rho^\infty$. As $\mcD_{\widetilde{\chi}}$ is killed by $d\rho(\n_-)$ and dense in $V_\sigma$, it follows that $d\rho(\n_-)V_\sigma^{\mcO_b} = \{0\}$. Thus $(1)$ in \Fref{lem: functions_on_bundle_holomorphic} is satisfied, from which we obtain that $V_\sigma^{\mcO_b} \subseteq \H_\rho^\omega$. This means that $V_\sigma^{\mcO_b} \subseteq \mcD_{\widetilde{\chi}}$. So $(2)$ is satisfied using the trivial extension $\widetilde{\chi}$ on the subspace $\mcD_{\widetilde{\chi}} \subseteq V_\sigma$.
	\end{proof}

	\section{Arveson spectral theory}\label{sec: Arveson_spectral_theory}
	
	\noindent
	In \Fref{sec: pe_and_hol_ind} below, we shall have need for a suitably general notion of Arveson spectral subspaces. As such, we extend the already existing notion to a more general setting. Let $V$ be a complete locally convex vector space over $\C$ that is Hausdorff. We define Arveson spectral subspaces of $V$ associated to a strongly continuous $\R$-representation $\alpha$ on $V$ that satisfies a suitable condition, using the convolution algebra $\S(\R)$ of $\C$-valued Schwartz functions on $\R$. The results are adaptations of those in \citep[Sec.\ 2]{Arveson_groups_of_aut_of_oa}, \citep[Sec.\ A.3]{KH_Zellner_inv_smooth_vectors} and \citep[Sec.\ A.2]{Neeb_hol_reps}. 
	
	\subsection{Certain classes of $\R$-representations}
	Throughout the section, let $\alpha : \R \to \B(V)^\times$ be a strongly continuous representation of $\R$ on $V$. In \citep[Sec.\ A.3]{KH_Zellner_inv_smooth_vectors}, the $\R$-action $\alpha$ is required to be polynomially bounded (see \Fref{def: polynomial_bdd_action} below). It will however be convenient to define both a stronger and a weaker notion, that in turn are both still weaker than equicontinuity, which is used in \citep[Sec.\ A.2]{Neeb_hol_reps}. 
	
	\newpage
	\begin{definition}\label{def: polynomial_bdd_action}
		Let $\alpha : \R \to \B(V)^\times$ be a strongly continuous representation of $\R$ on $V$.
		\begin{itemize}
			\item $\alpha$ is said to be \textit{equicontinuous} if there is a basis of absolutely convex $\alpha$-invariant $0$-neighborhoods in $V$. Equivalently, if the topology of $V$ is defined by a family of $\alpha$-invariant continuous seminorms.
			\item $\alpha$ is said to have \textit{polynomial growth} if there is a basis $\mc{B}$ of absolutely convex $0$-neighborhoods in $V$ such that for every $U \in \mc{B}$ there is a monic polynomial $r \in \R[t]$ such that $\alpha_t(U) \subseteq r(|t|)U$ for all $t \in \R$. Equivalently, if there is a family $\mc{P}$ of defining seminorms on $V$ such that for every $p \in \mc{P}$ there exists a monic polynomial $r \in \R[t]$ such that $p(\alpha_t(v)) \leq r(|t|)p(v)$ for all $t \in \R$ and $v \in V$. 
			\item $\alpha : \R \to \B(V)^\times$ is called \textit{polynomially bounded} if for every continuous seminorm $p$ on $V$, there is a $0$-neighborhood $U \subseteq V$ and some $N \in \N$ such that 
			$$\sup_{v \in U}\sup_{t \in \R} {p(\alpha_t(v)) \over 1+ |t|^N} < \infty.$$
			\item $\alpha : \R \to \B(V)^\times$ is said to be \textit{pointwise polynomially bounded} if for every $v \in V$ and continuous seminorm $p$ on $V$, there exists $N \in \N$ such that 
			$$\sup_{t \in \R} {p(\alpha_t(v)) \over 1+ |t|^N} < \infty.$$
		\end{itemize} 
	\end{definition}	
	
	\begin{remark}
		Notice that we have the following implications:
		\begin{align*}
			\text{ $\alpha$ is equicontinuous } 
			\implies \text{ $\alpha$ has polynomial growth }
			\implies \text{ $\alpha$ is polynomially bounded}.
		\end{align*}
		If $V$ is a Banach space, then $\alpha$ has polynomial growth if and only if it is polynomially bounded.
	\end{remark}
	
	\begin{example}~
		\begin{enumerate}
			\item The $\R$-representations on both $L^2(\R)$ and $C^\infty(\T)$ by translation are equicontinuous.
			\item The $\R$-action $\alpha$ on $\S(\R)$ by translation is not equicontinuous but does have polynomial growth. Indeed, one checks that the open set 
			$$U := \set{f \in \S(\R) \st \sup_{x \in \R}|xf(x)| < 1} \subseteq \S(\R)$$
			satisfies $\bigcap_{t \in \R} \alpha_t(U) = \{0\}$. By \citep[Prop.\ A.1]{Neeb_hol_reps}, this implies that $\alpha$ is not equicontinuous. It does have polynomial growth, because the topology on $\S(\R)$ is generated by the seminorms 
			$$p_{n,m}(f):= \sup_{x \in \R} (1+|x|)^n |(\partial^m f)(x)|, \qquad \text{ for } n,m \in \N_{\geq 0},$$
			which satisfy $p_{n,m}(\alpha_t f) \leq \bigg[\sum_{k=0}^n {n \choose k}|t|^{n-k}\bigg] p_{n,m}(f)$ for $t \in \R$ and $f \in \S(\R)$. 
			\item The action of $\R$ on $C^\infty(\R)$ by translations is not pointwise polynomially bounded. For example, the smooth function $f(x) = e^x$ satisfies $\|\alpha_t(f)\|_{C([0, 1])} = \|f\|_{C[t, t+t]}\geq e^t$ for all $t \in \R$.
		\end{enumerate}
	\end{example}
	
	\noindent
	Let $\mc{P}$ denote the set of continuous seminorms on $V$. For $p \in \mc{P}$, let 
	$$\mc{N}_p := \set{v \in V \st p(v) = 0}$$
	denote its kernel. Let $V_p := \overline{V/\mc{N}_p}$ be the corresponding Banach space. If $p,q \in \mc{P}$ and $p \leq q$, then $\mc{N}_q \subseteq \mc{N}_p$ and hence there is a canonical contraction $\eta_{p,q} : V_q \to V_p$. 
	
	\noindent
	\begin{lemma}\label{lem: pol_growth_proj_lim}
		Assume that $\alpha$ is strongly continuous and has polynomial growth. Then $\alpha$ descends for each $p \in \mc{P}$ to a representation of $\R$ on $V_p$ with polynomial growth. Moreover $V = \varprojlim V_p$ as $\R$-representations.
	\end{lemma}
	\begin{proof}
		Let $p \in \mc{P}$. Since $\alpha$ has polynomial growth, we have $\alpha_t(\mc{N}_p) \subseteq \mc{N}_p$ for every $t \in \R$. Consequently, $\alpha$ descends to a strongly continuous $\R$-representation $\alpha^{(p)}$ on $V_p$ that again has polynomial growth. If $p,q \in \mc{P}$ and $t \in \R$, then $\eta_{p,q} \circ \alpha_t^{(q)} = \alpha_t^{(p)}$. We thus obtain an $\R$-action on the projective limit $\varprojlim V_p$ for which the canonical isomorphism $V \cong \varprojlim V_p$ is $\R$-equivariant.
	\end{proof}

	\begin{proposition}\label{prop: strongly_cts_and_pol_growth_implies_cts}
		Assume that $\alpha$ is strongly continuous and has polynomial growth. Then the action $\alpha : \R \times V \to V$ is continuous.
	\end{proposition}
	\begin{proof}
		By \Fref{lem: pol_growth_proj_lim} it follows that $V = \varprojlim V_p$ as $\R$-representation on locally convex space. If $p \in \mc{P}$, then since $V_p$ is a Banach space and the $\R$-representation on $V_p$ is strongly continuous, it follows from \citep[Prop.\ 5.1]{Neeb_diffvectors} that the action map $\R \times V_p \to V_p$ is jointly continuous. Using that $V \cong \varprojlim V_p$ as topological representations of $\R$, it follows that the action $\alpha : \R \times V \to V$ is jointly continuous.
	\end{proof}
	
	\subsection{Arveson spectral subspaces}
	\noindent
	Let $V$ be a complete locally convex vector space over $\C$ that is Hausdorff. Let $\alpha : \R \to \B(V)^\times$ be a strongly continuous representation of $\R$ on $V$. Assume that $\alpha$ is pointwise polynomially bounded. In the following, we define the Arveson spectral subspaces of $V$ associated to subsets $E$ of $\R$. We extend the results in \citep[A.3]{KH_Zellner_inv_smooth_vectors} to the case where $\alpha$ is only required to be pointwise polynomial bounded. We will use the convention that the Fourier transform $f\mapsto \widehat{f}$ on $\S(\R)$ is given by 
	\begin{equation}\label{eq: Fourier_transform}
		\widehat{f}(p) := \int_\R f(t)e^{itp}dt.	
	\end{equation}
	
	\begin{definition}\label{def: ideals_and_hulls}~
		\begin{itemize}
			\item If $I \subseteq \S(\R)$ is an ideal, define its \textit{hull} $h(I)\subseteq \R$ by 
			$$h(I) := \set{p \in \R \st \widehat{f}(p) = 0 \text{ for all } f \in I}.$$
			\item If $E \subseteq \R$ is a closed subset, define the ideal $I_0(E)$ of $\S(\R)$ by 
			$$I_0(E) := \set{f \in \S(\R) \st \supp(\widehat{f}) \cap E = \emptyset}.$$
		\end{itemize}
	\end{definition}
	
	\begin{lemma}[{\citep[Prop.\ A.8]{KH_Zellner_inv_smooth_vectors}}]\label{lem: ideals_and_hulls}~
		\begin{enumerate}
			\item If $E \subseteq \R$ is a closed subset, then $h(I_0(E)) = E$.
			\item If $I \subseteq \S(\R)$ is a closed ideal, then $I_0(h(I)) \subseteq I$.\vspace{.2cm}
		\end{enumerate}
	\end{lemma}
	
	\begin{corollary}\label{cor: empty_hull_dense}
		Let $I \subseteq \S(\R)$ be a closed ideal with $h(I) = \emptyset$. Then $I = \S(\R)$. 
	\end{corollary}
	\begin{proof}
		Since $I_0(\emptyset) = \S(\R)$ it follows using \Fref{lem: ideals_and_hulls} that $\S(\R) = I_0(\emptyset) = I_0(h(I)) \subseteq I$.
	\end{proof}

	\noindent
	We proceed by defining a representation of the convolution algebra $(\S(\R), \ast)$ on $V$.

	\begin{lemma}
		Let $f \in \S(\R)$ and $v \in V$. Then the weak integral $\int_\R f(t)\alpha_t(v)dt$ exists in $V$. 
	\end{lemma}
	\begin{proof}
		For any $a>0$, the weak integral $\int_{-a}^a f(t)\alpha_t(v)dt$ exists in $V$ because $\R \to V, \; t \mapsto f(t)\alpha_t(v)$ is continuous and $V$ is complete (cf.\ \citep[p.\ 1021]{milnor_inf_lie} or \citep[Prop.\ 1.1.15]{Neeb_Glockner_Book}). As $\alpha$ is pointwise polynomially bounded and $f \in \S(\R)$ is a Schwartz function, the limit $v_\ast := \lim_{a \to \infty} \int_{-a}^a f(t)\alpha_t(v)dt$ exists in $V$, and $v_\ast = \int_\R f(t)\alpha_t(v)dt \in V$.
	\end{proof}

	\begin{definition}\label{def: repr_of_schartz}
		For any Schwartz function $f \in \S(\R)$, define the linear operator $\alpha_f \in \L(V)$ by 
		$$\alpha_f(v) := \int_\R f(t)\alpha_t(v)dt.$$
		Then $f \mapsto \alpha_f$ defines a strongly continuous representation of the convolution algebra $(\S(\R), \ast)$ on $V$.	
	\end{definition}

	\begin{remark}\label{rem: pol_bdd_cts}
		If $\alpha$ is polynomially bounded, then $\alpha_f \in \B(V)$ is a continuous operator for every $f \in \S(\R)$.
	\end{remark}

	\begin{definition}\label{def: arv_spec}~
		\begin{itemize}
			\item Let $\Spec_\alpha(V) := h(\ker \alpha) \subseteq \R$ be the hull of the closed ideal $\ker \alpha$ in $\S(\R)$. 
			\item For $v \in V$, let $\S(\R)_v := \set{f \in \S(\R) \st \alpha_f(v) = 0}$ denote the annihilator of $v$ in $\S(\R)$, which is a closed ideal in $\S(\R)$, and let $\Spec_\alpha(v) := h(\S(\R)_v) \subseteq \R$ be its hull.
			\item If $E \subseteq \R$ is a subset, define 
			$$V_\alpha(E)_0 := \set{v \in V \st \Spec_\alpha(v) \subseteq \overline{E}}$$
			and let $V_\alpha(E) := \overline{V_\alpha(E)_0}$ be its closure in $V$. Define moreover 
			$$V_\alpha^+(E) := \bigcap_N V_\alpha(E + N),$$
			where $N$ runs over all $0$-neighborhoods in $\R$.
		\end{itemize}
		If the action $\alpha$ is clear from the context, we drop $\alpha$ from the notation and simply write $V(E)_0, V(E)$ and $V^+(E)$ instead of $V_\alpha(E)_0, V_\alpha(E)$ and $V_\alpha^+(E)$.\\
	\end{definition}
	
	\begin{example}\label{ex: spectra_Hilbert_space_case}
		Let $U : \R \to \U(\H)$ be a strongly continuous unitary representation of $\R$. Then $U_t = e^{tH}$ for some self-adjoint operator $H$ on $\H$. Suppose that $U_t = \int_\R e^{itp}dP(p)$ is the corresponding spectral decomposition of $U$, for some projection-valued measure $P$ on $\R$. With the convention \eqref{eq: Fourier_transform} we have $U_f := \int_\R f(t)U_t dt = \int_\R \widehat{f}(p)dP(p)$ for $f \in \S(\R)$. The Arveson spectrum $\Spec_U(\H)$ coincides with $\Spec(H)$, the spectrum of the self-adjoint operator $H$. Moreover, for a closed subset $E \subseteq \R$, the corresponding spectral subspace is given by $\H_U(E) = P(E)\H$.\\
	\end{example}
	
	\begin{remark}\label{rem: intersections}
		Let $\{E_i\}_{i \in \mc{I}}$ be a family of closed subsets of $\R$. Observe that $\bigcap_{i \in \mc{I}}V(E_i)_0 = V\big(\bigcap_{i \in \mc{I}}E_i\big)_0$.\\
	\end{remark}
	
	\begin{remark}\label{rem: spec}
		Notice for any $v \in V$ that $\ker \alpha \subseteq \S(\R)_v$, so $\Spec_\alpha(v) \subseteq \Spec_\alpha(V)$. Thus
		$$V = V(\Spec_\alpha(V))_0 = V\big(\Spec_\alpha(V)\big) = V^+\big(\Spec_\alpha(V)\big).$$
		Combining this with \Fref{rem: intersections}, we obtain for any closed subset $E \subseteq \R$ that $V(E)_0 = V(E \cap \Spec_\alpha(V))_0$.\\
	\end{remark}

	\begin{lemma}\label{lem: specfv_in_supp_f}
		Let $f \in \S(\R)$ and $v \in V$. Then $\Spec_{\alpha}(\alpha_f(v)) \subseteq \supp(\widehat{f})$.
	\end{lemma}
	\begin{proof}
		Let $p \in \R\setminus \supp(\widehat{f})$ and choose $g \in \S(\R)$ s.t.\ $\widehat{g}(p) \neq 0$ and $\restr{\widehat{g}}{\supp(\widehat{f})} = 0$. Then $g\ast f = 0$, because $\widehat{g}\widehat{f} = 0$. It follows that $\alpha_g\alpha_f v = \alpha_{f \ast g}v =0$. Since we also have $\widehat{g}(p) \neq 0$ it follows that $p \notin \Spec_{\alpha}(\alpha_f v)$.
	\end{proof}

	\begin{proposition}\label{prop: non-empty-spectrum}
		Let $v\in V$. Then $\S(\R)_v = \S(\R)$ implies $v = 0$. Moreover $v\neq 0$ implies $\Spec_\alpha(v) \neq \emptyset$. 
	\end{proposition}
	\begin{proof}
		Assume that $\S(\R)_v = \S(\R)$. If $\lambda \in V^\prime$ is a continuous functional, it follows that $\int_\R f(t) \lambda(\alpha_t v) dt = 0$ for any $f \in \S(\R)$. As $t \mapsto \lambda(\alpha_t v)$ is continuous, this implies that $\lambda(\alpha_t v) = 0$ for all $t \in \R$. In particular $\lambda(v) = 0$. As $V^\prime$ separates the points of $V$ by the Hahn-Banach Theorem \citep[Thm.\ I.3.4]{Rudin_FA}, it follow that $v = 0$. Finally, if $\Spec_\alpha(v) = \emptyset$ then by \Fref{cor: empty_hull_dense} it follows that $\S(\R)_v = \S(\R)$ and hence $v =0$.
	\end{proof}
	
	\begin{corollary}
		If $E_1, E_2 \subseteq \R$ are two disjoint closed subsets, then 
		$$V(E_1)_0 \cap V(E_2)_0 = \{0\}.$$ 
	\end{corollary}
	\begin{proof}
		We have $V(E_1)_0 \cap V(E_2)_0 = V(E_1 \cap E_2) = V(\emptyset) = \{0\}$ by \Fref{rem: intersections} and \Fref{prop: non-empty-spectrum}. \qedhere\\
	\end{proof}
	
	\noindent
	If $E \subseteq \R$ is a subset, recall from \Fref{def: ideals_and_hulls} that $I_0(\olE) \subseteq \S(\R)$ denotes the ideal of functions $f \in \S(\R)$ whose Fourier transform $\widehat{f}$ vanishes on a neighborhood of $\olE \subseteq \R$. \Fref{prop: spectral_subspace_char} below provides a convenient characterization of $V(E)_0$ in terms of $I_0(\olE)$, which will be used repeatedly.
	
	\begin{proposition}[{\citep[Prop.\ A.8]{KH_Zellner_inv_smooth_vectors}}]\label{prop: spectral_subspace_char}~\\
		For any subset $E \subseteq \R$ we have 
		\begin{align*}
			V(E)_0 
			&= \set{v \in V \st I_0(\olE) \subseteq \S(\R)_v} \\
			&= \set{ v \in V \st \forall f \in \S(\R) \st \supp(\widehat{f}) \cap \olE = \emptyset \implies \alpha_f(v) = 0}.	
		\end{align*}
		In particular $V(E)_0, V(E)$ and $V^+(E)$ are linear subspaces of $V$.		
	\end{proposition}
	\begin{proof}
		The proof of \citep[Prop.\ A.8]{KH_Zellner_inv_smooth_vectors} continues to hold when $\alpha$ is only pointwise polynomially bounded.
	\end{proof}
	
	\begin{corollary}\label{cor: pol_bdd_arv_subspaces}
		Assume that $\alpha$ is polynomially bounded. Then for any $E \subseteq \R$ we have 
		$$V(E)_0 = V(E) = V^+(E).$$
	\end{corollary}
	\begin{proof}
		Let $E \subseteq \R$ be a subset. By \Fref{rem: pol_bdd_cts} we know that $\alpha_f$ is a continuous linear operator for every $f \in \S(\R)$. It then follows from \Fref{prop: spectral_subspace_char} that $V(E)_0$ is closed, so $V(E)_0 = V(E)$. Using \Fref{rem: intersections}, we further obtain that 
		$$V^+(E) = \bigcap_N V(E+N) = \bigcap_N V(E+N)_0 = V\bigg(\bigcap_N E+N \bigg)_0 = V(\overline{E})_0 = V(E)_0.\qedhere$$
	\end{proof}
	
	\noindent
	The following will also be used frequently:
	
	\begin{corollary}\label{cor: pointwise_char_of_full_spec}
		Let $E \subseteq \R$ be a subset. The following assertions are equivalent:
		\begin{enumerate}
			\item $\Spec_\alpha(V) \subseteq \overline{E}\quad$.
			\item $V \subseteq V(E)_0$.
			\item $I_0(\overline{E}) \subseteq \ker \alpha$.
		\end{enumerate}
	\end{corollary}
	\begin{proof}
		Assume that $\Spec_\alpha(V) \subseteq \olE$. Then for any $v \in V$ we have 
		$$\Spec_\alpha(v) \subseteq \Spec_\alpha(V) \subseteq \olE,$$
		by \Fref{rem: spec}. This means that $V \subseteq V(E)_0$. Assume next that $V \subseteq V(E)_0$. By \Fref{prop: spectral_subspace_char}, this means that $I_0(\overline{E}) \subseteq \S(\R)_v$ for all $v \in V$. So elements of $I_0(\overline{E})$ annihilate every $v \in V$. Thus $I_0(\olE) \subseteq \ker \alpha$. If $I_0(\olE) \subseteq \ker \alpha$, then $\Spec_\alpha(V) = h(\ker \alpha) \subseteq h(I_0(\olE)) = \olE$, where the last equality uses \Fref{lem: ideals_and_hulls}.
	\end{proof}
	
	\begin{corollary}\label{cor: full_spec_as_closed_union_of_pts}
		$\Spec_\alpha(V) = \overline{\bigcup_{v \in V} \Spec_\alpha(v)}$.
	\end{corollary}
	\begin{proof}
		Define $E := \bigcup_{v \in V} \Spec_\alpha(v)$. By \Fref{rem: spec} we have $\Spec_\alpha(v) \subseteq \Spec_\alpha(V)$ for any $v \in V$. As $\Spec_\alpha(V)$ is closed, it follows that $\overline{E} \subseteq \Spec_\alpha(V)$. Conversely, recall that $V(E)_0 = \set{v \in V \st \Spec_\alpha(v) \in \overline{E}}$. So from our definition of $E$, we trivially have $V \subseteq V(E)_0$. Then $\Spec_\alpha(V) \subseteq \olE$ follows by \Fref{cor: pointwise_char_of_full_spec}.
	\end{proof}
	
	\noindent
	Let us next record the behavior of spectral subspaces under continuous linear and multi-linear maps:
	\begin{proposition}\label{prop: spectra_linear_map}
		For $j \in \{1,2\}$, let $\alpha_j : \R \to \B(V_j)^\times$ be a strongly continuous representation of $\R$ on the complete and Hausdorff complex locally convex vector space $V_j$. Assume that $\alpha_j$ is pointwise polynomially bounded. Let $T : V_1 \to V_2$ be a continuous $\R$-equivariant linear map. Then for every subset $E \subseteq \R$ we have $T(V_1(E)) \subseteq V_2(E)$. If $T$ is injective, then $\Spec_{\alpha_1}(V_1) \subseteq \Spec_{\alpha_2}(V_2)$. 
	\end{proposition}
	\begin{proof}
		Let $v \in V$. As $T$ is equivariant, we have that $\S(\R)_{v} \subseteq \S(\R)_{Tv}$. Hence $h(\S(\R)_{Tv}) \subseteq h(\S(\R)_v)$, which is to say that $\Spec_{\alpha_2}(Tv)\subseteq \Spec_{\alpha_1}(v)$. Thus if $E \subseteq \R$ is a subset then $TV_1(E)_0 \subseteq V_2(E)_0$. As $T$ is continuous, it also follows that $TV_1(E) \subseteq V_2(E)$. If $T$ is injective, we have $\S(\R)_{v} = \S(\R)_{Tv}$ for any $v \in V_1$. Consequently, $\Spec_{\alpha_1}(v) = \Spec_{\alpha_2}(Tv) \subseteq \Spec_{\alpha_2}(V_2)$. As $\Spec_{\alpha_2}(V_2)$ is closed, it follows using \Fref{cor: full_spec_as_closed_union_of_pts} that $\Spec_{\alpha_1}(V_1) = \overline{\bigcup_{v \in V_1} \Spec_{\alpha_1}(v)} \subseteq \Spec_{\alpha_2}(V_2)$.\qedhere\\
	\end{proof}

	\noindent
	In the multi-linear context, we have the following analogue of \citep[A.10]{KH_Zellner_inv_smooth_vectors}:
	
	\begin{restatable}{proposition}{arvspecbil}\label{prop: arv_spec_subs_additive}
		For $j \in \{1,2,3\}$, let $\alpha_j : \R \to \B(V_j)^\times$ be a strongly continuous representation of $\R$ on the complete and Hausdorff complex locally convex vector space $V_j$. Assume that $\alpha_j$ is pointwise polynomially bounded. Let $\beta : V_1 \times V_2 \to V_3$ be a continuous $\R$-equivariant bilinear map. Let $E_1, E_2 \subseteq \R$ be closed subsets. Then 
		$$ \beta(V_1(E) \times V_2(E)) \subseteq V_3^+(E_1 + E_2). $$
		In particular, if $\alpha_3$ is polynomially bounded then $\beta(V_1(E) \times V_2(E)) \subseteq V_3(E_1 + E_2)$. \\
	\end{restatable}	
	
	\noindent
	Before proceeding to to proof of \Fref{prop: arv_spec_subs_additive}, let us mention the following immediate consequence:
	\begin{corollary}\label{cor: arv_spec_additive}
		Consider the setting of \Fref{prop: arv_spec_subs_additive}. Assume additionally that $\beta$ has dense span and that $\alpha_3$ is polynomially bounded. Then 
		$$\Spec_{\alpha_3}(V_3) \subseteq \overline{\Spec_{\alpha_1}(V_1) + \Spec_{\alpha_2}(V_2)}.$$
	\end{corollary}
	\begin{proof}
		We know by \Fref{prop: arv_spec_subs_additive} that $\beta(V_1, V_2) \subseteq V_3^+\big(\Spec_{\alpha_1}(V_1) + \Spec_{\alpha_2}(V_2)\big)$. In view of \Fref{prop: spectral_subspace_char} and \Fref{cor: pol_bdd_arv_subspaces}, we further know that 
		$$V_3^+\big(\Spec_{\alpha_1}(V_1) + \Spec_{\alpha_2}(V_2)\big) = V_3\big(\Spec_{\alpha_1}(V_1) + \Spec_{\alpha_2}(V_2)\big)_0,$$
		and this is a closed linear subspace of $V_3$. As $\beta(V_1, V_2)$ has dense linear span in $V_3$, it follows that 
		$$V_3 \subseteq V_3\big(\Spec_{\alpha_1}(V_1) + \Spec_{\alpha_2}(V_2)\big)_0.$$
		According to \Fref{cor: pointwise_char_of_full_spec}, this equivalent with 
		$$\Spec_{\alpha_3}(V_3) \subseteq \overline{\Spec_{\alpha_1}(V_1) + \Spec_{\alpha_2}(V_2)}.\qedhere$$
	\end{proof}
	
	\noindent
	The proof of \Fref{prop: arv_spec_subs_additive} requires some preparation. It closely follows that of \citep[Prop.\ 2.2]{Arveson_groups_of_aut_of_oa} and \citep[Prop.\ A.14]{Neeb_hol_reps}. We first introduce some additional notation:
	
	\begin{definition}
		For a subset $E \subseteq \R$, define the ideal $J(E) \subseteq \S(\R)$ and the subspace $R_\alpha(E)_0 \subseteq V$ by 
		\begin{align*}
			J(E) &:= \set{ f \in \S(\R) \st \widehat{f} \in C^\infty_c(\R) \text{ and } \supp{\widehat{f}} \subseteq E },\\
			R_\alpha(E)_0 &:= \set{\alpha_f v \st f \in J(E), \; v \in V}\subseteq V
		\end{align*}
		Let $R_\alpha(E) := \overline{R_\alpha(E)_0}$ be its closure. If $\alpha$ is clear from the context, we write simply $R(E)_0$ and $R(E)$ instead of $R_\alpha(E)_0$ and $R_\alpha(E)$, respectively.
	\end{definition}

	\noindent
	If $E \subseteq \R$ is a subset, recall from \Fref{def: ideals_and_hulls} that $I_0(\overline{E})$ consists of all Schwartz functions $f$ whose Fourier transform $\widehat{f}$ vanishes on a neighborhood of $\overline{E} \subseteq \R$. On the other hand, $J(E)$ is the ideal in $\S(\R)$ generated by those $f \in \S(\R)$ for which $\widehat{f}$ has compact support contained in $E$.
	
	\begin{lemma}\label{lem: complementary_ideals_schwartz}
		Let $E \subseteq \R$ be a closed subset and let $N \subseteq \R$ be a $0$-neighborhood. Then 
		$$\overline{I_0(E) + J(E+N)}=\S(\R).$$
	\end{lemma}
	\begin{proof}
		Let $J_2 := \overline{J(E+N)_0 + I_0(E)}$ be the closed ideal of $\S(\R)$ generated by $J(E+N)$ and $I_0(E)$. Observe that $h(J(E+N)) \subseteq \R\setminus E$. On the other hand, $h(I_0(E)) \subseteq E$ by \Fref{lem: ideals_and_hulls}. We thus find that
		$$h(J_2) \subseteq h(I_0(E)) \cap h(J(E+N)) \subseteq \emptyset$$
		and hence $h(J_2) = \emptyset$. It follows using \Fref{cor: empty_hull_dense} that $J_2 = \S(\R)$.
	\end{proof}

	\begin{lemma}\label{lem: annihilator_large_implies_v_zero}
		Let $v \in V$ and $N\subseteq \R$ be a $0$-neighborhood. If $J(\Spec_\alpha(V)+N) \subseteq \S(\R)_v$, then $v = 0$.
	\end{lemma}
	\begin{proof}
		Let $E := \Spec_\alpha(V)$. Assume that $J(E+N) \subseteq \S(\R)_v$. Recall from \Fref{rem: spec} that $V = V(E)_0$. By \Fref{prop: spectral_subspace_char}, this means that $I_0(E) \subseteq \S(\R)_v$. On the other hand, $J(E+N) \subseteq \S(\R)_v$, by assumption. Since $\S(\R)_v$ is closed we obtain using \Fref{lem: complementary_ideals_schwartz} that $\S(\R) = \overline{I_0(E) + J(E+N)} \subseteq \S(\R)_v$. By \Fref{prop: non-empty-spectrum}, this implies that $v = 0$.
	\end{proof}

	\begin{lemma}\label{lem: arv_spec_towards_pf_bil}
		Let $E \subseteq \R$ be closed. Then $V(E) \subseteq \bigcap_N R(E + N) \subseteq V^+(E)$, where $N$ runs over all open $0$-neighborhoods in $\R$. 
	\end{lemma}
	\begin{proof}
		This proof follows that of \citep[Prop.\ 2.2]{Arveson_groups_of_aut_of_oa}. \Fref{lem: specfv_in_supp_f} entails that $\Spec_{\alpha}(\alpha_fv) \subseteq \supp(\widehat{f})$ for any $f \in \S(\R)$ and $v \in V$. If $N\subseteq \R$ is a $0$-neighborhood and $f \in J(E+N)$, then by definition $\supp \widehat{f} \subseteq E + N$ and hence $\Spec_\alpha(\alpha_f v) \subseteq E + N$ for any $v \in V$. Recalling that $R(E + N)_0$ is the subspace of $V$ generated by $J(E+N)$, we obtain that $R(E + N)_0\subseteq V(E + N)_0$. Consequently $\bigcap_N R(E + N) \subseteq \bigcap_N V(E + N) = V^+(E)$. Next, take $v \in V(E)_0$. We show that $v \in \bigcap_N R(E + N)$. Let $N$ be a $0$-neighborhood in $\R$. Let $\lambda \in V^\prime$ be a continuous functional with $\lambda(R(E+N)) = \{0\}$. Trivially, $\alpha_f(v) \in R(E+N)_0$ for any $f \in J(E+N)$, and hence $\lambda(\alpha_fv) = 0$. We further have $I_0(E) \subseteq \S(\R)_v$, by \Fref{prop: spectral_subspace_char}, and consequently $\lambda(\alpha_g v) = 0$ for any $g \in I_0(E)$. Thus $\lambda(\alpha_fv) = 0$ for any $f$ in the closed ideal $J_2 := \overline{I_0(E) + J(E+N)}$ of $\S(\R)$ spanned by $I_0(E)$ and $J(E+N)$. By \Fref{lem: complementary_ideals_schwartz} this ideal equals $\S(\R)$, so $\int_\R f(t)\lambda(\alpha_t v) dt = \lambda(\alpha_fv) = 0$ for any $f \in \S(\R)$. As $t\mapsto \lambda(\alpha_t v)$ is continuous, it follows that $\lambda(\alpha_t v) = 0$ for all $t \in \R$. In particular $\lambda(v) = 0$. Using the Hahn-Banach Theorem \citep[Thm.\ I.3.5]{Rudin_FA}, it follows that $v \in \bigcap_N R(E + N)$. Thus $V(E)_0 \subseteq \bigcap_N R(E + N)$ and consequently also $V(E) \subseteq \bigcap_N R(E + N)$.
	\end{proof}

	\begin{proof}[Proof of \Fref{prop: arv_spec_subs_additive}:]
		\noindent
		Having \Fref{lem: arv_spec_towards_pf_bil} at hand, we proceed as in \citep[Prop.\ A.14]{Neeb_hol_reps}. Let $N \subseteq \R$ be an open $0$-neighborhood. Let $N_1, N_2 \subseteq \R$ be open $0$-neighborhoods s.t.\ $N_1 + N_2 \subseteq N$. We show that 
		\begin{equation}\label{eq: bil_map_spectra}
			\beta\bigg(R_{\alpha_1}(E_1 + N_1)_0 \times R_{\alpha_2}(E_2 + N_2)_0\bigg) \subseteq V\big(E_1 + E_2 + N\big)_0.	
		\end{equation}
		As such, for $k \in \{1,2\}$, take $v_k \in V$ and $f_k \in J(E_k + N_k)$, meaning that $\supp(\widehat{f_k}) \subseteq E_k + N_k$. We show that $\beta(\alpha_1(f_1)v_1, \alpha_2(f_2)v_2) \in V\big(E_1 + E_2 + N\big)_0$. In view of \Fref{prop: spectral_subspace_char}, we must show that it is annihilated by $I_0(\overline{E_1 + E_2 + N})$. Let $f_3 \in I_0(\overline{E_1 + E_2 + N})$, so $\supp(\widehat{f_3}) \cap \overline{E_1 + E_2 + N} = \emptyset$. Then
		\begin{align}\label{eq: eq_bil_map_spectra_01}
			\begin{split}
				\alpha_{f_3}\beta(\alpha_{f_1}(v_1), \alpha_{f_2}(v_2)) 
				&= \int_\R \int_\R f_1(t_1)f_2(t_2)f_3(t_3)\beta(\alpha_1(t_1+t_3)v_1, \alpha_2(t_2+t_3)v_2)dt_1dt_2dt_3,	\\
				&= \int_\R \int_\R F(t_1, t_2) \beta(\alpha_1(t_1)v_1, \alpha_2(t_2)v_2)dt_1dt_2,	
			\end{split}
		\end{align}
		where $F \in \S(\R^2)$ is defined by 
		$$ F(t_1, t_2) := \int_\R f_3(t_3)f_1(t_1 - t_3)f_2(t_2 - t_3)dt_3. $$
		The Fourier transform $\widehat{F} \in \S(\R^2)$ of $F$ is given by 
		$$\widehat{F}(p_1, p_2) = \widehat{f_1}(p_1)\widehat{f_2}(p_2)\widehat{f_3}(p_1 + p_2).$$
		Observe that $\supp(\widehat{f_1}) + \supp(\widehat{f_2}) \subseteq (E_1 + N_1) + (E_2 + N_2) \subseteq E_1 + E_2 + N$. Since $\widehat{f_3}$ vanishes on $E_1 + E_2 + N$, we find that $\widehat{F} = 0$. Hence $F= 0$. From \Fref{eq: eq_bil_map_spectra_01} we obtain that $\alpha_3(f_3)\beta(\alpha_1(f_1)v_1, \alpha_2(f_2)v_2) = 0$. By \Fref{prop: spectral_subspace_char} we conclude that $\beta(\alpha_1(f_1)v_1, \alpha_2(f_2)v_2) \in V\big(E_1 + E_2 + N\big)_0$. Thus \eqref{eq: bil_map_spectra} is valid. As $\beta$ is continuous, it follows that
		$$ \beta(V_1(E) \times V_2(E)) \subseteq \beta\bigg(R_{\alpha_1}(E_1 + N_1) \times R_{\alpha_2}(E_2 + N_2)\bigg) \subseteq V\big(E_1 + E_2 + N\big),	 $$
		where the first inclusion uses \Fref{lem: arv_spec_towards_pf_bil}. Thus 
		$$\beta(V_1(E) \times V_2(E)) \subseteq \bigcap_N V\big(E_1 + E_2 + N\big) = V^+(E_1 + E_2).$$
		Assume next that $\alpha_3$ is polynomially bounded. Then $\alpha_f$ is continuous for every $f \in \S(\R)$, by \Fref{rem: pol_bdd_cts}. By \Fref{cor: pol_bdd_arv_subspaces} it follows that 
		$$V^+(E_1 + E_2) = V(E_1 + E_2).\qedhere$$
	\end{proof}

	\noindent
	Let us next consider the behavior of spectra under tensor products and spaces of continuous linear maps:
	
	\begin{proposition}\label{prop: arv_spec_tensor_prod_cont_maps}
		Let $\alpha$ and $\sigma$ be $\R$-representation on the complete and Hausdorff locally convex vector spaces $V$ and $W$ over $\C$, respectively. Assume that $\alpha$ and $\sigma$ are strongly continuous and have polynomial growth. Let $n \in \N$. 
		\begin{enumerate}
			\item The $\R$-representation $\alpha \widehat{\otimes} \sigma$ on the completed projective tensor product $V \widehat{\otimes}W$ has a continuous action $\R \times V \widehat{\otimes}W \to V \widehat{\otimes}W$, polynomial growth and satisfies
			\begin{equation}\label{eq: spec_of_tensor_prod}
				\Spec_{\alpha \widehat{\otimes} \sigma}(V \widehat{\otimes}W) \subseteq \overline{\Spec_\alpha(V) + \Spec_\alpha(W)}.	
			\end{equation}
			\item Equip $\B(V, W)$ either with the strong topology or that of uniform convergence on compact sets. The $\R$-representation $\gamma$ on $\B(V, W)$ defined by $\gamma_t T = \sigma_t \circ T\circ \alpha_{-t}$ is strongly continuous, pointwise polynomially bounded and satisfies 
			\begin{equation}\label{eq: spec_of_lin_maps}
				\Spec_\gamma(\B(V, W)) \subseteq \overline{\Spec_\sigma(W) - \Spec_{\alpha}(V)}.	
			\end{equation}
		\end{enumerate}
	\end{proposition}

	\begin{proof}
		Notice by \Fref{prop: strongly_cts_and_pol_growth_implies_cts} that the actions $\alpha : \R \times V \to V$ and $\sigma : \R \times W \to W$ are continuous.
		\begin{enumerate}
			\item We write $\gamma_t := \alpha_t \widehat{\otimes} \sigma_t$ for $t \in \R$. We first show that the $\R$-representation $\gamma$ on $V \widehat{\otimes}W$ has polynomial growth. Let $p$ and $q$ be continuous seminorms on $V$ and $W$ respectively. Assume that $p(\alpha_tv) \leq r_\alpha(|t|)p(v)$ and $q(\alpha_tw) \leq r_\sigma(|t|)q(w)$ for all $t \in \R$, $v \in V$ and $w \in W$, where $r_\alpha, r_\sigma \in \R[t]$ are monic polynomials. Using this inequality, it follows from the definition of the seminorm $p \otimes q$ on $V \widehat{\otimes}W$ (see \Fref{eq: seminorms_on_proj_tensors}) that $(p \otimes q)(\gamma_t\psi) \leq r_\alpha(|t|)r_\sigma(|t|) (p \otimes q)(\psi)$ for all $t \in \R$ and $\psi \in V \widehat{\otimes}W$. Thus $\alpha \widehat{\otimes} \sigma$ has polynomial growth. \\
			
			\noindent
			To see that $\alpha \widehat{\otimes} \sigma$ has a continuous action, it suffices by \Fref{prop: strongly_cts_and_pol_growth_implies_cts} to show it is strongly continuous. Let $\psi \in V \widehat{\otimes} W$. It suffices to show that $t \mapsto \gamma_t \psi$ is continuous at $t=0$. Assume first that $\psi \in V \otimes W$, so that $\psi = \sum_{k=1}^n v_k \otimes w_k$ for some $v_k \in V$ and $w_k \in W$. Let $p$ and $q$ be continuous seminorms on $V$ and $W$, respectively. Let $r_\alpha, r_\sigma \in \R[t]$ be as above. Let $\epsilon > 0$. As $\alpha$ and $\sigma$ are strongly continuous, we can find $\delta > 0$ s.t.\ $p(\alpha_tv_k - v_k)q(\sigma_tw_k) < \epsilon$ and $p(v_k)q(\sigma_tw_k - w_k) < \epsilon$ for all $t \in (-\delta, \delta)$ and $k \in \{1, \ldots, n\}$. Writing $\alpha_t v_k \otimes \sigma_tw_k - v_k \otimes w_k = (\alpha_t v_k - v_k)\otimes \sigma_tw_k + v_k\otimes(\sigma_t w_k - w_k)$, we obtain for any $t \in (-\delta, \delta)$ that
			$$ (p\otimes q)(\gamma_t\psi - \psi) \leq \sum_{k=1}^n p(\alpha_tv_k - v_k)q(\sigma_t w_k) + p(v_k)q(\sigma_t w_k - w_k) < 2k\epsilon. $$
			This proves that $\gamma_t \psi \to \psi$ as $t\to 0$, for any $\psi$ in the dense subspace $V \otimes W$. Let us next consider general $\psi \in V \widehat{\otimes} W$. Let $\eta \in V \otimes W$ be such that $(p \otimes q)(\psi - \eta) < \epsilon$. For small enough $\delta> 0$ we have $r_\alpha(|t|)r_\sigma(|t|) \leq 2$ and $(p\otimes q)(\gamma_t \eta - \eta) < \epsilon$ for all $t \in (-\delta, \delta)$. Using 
			$$(p \otimes q)(\gamma_t(\psi - \eta)) \leq  r_\alpha(|t|)r_\sigma(|t|)(p\otimes q)(\psi -\eta) < 2 \epsilon,$$
			we find for all $t \in (-\delta, \delta)$ that
			\begin{align*}
				(p\otimes q)(\gamma_t \psi -\psi) 
				&\leq (p \otimes q)(\gamma_t(\psi - \eta)) + (p \otimes q)(\psi - \eta)	+ (p\otimes q)(\gamma_t \eta - \eta) <4\epsilon
			\end{align*}
			Thus $\R \to V \widehat{\otimes}W, \; t \mapsto \gamma_t\psi$ is continuous.\\

			\noindent
			As the canonical bilinear map $\widehat{\otimes} :V \times W\to V \widehat{\otimes}W$ is continuous, $\R$-equivariant and has dense span in $V \widehat{\otimes} W$, the remaining assertion is immediate from \Fref{cor: arv_spec_additive}.
			
			\item It suffices to consider only the topology of uniform convergence on compact sets. Let $T \in \B(V, W)$. Let $q$ be a continuous seminorm on $W$ and let $K \subseteq V$ be compact. Consider the continuous seminorm on $\B(V, W)$ defined by $q_K(T) := \sup_{v \in K}q(Tv)$. As $T$ is bounded, there is a continuous seminorm $p$ on $V$ s.t.\ $q(Tv) \leq p(v)$ for all $v \in v$. Let $r_\sigma, r_\alpha \in \R[t]$ be monic polynomials s.t.\ $q(\sigma_t w) \leq r_\sigma(|t|)q(w)$ and $p(\alpha_{t}v) \leq r_\alpha(|t|)p(v)$ for all $t\in \R$, $v \in V$ and $w \in W$. Then 
			$$q_K(\gamma_t(T)) = \sup_{v \in K} q(\sigma_t T \alpha_{-t}v) \leq r_\sigma(|t|)r_\alpha(|t|) \sup p(K).$$
			This implies that $\gamma$ is pointwise polynomially bounded. \\
			
			\noindent
			We next show that $\gamma$ is strongly continuous. Let $T \in \B(V, W)$, $\epsilon > 0$ and ${\mcO} := q^{-1}([0, \epsilon)) \subseteq W$. The map 
			$$\Phi : \R \times V \to W, \quad \Phi(t,v) := \gamma_t(T) v - Tv = \sigma_t T \alpha_{-t}v - Tv$$
			is continuous, because the map $T : V \to W$ and the actions $\alpha : \R \times V \to V$ and $\sigma : \R \times W \to W$ are all continuous. Since $\{0\} \times K \subseteq \Phi^{-1}({\mcO})$ and $K$ is compact, it follows from the Tube Lemma (cf.\ \citep[Lem.\ 26.8]{Munkres_topology}) that there is an interval $I \subseteq \R$ containing $0$ s.t.\ $\Phi(I \times K) \subseteq {\mcO}$. This means that $q_K(\gamma_t(T) -T) < \epsilon$ for all $t \in I$, so $\gamma$ is strongly continuous. \\
			
			\noindent
			It remains to show \eqref{eq: spec_of_lin_maps}. Define $E_V := \Spec_\alpha(V)$ and $E_W := \Spec_\alpha(W)$. Let $N \subseteq \R$ be a $0$-neighborhood. Let $T \in \B(V, W)$ and $f_3 \in I_0(\overline{E_W - E_V + N})$, so $f_3 \in \S(\R)$ satisfies $\supp(\widehat{f_3}) \cap \overline{E_W - E_V + N} = \emptyset$. We show $\gamma_{f_3}(T) = 0$. Let $N_1, N_2 \subseteq \R$ be symmetric $0$-neighborhoods such that $N_1 + N_2 \subseteq N$. Let $v \in V$, $f_1 \in J(E_V + N_1)$ and $f_2 \in J(E_W + N_2)$. So $\widehat{f_1}$ and $\widehat{f_2}$ have compact support contained in $E_V + N_1$ and $E_W + N_2$, respectively. One verifies that
			\begin{align}\label{eq: spec_BVW_comp}
				\begin{split}
					\sigma_{f_2} \gamma_{f_3}(T) \alpha_{f_1}v 
					&= \int_\R \int_\R \int_\R f_1(t_1)f_2(t_2)f_3(t_3) \sigma_{t_2 + t_3}T \alpha_{t_1 -t_3}v \; dt_1dt_2dt_3\\
					&= \int_\R \int_\R F(t_1, t_2) \sigma_{t_2}T\alpha_{t_1} dt_1dt_2,
				\end{split}
			\end{align}
			where $F \in \S(\R^2)$ is given by
			$$ F(t_1, t_2) = \int_\R f_1(t_1+t_3)f_2(t_2 -t_3)f_3(t_3)dt_3. $$
			The Fourier transform $\widehat{F} \in \S(\R^2)$ of $F$ is given by 
			$$\widehat{F}(p_1, p_2) = \widehat{f_1}(p_1)\widehat{f_2}(p_2)\widehat{f_3}(p_2 - p_1).$$
			Recalling that $N_1$ is symmetric, notice that 
			\begin{align*}
				\supp(\widehat{f_2}) - \supp(\widehat{f_1}) 
				&\subseteq (E_W + N_2) - (E_V + N_V) \subseteq E_W - E_V + (N_1 + N_2) \\
				&\subseteq  E_W - E_V + N.
			\end{align*}
			As $\widehat{f_3}$ vanishes on $E_W -E_V + N$, it follows that $\widehat{F} = 0$ and hence $F = 0$. From \eqref{eq: spec_BVW_comp} we conclude that $\sigma_{f_2} \gamma_{f_3}(T) \alpha_{f_1}v = 0$ for all $f_2 \in J(E_W + N_2)$. This implies $\gamma_{f_3}(T) \alpha_{f_1}v = 0$, by \Fref{lem: annihilator_large_implies_v_zero}. Consequently, if $\lambda \in W^\prime$ is any continuous functional, then $\int_\R  f_1(t_1) \langle \lambda, \gamma_{f_3}(T) \alpha_{t_1} v\rangle dt = 0$. As the map $t \mapsto \langle \lambda, \gamma_{f_3}(T) \alpha_{t_1} v\rangle$ is continuous it follows that $\langle \lambda, \gamma_{f_3}(T) \alpha_{t_1} v\rangle = 0$ for all $t \in \R$. In particular $\langle \lambda, \gamma_{f_3}(T) v\rangle = 0$. As $W^\prime$ separates the points of $W$ by the Hahn-Banach Theorem \citep[Thm.\ I.3.4]{Rudin_FA}, it follows that $\gamma_{f_3}(T) v = 0$. As $v \in V$ was arbitrary we find that $\gamma_{f_3}(T) = 0$. We have thus shown that $I_0(\overline{E_W - E_V + N}) \subseteq \ker \gamma$. By \Fref{cor: pointwise_char_of_full_spec}, this is equivalent to $\Spec_\gamma(\B(V, W)) \subseteq \overline{E_W - E_V + N}$. Hence 
			$$\Spec_\gamma(\B(V, W)) \subseteq \bigcap_N \overline{E_W - E_V + N} = \overline{E_W - E_V}.\qedhere$$
		\end{enumerate}
	\end{proof}

	\noindent
	Recall from \Fref{sec: analytic_fns} that $P(V, W) = \prod_{k=0}^\infty P^k(V, W)$ is equipped with the product topology, where each $P^k(E,F)$ carries the topology of uniform convergence on compact sets. We will have need for the following result in \Fref{sec: pe_and_hol_ind} below:
	
	\begin{corollary}\label{cor: R_action_on_jets}
		Consider the setting of \Fref{prop: arv_spec_tensor_prod_cont_maps}. Assume that $V$ is Fr\'echet. Define the representation $\gamma$ of $\R$ on $P(V, W)$ by $\gamma_t(f)(v) := \sigma_t(f(\alpha_{-t}(v)))$. Then $\gamma$ is strongly continuous and pointwise polynomially bounded. Moreover, if $\Spec_\alpha(V) \subseteq (-\infty, 0]$, then 
		$$\inf \, \Spec_\gamma(P(V, W)) = \inf \, \Spec_{\sigma}(W) \in \{-\infty\} \cup \R.$$
	\end{corollary}
	\begin{proof}
		Let $n \in \N_{\geq 0}$. Notice that $\gamma$ leaves the homogeneous component $P^n(V, W) \subseteq P(V, W)$ invariant. Recall from \Fref{prop: isom_polys_mult_maps} and \Fref{prop: univ_property_proj_tensors} that 
		$$P^n(V, W) \cong \Sym^n(V, W) \subseteq \Mult(V^n, W) \cong \B(V^{\widehat{\otimes}n}, W)$$
		as locally convex vector spaces. The thus-obtained continuous linear embedding $\Phi_n : P^n(V, W) \hookrightarrow\B(V^{\widehat{\otimes}n}, W)$ is $\R$-equivariant when $\B(V^{\widehat{\otimes}n}, W)$ is equipped with the $\R$-action defined by $\widetilde{\gamma}_t(T) := \sigma_t T \alpha_{-t}$. By \Fref{prop: arv_spec_tensor_prod_cont_maps}, this action is strongly continuous and pointwise polynomially bounded. Consequently, also $\gamma$ is strongly continuous and pointwise polynomially bounded on $P^n(V, W)$. As $P(V, W)$ carries the product topology, the same holds for the $\R$-action $\gamma$ on $P(V, W)$.\\
		
		\noindent
		For the final statement, notice that $W = P^0(V, W) \subseteq P(V, W)$. By \Fref{prop: spectra_linear_map} it follows that $\Spec_\sigma(W) \subseteq\Spec_\gamma(P(V, W))$, thereby showing 
		$$\inf \, \Spec_\gamma(P(V, W)) \leq \inf \,\Spec_\sigma(W) $$
		Conversely, let $n \in \N$. As $\Phi_n$ is continuous, injective and $\R$-equivariant, we know that $\Spec_\gamma(P^n(V, W)) \subseteq \Spec_{\widetilde{\gamma}}\big(\B(V^{\widehat{\otimes}n}, W)\big)$, by \Fref{prop: spectra_linear_map}. Furthermore, using \Fref{prop: arv_spec_tensor_prod_cont_maps} we notice that $\Spec_{\alpha^{\otimes n }}(V^{\widehat{\otimes}n}) \subseteq (-\infty, 0]$ and therefore also that $\Spec_{\widehat{\gamma}}\big(\B(V^{\widehat{\otimes}n}, W)\big) \subseteq \Spec_{\sigma}(W) + [0, \infty) =: E$. Thus $\Spec_\gamma(P^n(V, W)) \subseteq E$ for any $n \in \N_{\geq 0}$. By \Fref{cor: pointwise_char_of_full_spec} this means that $\gamma_f\psi_n = 0$ for any $f \in I_0(E)$, $\psi_n \in P^n(V, W)$ and $n \in \N$. Consequently, $\gamma_f \psi = 0$ for any $f \in I_0(E)$ and $\psi \in P(V, W)$. So $I_0(E) \subseteq \ker\gamma$. By \Fref{cor: pointwise_char_of_full_spec}, this is equivalent with $\Spec_\gamma(P(V, W)) \subseteq E$. Hence $\inf \,\Spec_\sigma(W) = \inf E \leq \inf \, \Spec_\gamma(P(V, W))$.
	\end{proof}

	\noindent
	Finally, we record some useful facts regarding the space of smooth vectors of a unitary $G$-representation:
	
	\begin{proposition}\label{prop: properties_arveson_spec}
		Let $G$ be a regular locally convex Fr\'echet--Lie group. Let $\bm{d} \in \g$ and assume that the $\R$-action $\dot{\alpha}: \R \to \Aut(\g)$ defined by $\dot{\alpha}_t := \Ad(\exp(t\bm{d}))$ is polynomially bounded. Let $(\rho, \H_\rho)$ be a smooth unitary representation of $G$. Let $E \subseteq \R$ be a closed subset. Then the following assertions are valid:
		\begin{enumerate}
			\item The $\R$-representation $t \mapsto \restr{\rho(\exp(t\bm{d}))}{\H_\rho^\infty}$ on $\H_\rho^\infty$ is strongly continuous and pointwise polynomially bounded, where $\H_\rho^\infty$ is equipped with the strong topology.
			\item The operator $\pi(f) :=\int_\R f(t)\rho(\exp(t\bm{d})) dt$ on $\H_\rho$ leaves $\H_\rho^\infty$ invariant for any $f \in \S(\R)$.
			\item $\H_\rho^\infty(E) = \H_\rho(E) \cap \H_\rho^\infty$ 
			\item For any open subset $U \subseteq \R$, $\H_\rho^\infty(U)$ is dense in $\H_\rho(U)$.
			\item If $E_1, E_2 \subseteq \R$ are closed subsets then $d\rho(\g_\C(E_1))\H^\infty(E_2) \subseteq \H_\rho^\infty(E_1 + E_2)$.
		\end{enumerate}
	\end{proposition}
	\begin{proof}
		The second item follows from \citep[Thm.\ 2.3]{KH_Zellner_inv_smooth_vectors}, the fourth from \citep[Prop.\ 3.2]{KH_Zellner_inv_smooth_vectors} and the fifth from \citep[Thm.\ 3.1]{KH_Zellner_inv_smooth_vectors}. We provide an alternative proof of the second assertion and prove the first and third.\\

		\noindent
		By \citep[Prop.\ 3.19]{BasNeeb_ProjReps}, the locally convex space $\H_\rho^\infty$ is complete. Let $n \in \N$ and $B \subseteq \g^n$ be a bounded subset. Consider continuous seminorm 
		$$p_B(\psi) := \sup_{\xi \in B} \|d\rho(\xi_1\cdots \xi_n)\psi\|$$
		on $\H_\rho^\infty$. Let $\psi \in \H_\rho^\infty$. By \citep[Lem.\ 3.24]{BasNeeb_ProjReps}, the orbit map $\rho^\phi : G \to \H_\rho^\infty, g \mapsto \rho(g)\psi$ is smooth. It follows in particular that the $\R$-representation $t\mapsto \rho(\exp(t\bm{d}))$ on $\H_\rho^\infty$ is strongly continuous. It follows moreover that the multi-linear map $\g^n \to \H_\rho^\infty, (\xi_1, \ldots, \xi_n) \mapsto d\rho(\xi_1\cdots \xi_n)\psi$ is continuous. Using \Fref{prop: univ_property_proj_tensors}, we find that there exists a continuous seminorm $p$ on $\g$ such that $\|d\rho(\xi_1\cdots\xi_n)\psi\| \leq \prod_{k=1}^n p(\xi_k)$ for every $\xi \in \g^n$. Let $N \in \N$ and the $0$-neighborhood $U \subseteq \g$ be s.t.\ $C := \sup_{\xi \in U}\sup_{t \in \R}{1\over 1 + |t|^N}p(\dot{\alpha}_t(\xi)) < \infty$. As $B \subseteq \g^n$ is bounded, so is its projection $B_k \subseteq \g$ onto the $k^{\mathrm{th}}$ factor for every $k \in \{1, \ldots, n\}$. Thus there exists $s > 0$ such that $B_k \subseteq sU$ for all $1\leq k \leq n$. We obtain that 
		$$\sup_{\xi_k \in B_k} \sup_{t \in \R} {1\over 1 + |t|^N} p(\dot{\alpha}_t(\xi_k)) \leq s C$$
		for every $1\leq k\leq n$. Using that $\rho$ is unitary we find for all $t \in \R$ that
		$$ p_B(\rho(e^{-t\bm{d}})\psi) = \sup_{\xi \in B} \|d\rho(\dot{\alpha}_t(\xi_1)\cdots \dot{\alpha}_t(\xi_n)\psi\| \leq \sup_{\xi \in B}\prod_{k=1}^n p(\dot{\alpha}_t(\xi_k)) \leq C^n s^n (1 + |t|^N)^n.$$
		This implies that the $\R$-action $t \mapsto \rho(\exp(t\bm{d}))$ on $\H_\rho^\infty$ is pointwise polynomially bounded. As in \Fref{def: repr_of_schartz}, we conclude that $\pi^\infty(f) \psi := \int_\R f(t) \rho(\exp(t\bm{d}))\psi dt$ defines a representation $\pi^{\infty} : \S(\R) \to \L(\H_\rho^\infty)$ of $\S(\R)$ on $\H_\rho^\infty$ by linear operators. It is clear that $\pi^\infty(f) := \restr{\pi(f)}{\H_\rho^\infty}$, so this proves that $\pi(f)$ leaves $\H_\rho^\infty$ invariant for every $f \in \S(\R)$. It is further immediate from \Fref{def: arv_spec} that $\H_\rho^\infty(E) = \H_\rho(E) \cap \H_\rho^\infty$.
	\end{proof}

	\section{Positive energy representations and holomorphic induction}\label{sec: pe_and_hol_ind}
	
	\setcounter{subsection}{1}
	\setcounter{theorem}{0}
	
	\noindent
	In this section we explore the connection between positive energy representations and holomorphic induction. It is shown in \Fref{thm: pe_hol_ind} and \Fref{thm: ground_state_hol_ind} that these two are intimately related, as is to be expected from similar known results in more restrictive settings, such as \citep[Thm.\ 11.1.1]{Segal_Loop_Groups}, \citep[Sec.\ 4]{Neeb_hol_reps} and \citep[Thm.\ 6.1]{Neeb_semibounded_hilbert_loop}. This is used to transfer various results from holomorphic induction to the context of positive energy representations, under suitable assumptions. Before proceeding to the main results, let us clarify the setting and make some preliminary observations.
	
	\subsubsection{Notation and preliminary observations}

	\noindent
	Let $G$ be a connected regular BCH Fr\'echet--Lie group with Lie algebra $\g$. Let $\alpha : \R \to \Aut(G)$ be a homomorphism having a smooth action $\R \times G \to G$ and let $\dot{\alpha}$ be the corresponding $\R$-action on $\g_\C$, defined by $\dot{\alpha}_s(\xi) := \bm{\mrm{L}}(\alpha_s)\xi := \restr{d\over dt}{t=0}\alpha_s(e^{t \xi})$ for $s \in \R$. Let $D \xi := \restr{d\over ds}{s=0}\dot{\alpha}_s(\xi)$ be the corresponding derivation on $\g_\C$. Assume that $\dot{\alpha}$ has polynomial growth, in the sense of \Fref{def: polynomial_bdd_action}. Define the Lie group $G^\sharp := G \rtimes_\alpha \R$, which has Lie algebra $\g^\sharp := \g\rtimes_D \R \bm{d}$, where we have written $\bm{d} := 1\in \R \subseteq \g^\sharp$ for the standard basis element. Then $G^\sharp$ is again a connected regular Fr\'echet--Lie group, using \citep[Thm.\ V.I.8]{neeb_towards_lie}, but not necessarily BCH.\\
	
	\noindent
	As $\dot{\alpha}$ is assumed to have polynomial growth, we can define the Arveson spectral subspaces of $\g_\C$ as in \Fref{def: arv_spec}. If $E \subseteq \R$ is any subset, we write $\g_\C(E)$ for the spectral subspace of $\g_\C$ associated to $E$. Define $\h_\C := \ker D\subseteq \g_\C(\{0\})$, $\h := \h_\C \cap \g$ and
	$$ \n_- := \overline{\bigcup_{\delta > 0} \g_\C((-\infty, -\delta])}, \qquad \n_+ := \overline{\bigcup_{\delta > 0} \g_\C([\delta, \infty))}. $$
	\noindent
	We assume that $(\g_\C, \alpha)$ satisfies the so-called \textit{splitting condition}, meaning that 
	$$\g_\C = \n_- \oplus \h_\C \oplus \n_+.$$
	Define $\b_{\pm} := \h_\C \oplus \n_\pm \subseteq \g_\C$. Let $H := (G^\alpha)_0 \subseteq G$ be the connected subgroup of $\alpha$-fixed points in $G$. Let us first establish that the assumptions on $H$, $\n_\pm$ and $\h_\C$ made in \Fref{sec: hol_induced} are presently satisfied.
	
	\begin{lemma}
		$H$ is a locally exponential Lie subgroup of $G$ with Lie algebra $\h$. 
	\end{lemma}
	\begin{proof}
		Since $G$ is locally exponential, we can find a $0$-neighborhood $U_\g \subseteq \g$ s.t.\ $\exp_G$ restricts to a diffeomorphism on $U_\g$. Let $\xi \in U_\g$ arbitrary. Using the fact that $\alpha_t(\exp_G(\xi)) = \exp_G(\dot{\alpha}_t(\xi))$ for all $t \in \R$, observe that $\xi \in \ker D \iff \exp_G(\xi)\in G^\alpha$. This implies that $\exp_G(U_\g \cap \h) = \exp_G(U_\g) \cap H$. We also obtain that $\h = \set{\xi \in \g \st \exp_G(\R \xi) \subseteq H}$. It follows that $H$ is a locally exponential Lie subgroup with Lie algebra $\h$, by \citep[Thm.\ IV.3.3]{neeb_towards_lie}.
	\end{proof}
	
	\begin{lemma}\label{cor: triangular_decomp_subalgebras}
		The subspaces $\n_\pm$, $\h_\C$ and $\b_\pm$ are Lie subalgebras of $\g_\C$ and we have $[\h_\C, \n_\pm] \subseteq \n_\pm$. Moreover, $\Ad_H(\n_\pm) \subseteq \n_\pm$. Finally, $\theta(\n_\pm) \subseteq \n_\mp$ and $\theta(\h_\C) \subseteq \h_\C$.
	\end{lemma}
	\begin{proof}
		The Lie bracket $[\--, \--] : \g_\C \times \g_\C \to \g_\C$ is bilinear, continuous and $\R$-equivariant, meaning that $\dot{\alpha}_s([\xi, \eta]) = [\dot{\alpha}_s(\xi), \dot{\alpha}_s(\eta)]$ for all $s \in \R$ and $\xi, \eta \in \g_\C$. Using \Fref{prop: arv_spec_subs_additive} we obtain for any two closed subsets $E_1, E_2 \subseteq \R$ that $[\g_\C(E_1), \g_\C(E_2)] \subseteq \g_\C(E_1 + E_2)$. This implies that $\n_\pm$, $\h_\C$ and $\b_\pm$ are Lie subalgebras of $\g_\C$ and that $[\h_\C, \n_\pm] \subseteq \n_\pm$. We next show that $\Ad_H(\n_\pm) \subseteq \n_\pm$. Let $h \in H$. Then $\dot{\alpha}_s$ and $\Ad_h$ commute for any $s \in \R$, so $\Ad_h : \g_\C \to \g_\C$ is a continuous equivariant linear map. It follows using \Fref{prop: spectra_linear_map} that $\Ad_h(\g_\C(E)) \subseteq \g_\C(E)$ for any closed subset $E \subseteq \R$. Hence $\Ad_H(\n_\pm) \subseteq \n_\pm$. Let us next consider the conjugation $\theta$. Using that $\dot{\alpha}_t$ commutes with $\theta$ for any $t \in \R$, observe that $\theta \dot{\alpha}_{f}\theta = \dot{\alpha}_{\overline{f}}$ for any $f \in \S(\R)$. Consequently, $\S(\R)_{\theta(\xi)} = \set{\overline{f} \st f \in \S(\R)_\xi}$. Using that $\F(\overline{f})(p) = \overline{\F{f}}(-p)$ for $p \in \R$, we obtain for any $\xi \in \g_\C$ that 
		$$\Spec_{\dot{\alpha}}(\theta(\xi)) = h(\S(\R)_{\theta(\xi)}) = - h(\S(\R)_\xi) = - \Spec_{\dot{\alpha}}(\xi).$$
		We therefore have $\theta(\g_\C(E)) = \g_\C(-E)$ for any closed $E \subseteq \R$. This implies that $\theta(\n_\pm) \subseteq \n_\mp$. Since $\h_\C = \ker D$ and $\theta$ commutes with $D$, we also have $\theta(\h_\C) \subseteq \h_\C$.
	\end{proof}

	\noindent
	As the Lie group $G^\sharp = G \rtimes_\alpha \R$ need not be analytic, we only have access to the analytic structure of $G$:
	
	\begin{definition}\label{def: subspaces}
		If $(\rho, \H_\rho)$ is a unitary representation of $G^\sharp$, we write $\H_\rho^{\omega_G}$ for the space of $G$-analytic vectors in $\H_\rho$. We further define 
		$$\H_\rho^{\infty, \n_-} := \set{\psi \in \H_\rho^\infty \st d\rho(\n_-)\psi = \{0\}}$$
		and we write $V(\rho) := \overline{\H_\rho^{\infty, \n_-}}$ for its closure.
	\end{definition}
	
	\noindent
	Let us first clarify that $V(\rho)$ can equivalently be defined as the closure of the set of $G$-smooth vectors in $\H_\rho$ that are killed by $\n_-$, as opposed to the $G^\sharp$-smooth ones:
	
	\begin{lemma}\label{lem: Vrho_G_smooth}
		Let $\rho$ be a unitary $G^\sharp$-representation. Let $W(\rho)\subseteq \H_\rho$ be the closed linear subspace generated by the set of $G$-smooth vectors in $\H_\rho$ that are killed by $d\rho(\n_-)$. Then $W(\rho) = V(\rho)$.
	\end{lemma}
	\begin{proof}
		It is trivial that $V(\rho) \subseteq W(\rho)$. Let $\psi \in \H_\rho$ be a $G$-smooth vector s.t.\ $d\rho(\n_-)\psi = \{0\}$. Let $f \in C^\infty_c(\R)$ and define $\pi_f \psi := \int_\R f(t)\rho(t)\psi dt \in \H_\rho$. Then $\pi_f \psi$ is a smooth vector for $G^\sharp$, e.g.\ using \citep[Lem.\ A.4]{KH_Zellner_inv_smooth_vectors}. Let $\xi \in \n_-$. Then $\dot{\alpha}_{-t}(\xi) \in \n_-$ and hence $d\rho(\dot{\alpha}_{-t}(\xi)) \psi = 0$ for every $t \in \R$. Using \citep[Lem.\ A.4]{KH_Zellner_inv_smooth_vectors} to differentiate under the integral, we obtain:
		$$d\rho(\xi)\pi_f \psi = \int_\R f(t)d\rho(\xi)\rho(t)\psi dt = \int_\R f(t) \rho(t)d\rho(\dot{\alpha}_{-t}(\xi)) \psi dt = 0.$$
		So $\pi_f \psi \in \H_\rho^{\infty, \n_-}$ for any $f \in C^\infty_c(\R)$. Approximating $\psi$ by vectors of the form $\pi_f \psi$, we conclude that $\psi \in V(\rho)$. So $V(\rho) = W(\rho)$.
	\end{proof}
	
	\noindent
	To keep a uniform notation for $G$- and $G^\sharp$-representations, we complement \Fref{def: subspaces} with:
	
	\begin{definition}\label{def: subspaces_2}
		If $(\rho, \H_\rho)$ is a smooth unitary representation of $G$, we write $\H_\rho^{\omega_G} := \H_\rho^\omega$ for the space of $G$-analytic vectors in $\H_\rho$. Define 
		$$\H_\rho^{\infty, \n_-} := \set{\psi \in \H_\rho^\infty \st d\rho(\n_-)\psi = \{0\}}$$
		and let $V(\rho) := \overline{\H_\rho^{\infty, \n_-}}$ denote its closure.\\
	\end{definition}

	\noindent
	Let us proceed with the task of relating the positive energy condition with holomorphic induction. Notice that $V(\rho)\subseteq \H_\rho$ is $H \times \R$-invariant for any smooth unitary $G$-representation $\rho$, because $\n_-$ is invariant under $\dot{\alpha}_t$ and $\Ad_h$ for any $t \in \R$ and $h \in H$. The following makes use of the notation specified in \Fref{def: notation_hol_induced}:
	
	\newpage
	\begin{theorem}\label{thm: pe_hol_ind}
		Let $(\rho, \H_\rho)$ be a smooth unitary representation of $G^\sharp$ and let $\sigma$ be the unitary representation of $H \times \R$ on $V_\sigma := V(\rho)$ defined by $\sigma(h,t) := \restr{\rho(h,t)}{V(\rho)}$. The following assertions are equivalent:
		\begin{enumerate}
			\item $\rho$ is of positive energy at $\bm{d}$, $V(\rho)$ is cyclic for $\rho$ and $V(\rho) \cap \H_\rho^{\omega_G}$ is dense in $V(\rho)$.
			\item $\sigma$ is of positive energy at $\bm{d}$ and $\restr{\rho}{G} = \HolInd_H^G(\restr{\sigma}{H})$.
		\end{enumerate}
		If these conditions are satisfied, then $\inf\,\Spec(-i \overline{d\rho(\bm{d})}) = \inf\, \Spec(-i \overline{d\sigma(\bm{d})}) \geq 0$.
	\end{theorem}
	
	\noindent
	We start the proof of \Fref{thm: pe_hol_ind} with two lemmas:
	
	\begin{lemma}\label{lem: subspace_equals_Vrho}
		Let $W \subseteq V(\rho)$ be a $H$-invariant closed linear subspace that is cyclic for $G$ and contains a dense set of $G$-analytic vectors. Then $W = V(\rho)$.
	\end{lemma}
	\begin{proof}
		Let $W^\perp$ be the orthogonal complement of $W$ in $V(\rho)$, so $V(\rho) = W \oplus W^\perp$ as unitary $H$-representations. It suffices to show that $W^\perp \perp \rho(G) W$. Define $W^{\omega_G} := W \cap \H_\rho^{\omega_G}$. Let $w \in W^{\omega_G}$ and $v \in W^\perp\subseteq V(\rho)$. Consider the analytic function $f : G \to \C,\; f(g) := \langle v, \rho(g)w\rangle$. Let $E_0 : \mc{U}(\g_\C) \to \mc{U}(\h_\C)$ be defined as in \Fref{def: cond_exp_univ_env_alg}. As $d\rho(\n_-)$ kills both $\H_\rho^{\infty, \n_-}$ and $W^{\omega_G}$, observe that $\langle v, d\rho(x) w\rangle = \langle v, d\sigma(E_0(x))w\rangle = 0$ for any $x \in \mc{U}(\g_\C)$. It follows that $j^\infty_e(f) = 0$. As $G$ is connected and $f$ is analytic, we conclude using \Fref{prop: identity_theorem} that $f = 0$. Because $W^{\omega_G}$ is dense in $W$, it follows that $W^\perp \perp \rho(G) W$.\qedhere
	\end{proof}
	
	\begin{lemma}\label{lem: inf_subs_G_inv}
		Let $\mcD \subseteq \H_\rho^\omega$ be a linear subspace. Then $\overline{d\rho(\mc{U}(\g_\C))\mcD}$ is the closed $G$-invariant linear subspace of $\H_\rho$ generated by $\mcD$.
	\end{lemma}
	\begin{proof}
		Define $\mc{F} := \overline{d\rho(\mc{U}(\g_\C))\mcD}$ and let $\mc{F}^\prime$ denote the closed $G$-invariant linear subspace generated by $\mcD$. The inclusion $\mc{F} \subseteq \mc{F}^\prime$ is clear. Let $v \in \mc{F}^\perp \subseteq \H_\rho$ and take $\psi \in \mcD$. Consider the analytic function $f : G \to \C$ defined by $f(g) := \langle v, \rho(g)\psi\rangle$. Notice for any $x \in \mc{U}(\g_\C)$ that $\langle v, d\rho(x)\psi\rangle = 0$, because $d\rho(x)\psi \in \mc{F}$. It follows that $j^\infty_e(f) = 0$. As $G$ is connected and $f$ is analytic, we conclude using \Fref{prop: identity_theorem} that $f = 0$. We therefore find that $v \perp \rho(G)\mc{D}$, so in fact $v \perp \mc{F}'$. Hence $\mc{F}^\perp \subseteq (\mc{F}')^\perp$, which is equivalent to $\mc{F}' \subseteq \mc{F}$.
	\end{proof}

	\begin{proof}[Proof of \Fref{thm: pe_hol_ind}:]
		Define $\mcD_\chi := V(\rho) \cap \H_\rho^{\omega_G}$. Assume that $(1)$ holds true. Then in particular, $\sigma$ is of positive energy at $\bm{d}$. Let $\chi : \b_- \to \L(\mcD_\chi)$ be the trivial extension of $d\sigma$ to $\b_-$ with domain $\mcD_\chi$. By definition of $V(\rho)$, $\mcD_\chi$ is killed by $d\rho(\n_-)$. The conditions for $V_\sigma$ in \Fref{thm: hol_induced_equiv_char} are satisfied for the $(H, \b_-)$-extension pair $(\restr{\sigma}{H}, \chi)$, so $(2)$ follows from \Fref{thm: hol_induced_equiv_char}.\\
		
		\noindent
		Conversely, assume that $\restr{\rho}{G} = \HolInd_H^G(\restr{\sigma}{H})$ and that $\sigma$ is of p.e.\ at $\bm{d}$. It follows from \Fref{thm: hol_induced_equiv_char} that there is a $H$-invariant closed linear subspace $W \subseteq \H_\rho$ s.t. $W$ is cyclic for $\rho$ and $W \cap \H_\rho^{\omega_G}$ is both dense in $W$ and killed by $d\rho(\n_-)$. The last condition implies using \Fref{lem: Vrho_G_smooth} that $W \subseteq V(\rho)$. By \Fref{lem: subspace_equals_Vrho} we obtain that $W = V(\rho)$. To see that $(1)$ holds true, it only remains to show that $\rho$ is of positive energy at $\bm{d}$. Define 
		$$\Phi : \H_\rho^\infty \to C^\infty(G, V_\sigma)^H, \qquad \Phi_\psi(g) := p_V \rho(g)^{-1}\psi \qquad (\psi \in \H_\rho^\infty),$$
		where $p_V : \H_\rho \to V(\rho)$ is the orthogonal projection. Using the exponential map as a local chart, identify $J^\infty_e C^\infty(G, V_\sigma) \cong P(\g_\C, V_\sigma)$ $G$-equivariantly. Let $A$ denote the composition 
		$$A : \H_\rho^\infty \xrightarrow{\Phi} C^\infty(G, V_\sigma)^H \xrightarrow{j^\infty_e} P(\g_\C, V_\sigma)\xrightarrow{restr} P(\n_-, V_\sigma).$$
		Observe that
		\begin{align*}
			\Phi_{\rho(t)\psi}(g) &= p_V \rho(g)^{-1}\rho(t)\psi = p_V \rho(t)\rho(\alpha_{-t}(g))^{-1}\psi = \sigma(t) p_V \rho(\alpha_{-t}(g))^{-1}\psi \\
			&= \sigma(t)\Phi_\psi(\alpha_{-t}(g)).
		\end{align*}
		Consequently, $A$ is $\R$-equivariant if we equip $P(\n_-, V_\sigma)$ with the $\R$-action defined by $(\nu_t f)(\xi) := \sigma(t)f(\dot{\alpha}_{-t}(\xi))$ for $t \in \R$ and $f \in P^n(\n_-, V_\sigma)$. Equip $\H_\rho^\infty$ with the strong topology (cf.\ \Fref{def: smooth_analytic_vectors}), with respect to which it is complete because $G$ is a regular Fr\'echet--Lie group \citep[Prop.\ 3.19]{BasNeeb_ProjReps}. Recall that $P(\n_-, V_\sigma) = \prod_{n=0}^\infty P^n(\n_-, V_\sigma)$ carries the product topology and each $P^n(\n_-, V_\sigma)$ carries the topology of uniform convergence on compact sets. We show that $A$ is continuous with respect to these topologies. For any $\psi \in \H_\rho^\infty$, let $f_\psi \in C^\infty(G, \H_\rho), \; f_\psi(g) := \rho(g)\psi$ denote the orbit map. Using that $\rho$ is unitary, observe that the linear map $\H_\rho^\infty \to C^\infty(G, \H_\rho), \psi \mapsto f_\psi$ is continuous w.r.t.\ the smooth compact-open topology on $C^\infty(G, \H_\rho)$. This implies that $\Phi$ is continuous. As $j^\infty_e$ is continuous by \Fref{prop: jet_vanishes_then_function_trivial}, the continuity of $A$ follows. We remark further that the $\R$-representation $t \mapsto \rho(t)$ on $\H_\rho^\infty$ is strongly continuous and pointwise polynomially bounded by \Fref{prop: properties_arveson_spec}, so that its Arveson spectrum can be defined according to \Fref{def: arv_spec}. Similarly, because the $\R$-actions on $\n_-$ and $V_\sigma$ both have polynomially growth and are strongly continuous, it follows from \Fref{cor: R_action_on_jets} that the $\R$-action $\nu$ on $P(\n_-, V_\sigma)$ is strongly continuous and pointwise polynomially bounded. Since $\n_-$ and $V_\sigma$ have non-positive and non-negative spectrum, respectively (relative to the $\R$-actions $\dot{\alpha}_t$ and $\sigma(t)$, respectively), we further obtain from \Fref{cor: R_action_on_jets} and \Fref{ex: spectra_Hilbert_space_case} that 
		$$\inf\, \Spec_\nu\big(P\big(\n_-, V_\sigma\big)\big) = \inf\, \Spec(V_\sigma) = \inf \Spec(-i \overline{d\sigma(\bm{d})}) \geq 0$$
		We show next that $A$ is injective. Let $\psi \in \H_\rho^\infty$ and suppose that $A(\psi) = 0$. Then $p_V d\rho(\mc{U}(\n_-))\psi = \{0\}$, which implies $\psi \perp \overline{d\rho(\mc{U}(\n_+))\mcD_\chi}$. Since $\mcD_\chi$ is $d\rho(\b_-)$-invariant, notice that $\overline{d\rho(\mc{U}(\n_+))\mcD_\chi} = \overline{d\rho(\mc{U}(\g_\C))\mcD_\chi}$ by the PBW Theorem. By \Fref{lem: inf_subs_G_inv}, this is the closed $G$-invariant subspace of $\H_\rho$ generated by $\mcD_\chi$, which equals all of $\H_\rho$ because $\mcD_\chi$ is dense in $V(\rho)$ and $V(\rho)$ is cyclic for $\rho$. Thus $\psi \perp \H_\rho$ and so $\psi = 0$. Hence $A$ is injective, continuous and $\R$-equivariant. It follows by \Fref{prop: spectra_linear_map} that $\Spec(\H_\rho^\infty) \subseteq \Spec_\nu\big(P\big(\n_-, V_\sigma\big)\big)$, where we consider the $\R$-action $t \mapsto \rho(t)$ on $\H_\rho^\infty$. Thus
		$$ \inf \Spec(\H_\rho^\infty) \geq \inf\, \Spec_\nu\big(P\big(\n_-, V_\sigma\big)\big) = \inf \Spec(-i \overline{d\sigma(\bm{d})}), $$
		Notice that $\H_\rho$ and $\H_\rho^\infty$ have the same spectrum, because $\H_\rho^\infty$ is dense in $\H_\rho$. So 
		$$\inf\,\Spec(-i \overline{d\rho(\bm{d})}) = \inf\, \Spec(\H_\rho) = \inf\, \Spec(\H_\rho^\infty) \geq \inf\, \Spec(-i \overline{d\sigma(\bm{d})}) \geq 0.$$
		Thus, $\rho$ is of positive energy at $\bm{d}$. So $(1)$ holds true. Finally, the inclusion $V(\rho)~\subseteq~\H_\rho$ is $\R$-equivariant, so by \Fref{prop: spectra_linear_map} we also have the reverse inequality $\inf\,\Spec(-i \overline{d\rho(\bm{d})}) \leq \inf\, \Spec(-i \overline{d\sigma(\bm{d})})$.\\
	\end{proof}
	
	\noindent
	Let us state some important immediate consequences of \Fref{thm: pe_hol_ind}.
	
	\begin{lemma}\label{lem: pe_orth_proj_inner}
		Let $(\rho, \H_\rho)$ be a smooth unitary representation of $G$. Let $q_V\in\B(\H_\rho)$ denote the orthogonal projection onto $V(\rho)$. Then $q_V\in\rho(G)^{\prime \prime}$.
	\end{lemma}
	\begin{proof}
		Let $T \in \rho(G)^\prime = \B(\H_\rho)^G$. Then $T \H_\rho^\infty \subseteq \H_\rho^\infty$ and 
		$$d\rho(\n_-)T \H_\rho^{\infty, \n_-} = T d\rho(\n_-)\H_\rho^{\infty, \n_-} \subseteq \{0\}.$$
		Thus $T \H_\rho^{\infty, \n_-} \subseteq \H_\rho^{\infty, \n_-}$. It follows that $T V(\rho) \subseteq V(\rho)$. As $\B(\H_\rho)^G$ is $\ast$-closed, we have also shown that $T^\ast V(\rho) \subseteq V(\rho)$. Thus $q_V T = Tq_V$, and so $q_V\in\rho(G)^{\prime \prime}$. 
	\end{proof}
	
	\begin{corollary}\label{cor: pe_hol_ind_iso_commutants}
		Suppose that the unitary $G^\sharp$-representation $\rho$ satisfies the equivalent conditions of \Fref{thm: pe_hol_ind}. Then $T \mapsto \restr{T}{V(\rho)}$ defines isomorphisms of von Neumann algebras 
		$$\B(\H_\rho)^G \cong \B(V(\rho))^H \qquad\text{ and } \qquad \B(\H_\rho)^{G^\sharp} \cong \B(V(\rho))^{H \times \R}.$$
	\end{corollary}
	\begin{proof}
		That $T \mapsto \restr{T}{V(\rho)}$ defines an isomorphism $\B(\H_\rho)^G \to \B(V(\rho))^H$ is immediate from \Fref{lem: pe_orth_proj_inner} and \Fref{thm: commutant_trivial_extension}. Consequently, it suffices to show that any $T \in \B(\H_\rho)^{G}$ with $\restr{T}{V(\rho)} \in \B(V(\rho))^{H \times \R}$ automatically commutes with the $\R$-action $t \mapsto \rho(t)$ on $\H_\rho$. Consider such $T$ and let $t \in \R$. Then
		\begin{align}
			\begin{split}
				\rho(t)T\rho(g)v &= \rho(t)\rho(g)Tv = \rho(\alpha_t(g))\rho(t)Tv = \rho(\alpha_t(g))T\rho(t)v \\
				&= T\rho(\alpha_t(g))\rho(t)v = T \rho(t)\rho(g)v	
			\end{split}
		\end{align}
		for any $g \in G$ and $v \in V(\rho)$. As $V(\rho)$ is cyclic for $G$, it follows that $T\rho(t) = \rho(t)T$ for all $t \in \R$.
	\end{proof}
	
	\begin{corollary}\label{cor: pe_uniqueness}
		Suppose that the unitary $G^\sharp$-representations $\rho_1$ and $\rho_2$ satisfy the equivalent conditions of \Fref{thm: pe_hol_ind}. Then the following assertions are valid:
		\begin{enumerate}
			\item If $V(\rho_1) \cong V(\rho_2)$ as unitary $H$-representations, then $\restr{\rho_1}{G}\cong \restr{\rho_2}{G}$.
			\item If $V(\rho_1) \cong V(\rho_2)$ as unitary $H \times \R$-representations, then $\rho_1 \cong \rho_2$.
		\end{enumerate} 
	\end{corollary}
	\begin{proof}
		The first assertion is immediate from \Fref{thm: uniqueness}. Assume that the unitary $u : V(\rho_1) \to V(\rho_2)$ intertwines the $H \times \R$-actions. Consider the unitary $G^\sharp$-representation $\rho = \rho_1 \oplus \rho_2$ on $\H_{\rho_1} \oplus \H_{\rho_2}$. Notice that 
		$$V(\rho_1 \oplus \rho_2) = V(\rho_1) \oplus V(\rho_2) =: W$$
		and that $\rho$ satisfies the conditions in \Fref{thm: pe_hol_ind}. Define $S \in \B(W)^{H \times \R}$ by $S(v_1, v_2) := (0, uv_1)$. By \Fref{cor: pe_hol_ind_iso_commutants}, there is some $T \in  \B(\H_{\rho_1} \oplus \H_{\rho_2})^{G^\sharp}$ s.t.\ $\restr{T}{W} = S$. As $V(\rho_1)$ and $V(\rho_2)$ are cyclic for $G$ in $\H_{\rho_1}$ and $\H_{\rho_2}$, respectively, $T$ is of the form $T(\psi_1, \psi_2) = (0, U\psi_1)$ for some $U : \H_{\rho_1} \to \H_{\rho_2}$ intertwining the $G^\sharp$-actions. Notice that $S^\ast S$ and $SS^\ast$ are the orthogonal projections onto $V(\rho_1)$ and $V(\rho_2)$, respectively. By \Fref{cor: pe_hol_ind_iso_commutants} it follows that $T^\ast T$ and $TT^\ast$ are the orthogonal projections onto $\H_{\rho_1}$ and $\H_{\rho_2}$, respectively. This implies that $U$ is unitary.
	\end{proof}
	
	\subsubsection{The spectral gap condition}
	
	\noindent
	We will next assume that the so-called spectral gap condition is satisfied, a stronger variant of the splitting condition. We show that in this case, $V(\rho)$ is always cyclic for positive energy representations.
	
	\begin{definition}
		We say that \textit{the spectral gap} (SG) condition is satisfied if there is some $\delta > 0$ such that
		\begin{equation}\label{eq: SG_condition}
			\g_\C = \g_\C((-\infty, -\delta]) \oplus \h_\C \oplus \g_\C([\delta, \infty)).	
		\end{equation}
	\end{definition}
	
	\noindent
	If $\rho$ is a smooth unitary representation of $G^\sharp$ and $E \subseteq \R$ is a subset, we write $\H_\rho(E)$ and $\H_\rho^\infty(E)$ for the closed spectral subspaces associated to the $\R$-representation $t \mapsto \rho(t)$ on $\H_\rho$ and $\H_\rho^\infty$, respectively, where we recall that the $\R$-action on $\H_\rho^\infty$ is pointwise polynomially bounded by \Fref{prop: properties_arveson_spec}. Recall also from \Fref{prop: properties_arveson_spec} that $H_\rho^\infty(E) = \H_\rho(E)\cap \H_\rho^\infty$. 
	
	\begin{lemma}\label{lem: V_nonzero}
		Assume that $\textrm{(SG)}$ is satisfied. Let $\rho$ be a smooth unitary representation of $G^\sharp$ which is of positive energy at $\bm{d} \in \g^\sharp$. If $\H_\rho \neq \{0\}$, then $V(\rho) \neq \{0\}$.
	\end{lemma}
	\begin{proof}
		Let $\delta > 0$ be such that \eqref{eq: SG_condition} is satisfied. Let $E_0 := -i\inf \Spec(d\rho(\bm{d}))$. Let $0 < \epsilon < \delta$ and define $U := [E_0, E_0 + \epsilon)$. By definition of $E_0$, the spectral subspace $\H_\rho(U)$ is nonzero. By \Fref{prop: properties_arveson_spec}$(4)$, $\H_\rho^\infty(U)$ is dense in $\H_\rho(U) = \H_\rho((E_0 -\epsilon, E_0 + \epsilon))$. Since $\H_\rho(U)$ is nonzero, so is $\H_\rho^\infty(U)$. By the last point in \Fref{prop: properties_arveson_spec}, we obtain that $d\rho(\n_-)\H^\infty(U) \subseteq \H_\rho^\infty((-\infty, E_0 +\epsilon - \delta]) = \{0\}$. Hence $\H^\infty(U) \subseteq \H_\rho^{\infty, \n_-} \subseteq V(\rho)$. It follows that $V(\rho) \neq \{0\}$.
	\end{proof}
	
	\begin{proposition}\label{prop: Vrho_cyclic}
		Assume that $\textrm{(SG)}$ is satisfied. Let $\rho$ be a smooth unitary representation of $G^\sharp$ which is of positive energy at $\bm{d} \in \g^\sharp$. Then $V(\rho)$ is cyclic for $\rho$.
	\end{proposition}
	\begin{proof}
		Let $W$ be the closed $G^\sharp$-invariant subspace of $\H_\rho$ generated by $V(\rho)$. Then $W^\perp$ carries a smooth representation of $G^\sharp$ that is of positive energy at $\bm{d} \in \g^\sharp$. From $(W^\perp)^{\infty, \n_-} \subseteq \H_{\rho}^{\infty, \n_-} \subseteq V(\rho)$, we obtain that $(W^\perp)^{\infty, \n_-} \subseteq W^\perp \cap V(\rho) = \{0\}$. Using \Fref{lem: V_nonzero} we conclude that $W^\perp = \{0\}$, so $W = \H_\rho$.
	\end{proof}
	
	\subsubsection{Ground state representations}
	
	\noindent
	We now shift our attention to ground state representations, where \Fref{thm: pe_hol_ind} simplifies somewhat. If $\rho$ is a smooth unitary representation of $G^\sharp$ on $\H_\rho$, we define $\H_\rho(0) = \ker \overline{d\rho(\bm{d})}$, $\H_\rho^\infty(0) := \H_\rho(0)\cap \H_\rho^\infty$ and $\H_\rho^{\omega_G}(0) := \H_\rho(0)\cap \H_\rho^{\omega_G}$. It will be convenient to make the following definition:
	
	\begin{definition}
		Let $(\rho, \H_\rho)$ be a smooth unitary representation of $G^\sharp$ that is ground state at $\bm{d} \in \g^\sharp$. We say that $\rho$ is \textit{analytically ground state} at $\bm{d} \in \g^\sharp$ if $\H_\rho^{\omega_G}(0)$ is dense in $\H_\rho(0)$.
	\end{definition}
	
	\begin{lemma}\label{lem: killed_vectors_ground_states}
		Let $(\rho, \H_\rho)$ be a smooth unitary representation of $G^\sharp$ that is of positive energy at $\bm{d} \in \g^\sharp$. Then $\H_\rho^\infty(0) \subseteq\H_\rho^{\infty, \n_-}$. If $\rho$ is analytically ground state at $\bm{d}$, then $V(\rho) = \H_{\rho}(0)$.
	\end{lemma}
	\begin{proof}
		Using \Fref{prop: arv_spec_subs_additive}, we obtain that 
		$$d\rho(\g_\C((-\infty, -\delta]))\H_{\rho}^\infty(0) \subseteq \H_\rho^{\infty}((-\infty, -\delta]) = \{0\}, \qquad \forall \delta > 0.$$
		Hence $\H_\rho^\infty(0) \subseteq V(\rho)$. If $\rho$ is analytically ground state at $\bm{d}$, the preceding implies $\H_\rho(0) \subseteq V(\rho)$. Using \Fref{lem: subspace_equals_Vrho} we conclude that $\H_{\rho}(0) = V(\rho)$.
	\end{proof}
	
	\noindent
	The following clarifies the tight relation between unitary representations of $G\rtimes_\alpha \R$ that are analytically ground state at $\bm{d} \in \g^\sharp$ and holomorphic induction:
	
	\begin{theorem}\label{thm: ground_state_hol_ind}
		Consider the setting of \Fref{thm: pe_hol_ind}. The following assertions are equivalent:
		\begin{enumerate}
			\item $\rho$ is analytically ground state at $\bm{d} \in \g^\sharp$.
			\item $\restr{\rho}{G} = \HolInd_H^G(\restr{\sigma}{H})$ and $V(\rho) = \H_\rho(0)$.
		\end{enumerate}
	\end{theorem}
	\begin{proof}
		Assume that $(1)$ is valid. Then $V(\rho) = \H_\rho(0)$, by \Fref{lem: killed_vectors_ground_states}, so $(2)$ follows from \Fref{thm: pe_hol_ind}. Suppose conversely that $(2)$ holds true. \Fref{thm: pe_hol_ind} then yields that $\rho$ is of positive energy at $\bm{d}$, that $\H_\rho(0)$ is cyclic for $G$ and that $\H_\rho^{\omega_G}(0)$ is dense in $\H_\rho(0)$. Thus $(1)$ is valid.
	\end{proof}
	
	\noindent
	Let us complement \Fref{thm: ground_state_hol_ind} with the following observation:
	
	\begin{proposition}
		Let $\rho$ be a smooth unitary p.e.\ representation of $G$. Let $\rho_0$ denote its minimal positive extension to $G^\sharp$. Assume that $\rho_0$ satisfies the equivalent conditions of \Fref{thm: pe_hol_ind}. If $\rho$ is irreducible, then $\rho_0$ is analytically ground state and $V(\rho) = \H_{\rho_0}(0)$.
	\end{proposition}
	\begin{proof}
		Define $V_\sigma := V(\rho)$, $\sigma_0(h,t) := \restr{\rho_0(h,t)}{V_\sigma}$ and $\sigma(h) := \restr{\rho(h)}{V_\sigma}$. Let $\mc{M} := \rho(G)^{\prime \prime}$ be the von Neumann algebra generated by $\rho(G)$. By \Fref{cor: pe_hol_ind_iso_commutants} we have $\B(V_\sigma)^H = \C \, \id_{V_\sigma}$. Thus $\sigma_0(t) = e^{itp} \,\id_{V_\sigma}$ for some $p \in \R$. As $\sigma_0$ is of positive energy, we have $p \geq 0$. By \Fref{thm: pe_hol_ind} we know that $\inf\,\Spec(-i \overline{d\rho_0(\bm{d})}) = p$, so $\rho_1(t) := \rho_0(t)e^{-itp}$ defines a positive inner implementation of $\R \to \Aut(\M),\, t \mapsto \Ad(\rho_0(t))$. As $\rho_0(t)$ is minimal, it follows that $p \leq 0$. Hence $p=0$, and so $V(\rho) \subseteq \H_{\rho_0}(0)$. On the other hand, we know from \Fref{lem: killed_vectors_ground_states} that $\H_{\rho_0}^\infty(0) \subseteq V(\rho)$. Since $\H_{\rho_0}^\infty(0)$ contains all vectors of the form $\int_\R f(t) \rho_0(t)v$ for $f \in C^\infty_c(\R)$ and $v \in V(\rho)\cap \H_\rho^{\omega_G}$, $\overline{\H_{\rho_0}^\infty(0)}$ contains $V(\rho) \cap \H_{\rho}^{\omega_G}$, which is dense in $V(\rho)$. So $\overline{\H_{\rho_0}^\infty(0)} = V(\rho) = \H_{\rho_0}(0)$. This implies that $\rho_0$ is analytically ground state.
	\end{proof}
	
	\subsubsection{Strongly-entire ground state representations for $\T$-actions}
	
	\noindent
	The preceding results become particularly applicable for representations $\rho$ which are both strongly-entire and ground state w.r.t.\ a $\T$-action. In this case, we can always guarantee that they are analytically ground state:
	
	\begin{lemma}\label{lem: projection_leaves_entire_vectors_inv}
		Suppose that $\alpha$ descends to a $\T$-action. Let $\rho$ be a unitary p.e.\ representation of $G \rtimes_\alpha \T$. We write $\H_\rho^{{\mcO}_G}$ for the vectors in $\H_\rho$ that are strongly-entire for the $G$-action. Let $P : \H_\rho \to \H_{\rho}(0)$ denote the orthogonal projection. Then $P \H_\rho^{{\mcO}_G} \subseteq \H_\rho^{{\mcO}_G}$. In particular, if $\restr{\rho}{G}$ is strongly-entire then $\H_{\rho}(0) \cap \H_\rho^{{\mcO}_G}$ is dense in $\H_{\rho}(0)$.
	\end{lemma}
	\begin{proof}
		For a compact subset $B \subseteq \g_\C$ and $\psi \in \H_\rho^\infty$, we write 
		$$p_B^n(\psi) := \sup_{\xi_j \in B}\|d\rho(\xi_1\cdots \xi_n)\psi\|$$
		for $n \in \N_{\geq 0}$ and set $q_B(\psi) := \sum_{n=0}^\infty{1\over n!}p_B^n(\psi)$. Let $\psi \in \H_\rho^{{\mcO}_G}$ and let $B \subseteq \g_\C$ be compact. Then $B^\prime := \alpha(\T \times B) \subseteq \g_\C$ is compact, $\T$-invariant and satisfies $B \subseteq B^\prime$. Observe that 
		$$p_B^n(\rho(t)\psi) \leq p_{B^\prime}^n(\rho(t)\psi) = p_{\dot{\alpha}_{-t}(B^\prime)}^n(\psi) = p_{B^\prime}^n(\psi), \qquad \forall t \in \T.$$
		Identifying $\T \cong \R/2\pi\Z$, recall that $P = {1\over 2\pi}\int_{0}^{2\pi}\rho(t)dt$. Notice using e.g.\ \citep[Lem.\ A.4]{KH_Zellner_inv_smooth_vectors} that $P \H_\rho^\infty \subseteq \H_\rho^\infty$, and moreover that
		$$ p_{B}^n(P\psi) \leq {1\over 2\pi}\int_{0}^{2\pi}p_B^n(\rho(t)\psi) dt \leq  p_{B^\prime}^n(\psi), \qquad \forall \psi \in \H_\rho^\infty, \; n \in \N_{\geq 0}.$$
		We thus find that $q_B(P \psi) \leq q_{B^\prime}(\psi)$. So $P\H_\rho^{{\mcO}_G} \subseteq \H_\rho^{{\mcO}_G}$.
	\end{proof}
	
	\noindent
	Combining \Fref{thm: ground_state_hol_ind} and \Fref{lem: projection_leaves_entire_vectors_inv} we obtain the following:
	
	\begin{theorem}\label{thm: equiv_pe_hol_for_periodic_actions_and_entire_reps}
		Assume that $\alpha$ is a $\T$-action. Let $(\rho, \H_\rho)$ be a unitary representation of $G \rtimes_\alpha \T$. Assume that $\restr{\rho}{G}$ is strongly-entire. Let $\sigma$ be the unitary representation of $H \times \T$ on $V(\rho)$. The following are equivalent:
		\begin{enumerate}
			\item $\rho$ is ground state at $\bm{d} \in \g^\sharp$.
			\item $\restr{\rho}{G} = \HolInd_H^G(\restr{\sigma}{H})$ and $V(\rho) = \H_\rho(0)$.
		\end{enumerate}
		In this case, also $\sigma$ is strongly-entire.
	\end{theorem}
	
	\noindent
	By \Fref{prop: positive_extensions_of_repr}(3), we know that any smooth unitary representation $\rho$ of $G$ which is of p.e.\ w.r.t.\ a smooth $\T$-action $\alpha$ is automatically ground state, and also that the minimal positive extension $\rho_0$ of $\rho$ to $G^\sharp$ descends to $G \rtimes_\alpha \T$. Combining \Fref{thm: equiv_pe_hol_for_periodic_actions_and_entire_reps}, \Fref{cor: pe_hol_ind_iso_commutants} and \Fref{cor: pe_uniqueness}, we obtain:
	
	\begin{corollary}\label{cor: inclusion_of_pe_reps}
		Assume that $\alpha$ is a smooth $\T$-action and that every irreducible unitary representation of $G$ that is of positive energy w.r.t.\ $\alpha$ is strongly-entire. Then there is an injective map $\widehat{G}_{\pos(\alpha)} \hookrightarrow \widehat{H}$, obtained by sending $\rho \in \widehat{G}_{\pos(\alpha)}$ to the irreducible unitary $H$-representation on $V(\rho)$. 
	\end{corollary}
	
	\begin{remark}
		Recall from \Fref{thm: hol_extensions_of_reps} that if $G$ is a finite-dimensional Lie group of type $R$, then every continuous unitary $G$-representation is in fact strongly-entire. \\
	\end{remark}

	\noindent
	It would be beneficial to obtain sufficient conditions for $V(\rho)\cap \H_\rho^{\omega_G}$ to be dense in $V(\rho)$. We state the following related open problem:
	\begin{problem}
		Assume there are $0$-neighborhoods $U \subseteq \g_\C$, $U_- \subseteq \n_-$, $U_0 \subseteq \h_\C$ and $U_+ \subseteq \n_+$ for which the map 
		$$U_+ \times U_0 \times U_- \to U, \qquad (\xi_+, \xi_0, \xi_-) \mapsto \xi_+ \ast \xi_0 \ast \xi_-$$
		is biholomorphic, where $\ast$ is defined by the BCH series. We write $\xi \mapsto (\xi_+, \xi_0, \xi_-)$ for its inverse. Let $\rho$ be a unitary representation of $G$ that is of positive energy. Set $V_\sigma := V(\rho)$, considered as a unitary $H$-representation. Assume that $V_\sigma$ is cyclic for $\rho$. Is it true that $V_\sigma^\omega \subseteq \H_\rho^{\omega_G}$? Taking $v \in V_\sigma^\omega$, the assumptions imply that the map $U \to \C, \; \xi \mapsto \langle v, \sigma(e^{\xi_0}) v\rangle$ is analytic on some $0$-neighborhood. If it can be shown to locally extend the map $\g \to \C, \;\xi \mapsto \langle v, \rho(e^\xi)v\rangle$ on some $0$-neighborhood in $\g$, then it would follow from \citep[Thm.\ 5.2]{Neeb_analytic_vectors} that $v \in \H_\rho^{\omega_G}$.
	\end{problem}
	

	\section{Examples}\label{sec: examples}
	\setcounter{subsection}{1}
	\setcounter{theorem}{0}	
	
	\begin{example}[Finite-dimensional Lie groups of type $R$]\label{ex: type_R}~\\
		Let $G$ be a connected finite-dimensional Lie group of type $R$ and let $\alpha$ be a $\T$-action on $G$. Let $H := (G^\alpha)_0$ be the connected subgroup of $\alpha$-fixed points. In view of \Fref{thm: hol_extensions_of_reps} and \Fref{thm: equiv_pe_hol_for_periodic_actions_and_entire_reps}, any continuous ground state representation $\rho$ of $G$ is holomorphically induced from $V(\rho)$. According to \Fref{cor: inclusion_of_pe_reps}, this defines an injection $\widehat{G}_{\pos(\alpha)} \hookrightarrow \widehat{H}$.
	\end{example}
	
	\begin{example}[Holomorphically induced, but not geometrically]~
		\begin{enumerate}
			\item Consider $G = \SL(2, \R)$ and let $\rho$ be any non-trivial continuous unitary representation. Trivially, we have $\rho = \HolInd_G^G(\rho)$. However, as $\rho$ admits no non-trivial strongly-entire vectors by \Fref{thm: hol_extensions_of_reps}, it is \textit{not} geometrically holomorphically induced from itself.
			
			\item For a slightly less trivial example, consider the group $G = K \times \SL(2, \R)$, where $K$ is a connected compact simple Lie group. Let $T \subseteq K$ be a maximal torus and set $\t := \bm{\mrm{L}}(T)$. Pick a regular element $H \in \t_{reg}$ and let $\Delta_+ := \set{\alpha \in \Delta \st -i\alpha(H) > 0 }$ be the corresponding system of positive roots, where $\Delta\subseteq i\t^\ast$ denotes the set of all roots of $\k$. Consider the $\T$-action on $G$ defined by $\alpha_t(k, x) = (e^{tH}ke^{-tH}, x)$. Let $(\rho, \H_\rho)$ be a continuous irreducible unitary representation of $G$. Then $\rho$ decomposes as $\H_\rho = \H_\nu \otimes \H_\sigma$ for some irreducible unitary $K$- and $\SL(2,\R)$-representations $(\nu, \H_\nu)$ and $(\sigma, \H_\sigma)$, respectively. Then $\rho$ is of positive energy w.r.t.\ $\alpha$ and $V(\rho) = \C_\lambda \otimes \H_\sigma$, where $\C_\lambda \subseteq \H_\nu$ is a lowest-weight subspace. Since $\H_\rho^\omega = \H_\nu \otimes \H_\sigma^\omega$, \Fref{thm: hol_induced_equiv_char} implies that $\rho$ is holomorphically induced from the $T \times \SL(2,\R)$-representation on $\C_\lambda \otimes \H_\sigma$. The latter admits no strongly-entire vectors by \Fref{thm: hol_extensions_of_reps}, so $\rho$ is not geometrically holomorphically induced from the $T \times \SL(2,\R)$-representation on $\C_\lambda \otimes \H_\sigma$.
		\end{enumerate}		
	\end{example}
	
	\begin{example}[Positive energy representations of Heisenberg groups]\label{ex: pe_repr_heis}~\\
		Let $V$ be a real Fr\'echet space equipped with a non-degenerate continuous skew-symmetric bilinear form $\omega$. Let $\beta : \T\to \Sp(V,\omega)$ be a homomorphism with smooth action $\T\times V \to V$. Then $\beta$ is equicontinuous by \citep[Lem.\ A.3]{Neeb_hol_reps}. Define $Dv := \restr{d\over dt}{t=0}\beta_t v$ and consider the closed subspaces 
		\begin{equation}\label{eq: eff_trivial_subspaces}
				V_0 := \ker D = \set{v \in V \st \beta_t v = v \quad \forall t \in \R}, \qquad V_{\eff} := \overline{\Span}\set{\beta_tv - v \st t \in \R, v \in V}.
		\end{equation}
		As $\beta_t^\ast \omega = \omega$ of all $t \in \R$, notice that $V_0$ and $V_{\eff}$ are symplectic complements, so $(V, \omega) \cong (V_0, \omega_0) \oplus (V_\eff, \omega_1) $, where $\omega_0$ and $\omega_1$ are the restrictions of $\omega$ to $V_0$ and $V_\eff$, respectively. Assume that $(V_\eff)_\C$ decomposes as $(V_\eff)_\C \cong L_+ \oplus L_-$ into the positive $(L_+)$ and negative $(L_-)$ Fourier modes of the $\T$-action $\beta$. Let $\mf{heis}(V, \omega)\rtimes_D \R \bm{d}$ be the Lie algebra of $\Heis(V, \omega)\rtimes_\beta \T$. By \Fref{thm: ground_state_hol_ind}, we know for any unitary representation $\rho$ of $\Heis(V, \omega) \rtimes_\beta \T$ which is analytically ground state at $\bm{d}$ that $\restr{\rho}{\Heis(V, \omega)}$ is holomorphically induced by some analytic unitary representation of $\Heis(V_0, \omega_0)$. \\
		
		\noindent
		Let us consider a concrete example. Assume that $\omega_1(v, Dv) > 0 $ for every nonzero $v \in V_\eff$. Let $\mc{J}_1$ be the complex structure on $V$ defined by $\mc{J}_1(v + w) := iv -iw$ for $v \in L_+$ and $w \in L_-$. Then $\mc{J}_1^\ast \omega_1 = \omega_1$ and $\omega_1(v, \mc{J}_1v)>0$ for every $v \in V_\eff$, so $\mc{J}_1$ defines a compatible positive polarization on $V_\eff$. If $\mc{J}_0$ is a compatible positive polarization on $V_0$, then $\mc{J} = \mc{J}_0 \oplus \mc{J}_1$ defines one on $V$. As in \Fref{ex: pe_reprs_of_Heisenberg_groups}, we equip the (now complex) vector space $V$ with the inner product $\langle v, w\rangle_{\mc{J}} := \omega(v, \mc{J}w) + i \omega(v,w)$, making $V$ into a complex pre-Hilbert space, on which $\beta_t$ acts unitarily for any $t \in \T$. Let $\H$ be its Hilbert space completion, and let $\H_0$ and $\H_{1}$ be the closures in $\H$ of $V_0$ and $V_\eff$, respectively. Consider the unitary representation $u : \T \to \U(\H)$ of $\T$ on $\H$ defined by $t \mapsto u_t$, where $u_t$ is the unitary operator on $\H$ extending $\beta_t$. Notice that as unitary $\T$-representations, we have $(V_\eff, \langle \--, \--\rangle_{\mc{J}_1})\cong(L_+, \langle \--, \--\rangle_{L_+})$, where $\langle v, w\rangle_{L_+} := 2i \omega(\overline{v}, w)$ for $v,w \in L_+$. The unitary $\T$-representation $u$ on $\H$ is therefore of positive energy. Let $\F(\H)$ be the Hilbert space completion of the symmetric algebra $\mrm{S}^\bullet(\H)$ w.r.t.\ the inner product \eqref{eq: ip_fock_space}, and let $\rho$ be the b-strongly-entire unitary representation of $\Heis(V, \omega)$ on $\F(\H)$ constructed in \Fref{ex: pe_reprs_of_Heisenberg_groups}. Similarly, we write $\rho_0$ and $\rho_1$ for the representations of $\Heis(V_0, \omega_0)$ and $\Heis(V_\eff, \omega_1)$ on $\F(\H_0)$ and $\F(\H_1)$, respectively. Letting $\T$ act on $\F(\H)$ according to the second quantization $\F(u)$ of $u$, we obtain an extension of $\rho$ to a smooth representation of $\Heis(V, \omega) \rtimes_\beta \T$ on $\F(\H)$, which we denote again by $\rho$. Explicitly, we have $\rho(g, t) = \rho(g)\F(u_t)$ for $(g,t) \in \Heis(V, \omega) \rtimes_\beta \T$. This extension is ground state w.r.t.\ $\beta$. We have $\F(\H) \cong \F(\H_0)\otimes \F(\H_1)$ and 
		$$V(\rho) = \F(\H_0) \otimes \Omega_1 \subseteq \F(\H),$$
		where $\Omega_1 \in \F(\H_1)$ is the vacuum vector. \Fref{thm: equiv_pe_hol_for_periodic_actions_and_entire_reps} implies that $\restr{\rho}{\Heis(V, \omega)}$ is holomorphically induced from the representation $\rho_0$ of $\Heis(V_0, \omega_0)$ on $\F(\H_0)$. Moreover, we have $\F(\H_0)^\infty \otimes \Omega_1 \subseteq \H_\rho^\infty$. Indeed, the vacuum vector $\Omega_1$ is smooth for $\Heis(V_\eff, \omega_1)$, so if $\psi \in \F(\H_0)^\infty$ is a smooth vector for $\Heis(V_0, \omega_0)$ then 
		$$(z,v) \mapsto \rho(z,v)\psi = z\rho_0(v_0)\psi_0 \otimes \rho_1(v_1)\Omega_1$$
		is a smooth map $\Heis(V, \omega) \to \F(\H)$. So provided that $V_0$ is a Banach space, it follows using \Fref{thm: comparing_two_notions_hol_ind} and \Fref{ex: banach_trivial_extension_entire} that $\restr{\rho}{\Heis(V, \omega)}$ is also geometrically holomorphically induced from $\rho_0$. 
	\end{example}
	
	\begin{example}[Metaplectic representation]\label{ex: metaplectic}~\\
		We continue in the notation and setting of \Fref{ex: pe_repr_heis}. Let $\H_\R$ be the real vector space underlying $\H$. The symplectic form $\omega$ on $V$ extends to $\H_\R$ by setting $\omega(v,w) := \Im\langle v, w\rangle_{\mc{J}}$ for $v,w \in \H_\R$. Define 
		$$ \B_\res(\H_\R) := \set{ A \in \B(\H_\R) \st [\mc{J}, A] \in \B_2(\H)},$$
		whose elements are \squotes{close} to being $\C$-linear. It is a real Banach algebra with norm $\|A\|_{\res} := \|A\| + \|[\mc{J}, A]\|_2$, where $\B_2(\H)$ denotes the space of Hilbert-Schmidt operators on $\H$. The restricted symplectic group is defined by 
		$$\Sp_\res(\H_\R, \omega) := \Sp(\H_\R, \omega)\cap \B_\res(\H_\R),$$
		equipped with the subspace topology. Being an algebraic subgroup of $\B_\res(\H_\R)^\times$, we obtain using \citep[Prop.\ IV.14]{Neeb_inf_dim_gps_reps_2004} that $\Sp_\res(\H_\R, \omega)$ is a Banach--Lie group modeled on the Banach--Lie algebra 
		$$\sp_\res(\H_\R, \omega) := \sp(\H_\R, \omega) \cap \B_\res(\H_\R).$$
		By \Fref{ex: pe_repr_heis}, there is an irreducible projective unitary representation $\overline{\rho}$ of the Abelian Banach--Lie group $(\H_\R, +)$ on the symmetric Fock space $\F(\H)$, which is well-known to extend to the semi-direct product $\HSp_\res(\H_\R, \omega) := \H_\R \rtimes \Sp_\res(\H_\R, \omega)$ \citep[Rem.\ 9.12]{neeb_semibounded_inv_cones}. We denote this extension again by $\overline{\rho}$. Let $\widetilde{\Sp}_\res(\H_\R, \omega)$ and $\widetilde{\HSp}_\res(\H_\R, \omega)$ be the central $\T$-extensions of $\Sp_\res(\H_\R, \omega)$ and $\HSp_\res(\H_\R, \omega)$, respectively, obtained by pulling back the central $\T$-extension $\U(\F(\H)) \to \P\U(\F(\H))$ along $\overline{\rho}$. Let 
		\[\rho : \widetilde{\HSp}_\res(H_\R, \omega) \to \U(\F(\H))\]
		be the corresponding lift of $\overline{\rho}$. It is proven in \citep[Thm.\ 9.3, Rem.\ 9.12]{neeb_semibounded_inv_cones} that both $\widetilde{\HSp}_\res(\H_\R, \omega)$ and $\widetilde{\Sp}_\res(\H_\R, \omega)$ are again Banach--Lie groups, and that the (cyclic) vacuum vector $\Omega \in \F(\H)$ is smooth for $\rho$, which implies that $\rho$ is a smooth representation of $\widetilde{\HSp}_\res(H_\R, \omega)$. Identify $\F(\H_0)$ as a subspace of $\F(\H)$ via $\F(\H_0) \hookrightarrow \F(\H), \; \psi_0 \mapsto \psi_0 \otimes \Omega_1$. We observe first that $\rho$ is in fact analytic, and that $\F(\H_0)$ contains a dense set of $\widetilde{\HSp}_\res(\H_\R, \omega)$-analytic vectors. To see this, notice using equation $\textrm{(35)}$ and the subsequent remarks in the proof of \citep[Thm.\ 9.3]{neeb_semibounded_inv_cones} that $\Omega$ is not just a smooth-, but even an analytic vector for the action of $\widetilde{\Sp}_\res(H_\R, \omega)$ on $\F(\H)$. We furthermore know from \Fref{ex: pe_repr_heis} that $\Omega$ is analytic for $\Heis(\H_\R, \omega)$ (and even b-strongly-entire), so the function
		\begin{align*}
			&\Heis(\H_\R, \omega) \times \widetilde{\Sp}_\res(\H_\R, \omega) \to \C,\\
			&(v, A) \mapsto \langle \Omega, \rho(v)\rho(A) \Omega\rangle = \langle \rho(v)^{-1}\Omega, \rho(A)\Omega\rangle
		\end{align*}
		is real-analytic. This implies using \citep[Thm.\ 5.2]{Neeb_analytic_vectors} that $\Omega$ is an analytic vector for the representation $\rho$ of $\widetilde{\HSp}_\res(\H_\R, \omega)$ on $\F(\H)$. As $\Omega$ is cyclic for the action $\rho_0$ of $\Heis((\H_0)_\R, \omega_0)$ on $\F(\H_0)$, it follows that the set of vectors in $\F(\H_0) \subseteq \F(\H)$ that are analytic for the action of $\widetilde{\HSp}_\res(\H_\R, \omega)$ is dense in $\F(\H_0)$. It also follows that $\rho$ is analytic, because $\Omega$ is cyclic for $\rho$. \\
		
		\noindent
		Now, suppose that $G$ is a connected regular BCH Fr\'echet--Lie group with Lie algebra $\g$, and that the homomorphism $\alpha : \T \to \Aut(G)$ defines a smooth $\T$-action on $G$. Assume that $\alpha$ satisfies the assumptions made in \Fref{sec: pe_and_hol_ind}. Let $H := (G^\alpha)_0$ be the connected subgroup of $\alpha$-fixed points in $G$. Consider the $\T$-actions on $\HSp_\res(\H_\R, \omega)$ and $\U(\F(\H))$ defined by $t \cdot (v, x) := (u_t v, u_t xu_t^{-1})$ and $t \cdot U := \F(u_t)U\F(u_t)^{-1}$ respectively, where $(v,x) \in \HSp_\res(\H_\R, \omega)$, $U \in \U(\F(\H))$ and $t \in \T$. This also equips the $\T$-invariant subgroup $\widetilde{\HSp}_\res(H_\R, \omega)\subseteq\HSp_\res(\H_\R, \omega) \times \U(\F(\H))$ with a $\T$-action. Assume that
		\[ \eta : G \to \widetilde{\HSp}_\res(\H_\R, \omega)\]
		is a continuous and $\T$-equivariant homomorphism. Then $\eta$ is automatically analytic by \citep[Thm.\ IV.1.18]{neeb_towards_lie}. Letting $\T$ act on $\F(\H)$ by $t \mapsto \F(u_t)$, notice that $\rho \circ \eta$ extends to a smooth unitary representation of $G \rtimes_\alpha \T$ which is of positive energy at $(0,1) \in \g \rtimes \R$. Moreover, it follows from the preceding that $\F(\H_0)$ contains a dense set of vectors that are analytic for the $G$-action $\rho \circ \eta$ on $\F(\H)$. Letting $\mc{K}\subseteq \F(\H)$ be the closed $(G \rtimes_\alpha \T)$-invariant linear subspace generated by $\F(\H_0)$, it follows that the unitary representation of $G \rtimes_\alpha \T$ on $\mc{K}$ is analytically ground-state at $(0,1) \in \g \rtimes \R$. According to \Fref{thm: ground_state_hol_ind}, it further follows that the $G$-representation $\rho \circ \eta$ on $\mc{K}$ is holomorphically induced from the $H$-representation on $\F(\H_0)$.
	\end{example}

	\begin{example}[Groups of jets]\label{ex: jets}~\\
		Let $K$ be a $1$-connected compact simple Lie group with Lie algebra $\k$. Let $V$ be a finite-dimensional real vector space. We consider the Lie group $J^n_0(V, K)$ of $n$-jets at $0 \in V$ of smooth maps $V \to K$. Let $\gamma : \R \to \GL(V)$ be a continuous representation of $\R$ on $V$ and let $\phi \in \k$. Assume that the $\R$-action $\widetilde{\alpha}_t(f)(x):= e^{t\phi}f(\gamma_{-t}(x))e^{-t\phi}$ on $C_c^\infty(V, K)$ factors through $\T := \R/\Z$. As $\gamma$ fixes the origin, $\widetilde{\alpha}$ descends to a smooth $\T$-action on $J^n_0(V, K)$, denoted $\alpha$. Let $D := \restr{d\over dt}{t=0}\dot{\alpha}_t$ be the corresponding derivation on the Lie algebra $J^n_0(V, \k)$. Let $G$ be a central $\T$-extension of $J^n_0(V, K) \rtimes_\alpha \T$, and let $\bm{d} \in \g := \bm{\mrm{L}}(G)$ cover $(0,1) \in J^n_0(V,\k) \rtimes_D \R$. As usual, we write $H := (G^\alpha)_0 \subseteq G$ for the connected Lie subgroup of $\alpha$-fixed points in $G$, whose Lie algebra is $\h = \ker D$. As $G \cong N \rtimes K$ for some nilpotent Lie group $N$, it follows from \Fref{prop: classification_type_R_LA} that $G$ is of type $R$. By \Fref{ex: type_R}, we thus obtain that any continuous unitary $G$-representation which is of positive energy w.r.t.\ $\alpha$ is holomorphically by some unitary $H$-representation. A classification of $\widehat{G}_{\pos(\alpha)}$ amounts to determining the holomorphically inducible elements in $\widehat{H}$. Unitary positive energy representations of groups of jets are studied in more detail in \citep{Milan_reps_jets}.\\
		
		\noindent
		To make the preceding concrete, suppose that $V = \R^2$, $n = 2k$ for some $k \in \N$, that $\gamma$ is the action of $\T$ on $\R^2$ by rotations and that $\phi=0$. Then $\h \cong \R\oplus_\omega (\R_k[x^2 + y^2]\otimes \k)$, where $\R_k[c]$ denotes the polynomial ring in $c$ truncated at the $k^{th}$ degree, and where $\omega$ is a $2$-cocycle on the Lie algebra $\R_k[x^2 + y^2]\otimes \k$ (which in this case actually must be a coboundary). Every continuous unitary representation $\rho$ of $G$ that is of positive energy w.r.t.\ $\alpha$ is holomorphically induced from the $H$-representation on $V(\rho)$. 
	\end{example}
	
	\begin{example}[Gauge groups]~\\
		Let $M$ be a compact manifold and let $P \to M$ be a principal bundle with structure group $K$, a simple compact Lie group with Lie algebra $\k$. Consider the group of gauge transformations $\Gau(P) = \Gamma(M, \Ad(P))$, where $\Ad(P) = P \times_\Ad K$ is the adjoint bundle. This group is a regular BCH Fr\'echet--Lie group with Lie algebra $\gau(P) = \Gamma(M, P \times_\Ad \k)$ \citep[Thm.\ IV.1.12]{neeb_towards_lie}. Suppose that $\gamma : \T \to \Aut(P)$ is a smooth $\T$-action on $P$ by automorphisms of $P$. Let $\eta : \T \to \Aut(\Ad(P))$ and $\overline{\gamma} : \T \to \Diff(M)$ denote the induced $\T$-actions on $\Ad(P)$ and $M$, respectively. Explicitly, $\eta$ is given by $\eta_t([p, k]) := [\gamma_t(p), k]$ for $p \in P, \, k \in K$ and $t \in \T$. Then $\T$ acts smoothly on $\Gau(P)$ by $\alpha_t(s)  := \eta_t \circ s \circ \overline{\gamma}_{-t}$ for $s \in \Gau(P)$ and $t \in \T$. The paper \citep{BasNeeb_PE_reps_I} studies projective unitary representations of $\Gau(P)$ which are smooth in the sense of admitting a dense set of smooth rays. According to \citep[Cor.\ 4.5, Thm.\ 7.3]{BasNeeb_ProjReps}, these correspond to smooth unitary representations of a central $\T$-extension of $\Gau(P)$. One of the main results of \citep{BasNeeb_ProjReps} is the full classification of smooth projective unitary of the identity component $\Gau(P)_0$ which are of positive energy w.r.t\ $\alpha$, provided that $M$ has no $\T$-fixed points for $\overline{\gamma}$ \citep[Thm.\ 8.10]{BasNeeb_PE_reps_I}. Let us consider a central $\T$-extension $G$ of the connected component $\Gau(P)_0$ of the identity. Suppose that $\alpha$ lifts to a smooth $\T$-action $\widetilde{\alpha}$ on $G$. In view of \citep[Prop.\ 8.6]{BasNeeb_PE_reps_I}, a consequence of the classification \citep[Thm.\ 8.10]{BasNeeb_PE_reps_I} is that every smooth unitary representation $\rho$ of $G$ which is of positive energy w.r.t.\ $\widetilde{\alpha}$ is holomorphically induced from the corresponding representation of $H := (G^{\widetilde{\alpha}})_0$ on $V(\rho)$. The more general case where $M$ is allowed to have $\T$-fixed points is not yet fully understood. One approach would be to determine, in specific cases, the irreducible unitary factor representations of $H$ that are holomorphically inducible to $G$, as an intermediate step towards the classification of the possibly larger class of all p.e.\ factor representations of $G$.
	\end{example}
	
	\begin{example}[Unitary groups of CIA's]~\label{ex: unitary_groups_cia}\\
		An interesting class of examples to which the theory of \Fref{sec: pe_and_hol_ind} applies can be obtained using so-called continuous inverse algebras (CIAs). Suppose that $\A$ is a unital complex Fr\'echet algebra that is a CIA, meaning that its group of units $\A^\times$ is open in $\A$ and that the inversion $a \mapsto a^{-1}$ is continuous $\A \to \A$. Let us suppose further that $\A$ carries a continuous conjugate-linear algebra involution $\A \to \A, \, a \mapsto a^\ast$. In this setting, $\A^\times$ is a complex BCH Fr\'echet--Lie group modeled on $\A$ \citep[Thm.\ 5.6]{Glockner_unit_gps_lie}. Assume that the Lie group $\A^\times$ is moreover regular. A sufficient condition for this is provided in \citep{GlocknerNeeb_unit_gp_regular}. The unitary subgroup
		$$ \U(\A) := \set{a \in \A^\times \st a^\ast = a^{-1}} $$
		is a real Lie subgroup of $\A^\times$, so that it is an embedded submanifold. It is modeled on the Lie algebra
		$$ \mf{u}(\A) := \set{a \in \A \st a^\ast = - a},$$
		equipped with the commutator bracket. To see this, let $U \subseteq \A$ be a $0$-neighborhood s.t.\ $\exp_\A$ maps $U$ diffeomorphically onto its image in $\A^\times$. We may assume that $U =-U$ and that $U^\ast = U$, by shrinking $U$ if necessary. By \citep[Cor.\ 4.11]{Glockner_unit_gps_lie} we know that $\exp_\A(a) = \sum_{n=0}^\infty{1\over n!}a^n$ for all $a \in \A$. Using that both $a\mapsto a^{-1}$ and $a \mapsto a^\ast$ are continuous, it follows that $\exp_\A(a)^\ast = \exp(a^\ast)$ and $\exp_\A(a)^{-1} = \exp_\A(-a)$ for all $a \in U$. This implies that $\exp_\A(U \cap \mf{u}(\A)) = \exp_\A(U) \cap \U(\A)$. As $\U(\A)$ is a closed subgroup of the locally exponential Lie group $\A^\times$, it follows from \citep[Thm.\ IV.3.3]{neeb_towards_lie} that $\U(\A) \subseteq \A^\times$ is a locally exponential Lie subgroup. It is therefore a regular BCH Fr\'echet--Lie group. Notice further that $\mf{u}(\A)_\C = (\A, [\--, \--])$ as complex Lie algebras. \\
		
		\noindent
		Suppose that $\alpha : \R \to \Aut(\A)$ is a homomorphism that has a smooth action $\R \times \A \to \A$ and that has polynomial growth. Assume further that the splitting condition
		$$ \A = \A_- \oplus \A_0 \oplus \A_+ $$
		is satisfied. Setting $G := \U(\A)_0$ and $H := \U(\A_0)_0 = (G^\alpha)_0$, all assumptions of both \Fref{sec: hol_induced} and \Fref{sec: pe_and_hol_ind} are satisfied.\\
		
		\noindent
		Typically, such triples $(\A, \R, \alpha)$ can be obtained as the set of smooth points of a $C^\ast$-dynamical system $(\mc{B}, G, \gamma)$, where $\mc{B}$ is a unital $C^\ast$-algebra, $G$ is a Banach--Lie group and $\gamma : G \to \Aut(\mc{B})$ is a strongly continuous $G$-action on $\mc{B}$ by automorphisms. By \citep[Def.\ 4.1, Thm.\ 6.2]{Neeb_diffvectors}, we know in this setting that the set of smooth points $\A := \mc{B}^\infty$ is a $G$-invariant and $\ast$-closed subalgebra which naturally carries a Fr\'echet topology. Moreover, $\A$ is a CIA and the $G$-action $\gamma : G \times \A \to \A$ is smooth w.r.t.\ this topology. If $G$ is finite-dimensional, then this topology coincides with the one obtained from the embedding $\A \hookrightarrow C^\infty(G, \mc{B})$, where $C^\infty(G, \mc{B})$ carries the smooth compact-open topology \citep[Prop.\ 4.6]{Neeb_diffvectors}. If $\iota : \R \hookrightarrow G$ is a one-parameter subgroup of $G$ for which the corresponding $\R$-action $\alpha := \gamma \circ \iota$ on $\A$ has polynomial growth and satisfies the splitting condition $\A= \A_- \oplus \A_0 \oplus \A_+$, then the triple $(\A, \R, \alpha)$ satisfies all the above assumptions. \\

		\noindent
		As a concrete example, let $\A_\theta := C^\infty_\theta(\T^2)$ be the smooth non-commutative $2$-torus with parameter $\theta \in [0,{1\over 2}]$:
		$$\A_\theta := \set{ \sum_{n,m \in \Z} a_{n,m}u^nv^m \st \sum_{n,m \in \Z}(1 + |n|+ |m|)^k |a_{n,m}| < \infty \text{ for all } k\in \N }, $$
		where $u$ and $v$ are unitary operators satisfying $uv = e^{i2\pi\theta}vu$, and where $\A_\theta$ is equipped with the seminorms $p_k(a) := \sum_{n,m \in \Z}(1 + |n| + |m|)^k |a_{n,m}|$ for $k \in \N_{\geq 0}$. This is a unital Fr\'echet CIA carrying a continuous involution, obtained as the smooth points of the natural $\T^2$-action on the \squotes{continuous} non-commutative $2$-torus $C_\theta(\T^2)$ with parameter $\theta$. It is moreover shown in \citep{Neeb_Glockner_Book} that the Lie group $\A_\theta^\times$ is regular. Consider the smooth and equicontinuous $\T$-action $\alpha$ on $C^\infty_\theta(\T^2)$ satisfying $\alpha_z(u^nv^m) := z^{m}u^nv^m$ for all $n,m \in \Z$ and $z \in \T$. Define $G := \U(\A_\theta)_0$. Then for any unitary representation $\rho$ of $G \rtimes_\alpha \T$ that is analytically ground state w.r.t.\ $\alpha$, we obtain from \Fref{thm: ground_state_hol_ind} that $\restr{\rho}{G}$ is holomorphically induced from the corresponding unitary representation of the connected Abelian group $H :=(\U(\A_\theta)^\alpha)_0 \cong C^\infty(\T, \T)_0$ on $\H_\rho(0)$. In particular, if $\rho(G)^{\prime \prime}$ is a factor, then as $H$ is Abelian, we obtain with \Fref{cor: pe_hol_ind_iso_commutants} that $\restr{\rho}{G}$ is holomorphically induced from a character of $H$. By \Fref{cor: pe_hol_ind_iso_commutants} this implies that $\restr{\rho}{G}$ is irreducible.
	\end{example}
	
	\appendix
	\section{Representations on reproducing kernel Hilbert spaces}\label{sec: reproducing_hspaces}
	
	\setcounter{subsection}{1}
	\setcounter{theorem}{0}
	
	\noindent
	In the following we summarize relevant properties concerning reproducing kernel Hilbert spaces in the context of unitary group representations. Let $\H$ and $V$ be Hilbert spaces and let $G$ be a group. We write $V^G$ or $\Map(G, V)$ for the space of functions $G \to V$ and $V^{(G)}$ for the space of finitely-supported functions $G \to V$.
	
	\begin{definition}
		Suppose that $\H \subseteq V^G$. Then $\H$ is said to \textit{have continuous evaluation maps} if for every $x \in G$ the linear map $\mc{E}_x : \H \to V, \psi\mapsto \psi(x)$ is bounded.
	\end{definition}
	
	\begin{definition}\label{def: pos_kernel}
		A function $Q : G \times G \to \B(V)$ is said to be \textit{positive definite} if 
		$$ \| v\|^2_{Q} := \sum_{x,y \in \supp(v)}\langle v_x, Q(x,y)v_y\rangle_{V} \geq 0, \qquad \forall v \in V^{(G)}. $$
	\end{definition}
	
	\begin{theorem}[{\citep[Thm.\ I.1.4]{Neeb_book_hol_conv}}]\label{thm: properties_repr_kernel}~\\
		Let $Q : G \times G \to \B(V)$ be a function. The following assertions are equivalent:
		\begin{enumerate}
			\item $Q$ is positive definite			
			\item There exists a Hilbert space $\H_Q\subseteq V^G$ with continuous point-evaluations $\mc{E}_x~:~\H_Q \to V$ such that $Q(x,y) = \mc{E}_x \mc{E}_y^\ast$ for all $(x,y) \in G \times G$.
		\end{enumerate}
		In this case $\H_Q$ is unique up to unitary equivalence and $\set{ \mc{E}_x^\ast v \st x \in G, v \in V }$ is total in $\H_Q$.
	\end{theorem}
	
	\begin{definition}
		A function $Q : G \times G \to \B(V)$ is said to be a \textit{reproducing kernel} for the Hilbert space $\H$ if $Q$ is positive definite and $\H \cong \H_Q$.
	\end{definition}
	
	\begin{proposition}\label{prop: homog_repr_h_space_unitary_repr}
		Let $G$ be a topological group and $H \subseteq G$ be a closed subgroup. Let $(\sigma, V_\sigma)$ be a strongly continuous unitary $H$-representation. Let $Q \in C(G \times G, \B(V_\sigma))^{H \times H}$, so $Q(xh_1,yh_2) = \sigma(h_1)^{-1}Q(x,y)\sigma(h_2)$ for all $x_1, x_2 \in G$ and $h_1, h_2 \in H$. Assume that $Q$ is positive definite. 
		\begin{enumerate}
			\item The left-regular action of $G$ on $V_\sigma^{(G)}$ induces a unitary $G$-action $\pi$ on $\H_Q$ if and only if $Q$ is $G$-invariant. In this case, there is a function $F : G \to \B(V_\sigma)$ such that $Q(x,y) = F(x^{-1}y)$.
			\item Assume that $Q$ is $G$-invariant. There exists a $G$-equivariant linear map $\H_Q \hookrightarrow \Map(G, V_\sigma)^H$ with continuous point-evaluations $\mc{E}_x$ for $x \in G$. These satisfy the equivariance condition $\mc{E}_{x}\pi(g) = \mc{E}_{g^{-1}x}$ for all $x,y \in G$.
			\item Assume that $Q:G\times G \to \B(V_\sigma)$ is $G$-invariant and strongly continuous. Then the unitary $G$-representation $\H_Q$ is strongly continuous. 
			\item Suppose that $(\rho, \H_\rho)$ is a unitary $G$-representation and that there is a $G$-equivariant injective linear map $\Phi: \H_\rho \hookrightarrow \Map(G, V_\sigma)^H$ having continuous point evaluations $\mc{E}_x := \ev_x \circ \Phi$ for $x \in G$. Then the corresponding kernel $Q$ is $G$-invariant, and $\H_\rho \cong \H_Q$ as unitary $G$-representations.
		\end{enumerate}		
	\end{proposition}
	\begin{proof}
		Let $l_g$ denote the left $G$-action on itself by left-multiplication. Recall that $\H_Q = \overline{V_\sigma^{(G)}/\mc{N}_Q}^{\langle \--, \--\rangle_Q}$, where $\mc{N}_Q := \set{f \in V_\sigma^{(G)} \st \|f\|_Q = 0}$. For any $x \in G$ we have a map $\delta_x : V_\sigma \hookrightarrow V_\sigma^{(G)}$ defined by considering elements of $V_\sigma$ as functions on $G$ with support $\{x\}$. Let $q_x : V_\sigma \to \H_Q, v \mapsto [\delta_x(v)]$ be its composition with the quotient map $V_\sigma^{(G)} \to \H_Q$. We then have $\mc{E}_x = q_x^\ast$ (cf.\ \citep[Thm.\ I.1.4]{Neeb_book_hol_conv} for more details). The embedding $\H_Q \hookrightarrow V_\sigma^G$ is defined by $f\mapsto f_\psi$, where $f_\psi(x) = \mc{E}_x(\psi)$.
		\begin{enumerate}
			\item For $g \in G$ and $f \in V_\sigma^{(G)}$, we write $g.f:= f \circ l_g^{-1}$ for the left-regular action of $G$ on $V_\sigma^{(G)}$. Let $x,y \in G$. Take $v,w \in V_\sigma$. Then $g.\delta_x(v) = \delta_{gx}(v)$ and $g.\delta_y(w) = \delta_{gy}(w)$ have support on $\{gx\}$ and $\{gy\}$, respectively. Thus $\langle g.\delta_x(v), g.\delta_y(w)\rangle_Q = \langle v, Q(gx, gy)w\rangle$ whereas $\langle q_x(v), q_y(w)\rangle_Q = \langle v, Q(x,y)w\rangle$. The first assertion follows. If $Q$ is $G$-invariant, then $F(x) := Q(e,x)$ satisfies $F(x^{-1}y) = Q(x,y)$.
			\item Let $x \in G$ and $h \in H$. From $Q(xh,y) = \sigma(h)^{-1}Q(x,y)$ it follows that $\mc{E}_{xh}\mc{E}_y^\ast v = \sigma(h)^{-1}\mc{E}_x \mc{E}_y^\ast v$ for any $y \in G$ and $v \in V_\sigma$. As $\{\mc{E}_y^\ast v \st y \in G, v \in V_\sigma\}$ is total in $\H_Q$ by \Fref{thm: properties_repr_kernel}, it follows that $\mc{E}_{xh} = \sigma(h)^{-1}\mc{E}_x$. Thus $f_\psi \in \Map(G, V_\sigma)^H$ for any $\psi \in \H_Q$. We show that $\psi \mapsto f_\psi$ is $G$-equivariant. We have $\pi(g)\mc{E}_x^\ast v = \pi(g)q_x(v) = q_{gx}(v) = \mc{E}_{gx}^\ast(v)$ for every $x,g \in G$ and $v \in V_\sigma$. Hence $\pi(g)\mc{E}_x^\ast = \mc{E}_{gx}^\ast$ and $\mc{E}_x\pi(g) = \mc{E}_{g^{-1}x}$ for every $x,g \in G$. Thus for $\psi \in \H_Q$ we obtain $f_\psi(g^{-1}x) = \mc{E}_{g^{-1}x}\psi = \mc{E}_x\pi(g)\psi = f_{\pi(g)\psi}(x)$, so $\psi \mapsto f_\psi$ is $G$-equivariant.
			
			\item As $G$ acts unitarily on $\H_Q$, it suffices to show that $G \to \C\; g \mapsto \langle \psi, \pi(g)\psi\rangle_Q$ is continuous for any $\psi$ in some total subspace. Consider $\psi = \mc{E}_x^\ast v$ for arbitrary $x \in G$ and $v \in V_\sigma$. Such vectors form a total set in $\H_Q$ by \Fref{thm: properties_repr_kernel}. For $g \in G$, we have
			\begin{equation}\label{eq: apx_comp_cty}
				\langle \psi, \pi(g)\psi\rangle_Q = \langle v, \mc{E}_x \pi(g)\mc{E}_x^\ast v\rangle_V = \langle v, \mc{E}_{x}\mc{E}_{gx}^\ast v\rangle_V = \langle v, Q(x,gx)v\rangle_V.
			\end{equation}
			As $Q : G \times G \to \B(V_\sigma)$ is continuous w.r.t.\ the strong topology, the map $g \mapsto \langle \psi, \pi(g)\psi\rangle_Q$ is continuous.
			
			\item As $\Phi$ is $G$-equivariant, we have $\mc{E}_x \rho(g) = \mc{E}_{g^{-1}x}$ for every $x,g \in G$. As $\rho$ is unitary this implies that the corresponding kernel $Q(x,y) := \mc{E}_x\mc{E}_y^\ast$ is $G$-invariant. This kernel is also positive definite by \Fref{thm: properties_repr_kernel}, so $\H_Q$ is a unitary $G$-representation by the first item. We already know from \Fref{thm: properties_repr_kernel} that $\H_Q \cong \H_\rho$ as Hilbert spaces. The unitary isomorphism $U : \H_Q \to \H_\rho$ is on the dense subspace $V_\sigma^{(G)}/\mc{N}_Q$ given by $U q(f) := \sum_{x \in \supp(f)} \mc{E}_x^\ast f(x)$, where $q : V_\sigma^{(G)} \to \H_Q$ denotes the quotient map. Write $\pi$ for the unitary $G$-action on $\H_Q$. Using $q_x = \mc{E}_x^\ast$, $q_{gx} = \pi(g)q_x$ and $\rho(g)\mc{E}_x^\ast  = \mc{E}_{gx}^\ast$, we obtain that
			$$ U \pi(g)q_x(v) = Uq_{gx}(v) = \mc{E}_{gx}^\ast v = \rho(g)\mc{E}_x^\ast v = \rho(g)Uq_x(v), \qquad \forall v \in V_\sigma.\qedhere$$
		\end{enumerate}
	\end{proof}

%


\begin{thebibliography}{BGN20}

\bibitem[AM66]{Auslander_Moore_unireps_solvable}
L.~Auslander and C.C Moore.
\newblock {\em Unitary Representations of Solvable Lie Groups}.
\newblock Mem. Amer. Math. Soc., 1966.

\bibitem[Arv74]{Arveson_groups_of_aut_of_oa}
W.~Arveson.
\newblock On groups of automorphisms of operator algebras.
\newblock {\em J. Func. Anal.}, 15:217--243, 1974.

\bibitem[Bel05]{Beltita_inv_cpls_str}
D.~Belti\c{t}\u{a}.
\newblock Integrability of analytic almost complex structures on {B}anach
  manifolds.
\newblock {\em Ann. Global Anal. Geom.}, 28(1):59--73, 2005.

\bibitem[BGN20]{beltita_neeb_covariant_reps}
D.~Beltita, H.~Grundling, and K.-H. Neeb.
\newblock Covariant representations for possibly singular actions on
  {$C^*$}-algebras.
\newblock {\em Diss. Math.}, 549:1--94, 2020.

\bibitem[Bor66]{Borchers_energy_momentum}
H.-J. Borchers.
\newblock Energy and momentum as observables in quantum field theory.
\newblock {\em Comm. Math. Phys.}, 2:49--54, 1966.

\bibitem[Bor87]{Borchers_spec_locality}
H.-J. Borchers.
\newblock On the interplay between spectrum condition and locality in quantum
  field theory.
\newblock In {\em Operator Algebras and Mathematical Physics}, volume~62, pages
  143--152. Amer. Math. Soc., 1987.

\bibitem[BR87]{bratelli_robinson_1}
O.~Bratteli and D.W. Robinson.
\newblock {\em Operator Algebras and Quantum Statistical Mechanics 1}.
\newblock Springer-Verlag, New York, second edition, 1987.
\newblock $C^\ast$- and $W^\ast$-algebras, symmetry groups, decomposition of
  states.

\bibitem[BR97]{bratelli_robinson_2}
O.~Bratteli and D.W. Robinson.
\newblock {\em Operator Algebras and Quantum Statistical Mechanics 2}.
\newblock Springer-Verlag, Berlin, second edition, 1997.
\newblock Equilibrium states. Models in quantum statistical mechanics.

\bibitem[BS71a]{Bochnak_Siciak_2}
J.~Bochnak and J.~Siciak.
\newblock Analytic functions in topological vector spaces.
\newblock {\em Studia Math.}, 39:77--112, 1971.

\bibitem[BS71b]{Bochnak_Siciak_1}
J.~Bochnak and J.~Siciak.
\newblock Polynomials and multilinear mappings in topological vector spaces.
\newblock {\em Studia Math.}, 39:59--76, 1971.

\bibitem[DK00]{Duistermaat_book}
J.J. Duistermaat and J.A.C. Kolk.
\newblock {\em Lie Groups}.
\newblock Springer-Verlag, Berlin, 2000.

\bibitem[Gl{\"o}02a]{Glockner_unit_gps_lie}
H.~Gl{\"o}ckner.
\newblock Algebras whose groups of units are {L}ie groups.
\newblock {\em Studia Math.}, 153(2):147--177, 2002.

\bibitem[Gl{\"o}02b]{Glockner_inf_dim_lie_without_completeness}
H.~Gl{\"o}ckner.
\newblock Infinite-dimensional {L}ie groups without completeness restrictions.
\newblock In {\em Geometry and Analysis on Finite- and Infinite-dimensional
  {L}ie Groups}, volume~55, pages 43--59. Polish Acad. Sci. Inst. Math., 2002.

\bibitem[GN]{Neeb_Glockner_Book}
H.~Gl{\"o}ckner and K.-H. Neeb.
\newblock {I}nfinite {D}imensional {L}ie {G}roups.
\newblock Book in preparation.

\bibitem[GN12]{GlocknerNeeb_unit_gp_regular}
H.~Gl{\"o}ckner and K.-H. Neeb.
\newblock When unit groups of continuous inverse algebras are regular {Lie}
  groups.
\newblock {\em Stud. Math.}, 211(2):95--109, 2012.

\bibitem[Goo69]{Goodman_analytic_entire_vectors}
R.W. Goodman.
\newblock Analytic and entire vectors for representations of {L}ie groups.
\newblock {\em Trans. Amer. Math. Soc.}, 143:55--76, 1969.

\bibitem[Haa92]{Haag_loc_qt_ph}
R.~Haag.
\newblock {\em Local Quantum Physics}.
\newblock Springer-Verlag, Berlin, 1992.

\bibitem[Jen73]{Jenkins_growth_rate}
J.W. Jenkins.
\newblock Growth of connected locally compact groups.
\newblock {\em J. Func. Anal.}, 12:113--127, 1973.

\bibitem[JN19]{BasNeeb_ProjReps}
B.~Janssens and K.-H. Neeb.
\newblock Projective unitary representations of infinite-dimensional {L}ie
  groups.
\newblock {\em Kyoto J. Math.}, 59(2):293--341, 2019.

\bibitem[JN21]{BasNeeb_PE_reps_I}
B.~Janssens and K.-H. Neeb.
\newblock Positive energy representations of gauge groups {I}: Localization,
  2021.
\newblock In preparation. Available at arXiv:2108.03501.

\bibitem[Kir76]{Kirillov_book_representations}
A.A. Kirillov.
\newblock {\em Elements of the Theory of Representations}.
\newblock Springer-Verlag, Berlin-New York, 1976.

\bibitem[KM97]{michor_convenient}
A.~Kriegl and P.W. Michor.
\newblock {\em The Convenient Setting of Global Analysis}, volume~53.
\newblock AMS, Providence, 1997.

\bibitem[Lis95]{Lisiecki_coh_states_survey}
W.~Lisiecki.
\newblock Coherent state representations. {A} survey.
\newblock {\em Rep. Math. Phys.}, 35(2-3):327--358, 1995.

\bibitem[LM75]{Luscher_Mack_glob_conf_inv}
M.~L\"{u}scher and G.~Mack.
\newblock Global conformal invariance in quantum field theory.
\newblock {\em Comm. Math. Phys.}, 41:203--234, 1975.

\bibitem[Mil84]{milnor_inf_lie}
J.~Milnor.
\newblock Remarks on infinite-dimensional {L}ie groups.
\newblock In {\em Relativity, Groups and Topology}, pages 1007--1057.
  North-Holland, Amsterdam, 1984.

\bibitem[Mun00]{Munkres_topology}
J.R. Munkres.
\newblock {\em Topology}.
\newblock Prentice Hall, Inc., 2000.

\bibitem[Mur90]{Murphy_operator_algs}
G.J. Murphy.
\newblock {\em {$C^*$}-algebras and Operator Theory}.
\newblock Academic Press, Inc., Boston, MA, 1990.

\bibitem[Nee98]{Neeb_open_problems}
K.-H. Neeb.
\newblock Some open problems in representation theory related to complex
  geometry.
\newblock In {\em Positivity in {L}ie Theory: Open Problems}, volume~26, pages
  195--220. de Gruyter, Berlin, 1998.

\bibitem[Nee00]{Neeb_book_hol_conv}
K.-H. Neeb.
\newblock {\em Holomorphy and Convexity in {L}ie Theory}, volume~28.
\newblock Walter de Gruyter \& Co., Berlin, 2000.

\bibitem[Nee01]{Neeb_borel_weil_loop_gps}
K.-H. Neeb.
\newblock Borel-{W}eil theory for loop groups.
\newblock In {\em Infinite Dimensional {K}\"{a}hler Manifolds}, volume~31,
  pages 179--229. Birkh\"{a}user Verlag, 2001.

\bibitem[Nee04]{Neeb_inf_dim_gps_reps_2004}
K.-H. Neeb.
\newblock Infinite-dimensional groups and their representations.
\newblock In {\em Lie Theory}, volume 228, pages 213--328. Birkh\"{a}user
  Boston, 2004.

\bibitem[Nee06]{neeb_towards_lie}
K.-H. Neeb.
\newblock Towards a {L}ie theory of locally convex groups.
\newblock {\em Jpn. J. Math.}, 1(2):291--468, 2006.

\bibitem[Nee10a]{Neeb_diffvectors}
K.-H. Neeb.
\newblock On differentiable vectors for representations of infinite dimensional
  {L}ie groups.
\newblock {\em J. Funct. Anal.}, 259(11):2814--2855, 2010.

\bibitem[Nee10b]{neeb_semibounded_inv_cones}
K.-H. Neeb.
\newblock Semibounded representations and invariant cones in infinite
  dimensional {L}ie algebras.
\newblock {\em Confl. Math.}, 2(1):37--134, 2010.

\bibitem[Nee11]{Neeb_analytic_vectors}
K.-H. Neeb.
\newblock On analytic vectors for unitary representations of infinite
  dimensional {L}ie groups.
\newblock {\em Ann. Inst. Fourier}, 61(5):1839--1874, 2011.

\bibitem[Nee12]{Neeb_semibounded_hermitian}
K.-H. Neeb.
\newblock Semibounded representations of {H}ermitian {L}ie groups.
\newblock {\em Travaux math\'{e}matiques.}, 21(5):29--109, 2012.

\bibitem[Nee13]{Neeb_hol_reps}
K.-H. Neeb.
\newblock Holomorphic realization of unitary representations of {B}anach-{L}ie
  groups.
\newblock In {\em Lie Groups: Structure, Actions, and Representations}, volume
  306, pages 185--223. Birkh\"{a}user/Springer, New York, 2013.

\bibitem[Nee14]{Neeb_semibounded_hilbert_loop}
K.-H. Neeb.
\newblock Semibounded unitary representations of double extensions of
  {H}ilbert-loop groups.
\newblock {\em Ann. Inst. Fourier}, 64(5):1823--1892, 2014.

\bibitem[Nie23]{Milan_reps_jets}
M.~Niestijl.
\newblock Generalized positive energy representations of groups of jets.
\newblock {\em Doc. Math}, 28(3):709--763, 2023.

\bibitem[NR22]{neeb_russo_ground_state_top}
K.-H. Neeb and F.G. Russo.
\newblock Ground state representations of topological groups.
\newblock {\em Math.\ Ann.}, pages 1--60, 2022.

\bibitem[NS15]{Neeb_Salmasian_Virasoro_pe}
K.-H. Neeb and H.~Salmasian.
\newblock Classification of positive energy representations of the {V}irasoro
  group.
\newblock {\em Int. Math. Res. Not.}, Volume 2015(18):8620--8656, 2015.

\bibitem[NSZ15]{KH_Zellner_inv_smooth_vectors}
K.-H. Neeb, H.~Salmasian, and C.~Zellner.
\newblock On an invariance property of the space of smooth vectors.
\newblock {\em Kyoto J. Math.}, 55(3):501--515, 2015.

\bibitem[Ol'81]{Olshanski_inv_cones_hol_discr_series}
G.I. Ol'shanski\u{\i}.
\newblock Invariant cones in {L}ie algebras, {L}ie semigroups and the
  holomorphic discrete series.
\newblock {\em Funktsional. Anal. i Prilozhen.}, 15(4):53--66, 1981.

\bibitem[Pen74]{Penney_holomorphic_extensions}
R.~Penney.
\newblock Entire vectors and holomorphic extension of representations.
\newblock {\em Trans. Amer. Math. Soc.}, 198:107--121, 1974.

\bibitem[Per86]{Perelomov_coh_states}
A.~Perelomov.
\newblock {\em Generalized Coherent States and Their Applications}.
\newblock Springer-Verlag, Berlin, 1986.

\bibitem[PS86]{Segal_Loop_Groups}
A.~Pressley and G.~Segal.
\newblock {\em Loop Groups}.
\newblock Oxford University Press, 1986.

\bibitem[Rud91]{Rudin_FA}
W.~Rudin.
\newblock {\em Functional Analysis}.
\newblock McGraw-Hill, Inc., New York, second edition, 1991.

\bibitem[Seg81]{Segal_unreps_of_some_inf_dim_gps}
G.~Segal.
\newblock Unitary representations of some infinite-dimensional groups.
\newblock {\em Comm. Math. Phys.}, 80(3):301--342, 1981.

\bibitem[SW64]{Wightman_pct_spin_statistics}
R.F. Streater and A.S. Wightman.
\newblock {\em P{CT}, Spin and Statistics, and all That}.
\newblock Benjamin, New York, 1964.

\bibitem[Tre67]{Treves_tvs}
F.~Treves.
\newblock {\em Topological Vector Spaces, Distributions and Kernels}.
\newblock Academic Press, New York-London, 1967.

\bibitem[TW71]{Tirao_Wolf_hom_hol_vb}
J.A. Tirao and J.A. Wolf.
\newblock Homogeneous holomorphic vector bundles.
\newblock {\em Indiana Univ. Math. J.}, 20:15--31, 1970/71.

\end{thebibliography}
\end{document}